\documentclass[11pt]{article}

\usepackage{amsfonts}
\usepackage{amssymb}
\usepackage{amsmath}
\usepackage{graphicx}
\usepackage{setspace}
\usepackage{amsthm}
\usepackage[labelfont=bf]{caption}
\usepackage[left=2.2cm,top=2.4cm,right=2.2cm,bottom=2.4cm]{geometry}

\theoremstyle{plain}
\newtheorem{thm}{Theorem}[section] 
\newtheorem{defn}[thm]{Definition} 
\newtheorem{lem}[thm]{Lemma} 
\newtheorem{pro}[thm]{Proposition} 
\newtheorem{rem}[thm]{Remark}

\begin{document}
\begin{center}
\section*{Some results on the Signature and Cubature of the Fractional Brownian motion for $H>\frac{1}{2}$}
\subsection*{Riccardo Passeggeri\footnote[1]{Imperial College London, UK. Email: riccardo.passeggeri14@imperial.ac.uk}}
\end{center}
\begin{abstract}
In this work we present different results concerning the signature and the cubature of fractional Brownian motion (fBm). The first result regards the rate of convergence of the expected signature of the linear piecewise approximation of the fBm to its exact value, for a value of the Hurst parameter $H\in(\frac{1}{2},1)$. We show that the rate of convergence is given by $2H$. We believe that this rate is sharp as it is consistent with the result of Ni and Xu \cite{NiXu}, who showed that the sharp rate of convergence for the Brownian motion (\textit{i.e.} fBm with $H=\frac{1}{2}$) is given by $1$. The second result regards the bound of the \textit{coefficient} of the rate of convergence obtained in the first result. We obtain an uniform bound for the coefficient for the $2k$-th term of the signature of $\frac{\tilde{A}k(2k-1)}{(k-1)!2^{k}}$, where $\tilde{A}$ is a finite constant independent of $k$. The third result regards the sharp decay rate of the expected signature of the fBm. We obtain a sharp bound for the $2k$-th term of the expected signature of $\frac{1}{k!2^{k}}$. The last results concern the cubature method for the fBm for $H>\frac{1}{2}$. In particular, we develop the framework of the cubature method for fBm, provide a bound for the approximation error in the general case, and obtain the cubature formula for the fBm in a particular setting. These results extend the work of Lyons and Victoir \cite{LV}, who focused on the Brownian motion case.
\\
\\
\textbf{Key words:} fractional Brownian motion, signature, rate of convergence, sharp decay rate, cubature method.
\end{abstract}

\section{Introduction}
The signature of a $d$-dimensional fractional Brownian motion (fBm) is a sequence of iterated Stratonovich integrals along the paths of the fBm; it is an object taking values in the tensor algebra over $\mathbb{R}^{d}$. 
\\
Signatures were firstly studied by K.T.-Chen in 1950’s in a series of papers \cite{Chen1}, \cite{Chen} and \cite{Chen3}. In the last twenty years the attention devoted to signatures has increased rapidly. This has been caused by the pivotal role they have in rough path theory, a field developed in the late nineties by Terry Lyons culminating in the paper \cite{L}, which is also at the base of the newly developed theory of regularity structures \cite{Hairer}. The signature of a path summarises the essential properties of that path allowing the possibility to study SPDEs driven by that path. 
\\
We remark also the increasing importance of the signatures in the field of machine learning, see for example \cite{machine} among many. In light of this increasing importance and of the fact that rate of convergences are crucial to evaluate algorithmic efficiency, our results may have a \textit{direct} impact to real world applications.

In 2015 Hao Ni and Weijun Xu \cite{NiXu} computed the sharp rate of convergence for expected Brownian signatures. They obtained a rate of convergence of $1$; in formulas,
\begin{equation*}
\Bigg|\mathbb{E}\left(\int_{\Delta^{2k}[0,1]}dB^{I}\right)-\mathbb{E}\left(\int_{\Delta^{2k}[0,1]}dB^{m,I}\right)\Bigg|\leq Cm^{-1},
\end{equation*}
where $C>0$, $\mathbb{E}\left(\int_{\Delta^{2k}[0,1]}dB^{I}\right)$ is the (truncated) expected signature of the Brownian motion and $\mathbb{E}\left(\int_{\Delta^{2k}[0,1]}dB^{m,I}\right)$ is the (truncated) expected signature of the linear piecewise approximation of the Brownian motion with mesh size $\frac{1}{m}$. A more formal introduction is given in the next section.
\\
On the other hand, for the fBm no progress has been made in this direction. In particular, the rate of convergence for the expected signature of the fBm is not known for any value of the Hurst parameter $H\in(0,1)$. The first result of this article address this problem, obtaining the  ``sharp" rate of convergence for $H\in(\frac{1}{2},1)$. In formulas,
\begin{equation}\label{prima}
\Bigg|\mathbb{E}\left(\int_{\Delta^{2k}[0,1]}dB^{H,I}\right)-\mathbb{E}\left(\int_{\Delta^{2k}[0,1]}dB^{H,m,I}\right)\Bigg|\leq C'm^{-2H},
\end{equation}
where $C'>0$, $\mathbb{E}\left(\int_{\Delta^{2k}[0,1]}dB^{H,I}\right)$  and $\mathbb{E}\left(\int_{\Delta^{2k}[0,1]}dB^{H,m,I}\right)$ are respectively the (truncated) expected signature of the fBm and of its linear piecewise approximation with mesh size $\frac{1}{m}$. To achieve this result, we used the results of Baudoin and Coutin in \cite{BauCou}, who computed the expected signature for the fractional Brownian motion for $H>\frac{1}{2}$ and also for small times for $H\in(\frac{1}{3},\frac{1}{2})$. We mention also the works \cite{1/4} and \cite{LVextension}, where further properties of the signature of the fBm are studied.
\\
In this work we focus on the weak rate of convergence and we refer to the work of Friz and Riedel \cite{FR} for the strong rate of convergence. It is possible to derive from their work a rate of $H$ for $H>\frac{1}{2}$ (see also Deya, Neuenkirch and Tindel \cite{Deya} for similar results), while here we obtain a weak convergence rate of $2H$.

The second result of this work regards the bound of the coefficient of the first result, namely $C'$ in $(\ref{prima})$. We obtain an uniform bound for the coefficient for the $2k$-th term of the signature of
\begin{equation*}
\frac{\tilde{A}k(2k-1)}{(k-1)!2^{k}},
\end{equation*}
where $\tilde{A}$ is a finite constant independent of $k$. This shows that for a fixed linear piecewise approximation of the fBm as the number of iterated integrals increases the difference between the expected signature of the fBm and of its linear piecewise approximation goes to zero (and it goes fast).

For the third result we move away from the linear piecewise approximation and focus just on the expected signature. In \cite{CheLyo}, Chevyrev and Lyons showed that the expected signature has infinite radius of convergence, but not a factorial decay. In this work we show that the expected signature has a factorial decay, indeed the sharp bound for the $2k$-th term of the signature is simply given by $\frac{1}{k!2^{k}}$ for all $H\in(\frac{1}{2},1)$. In the $H > \frac{1}{2}$ case, our result gives an alternative proof, with sharper estimates, of the fact that the expected signature of the fractional Brownian motion has infinite radius of convergence, which by \cite{CheLyo} implies that the expected signature determines the signature in distribution. Our estimate is also sharper than the one obtained by Friz and Riedel \cite{FR2}, from which is possible to obtain a bound of $\frac{1}{(k/2)!}$, and Neuenkirch, Nourdin, Rössler and Tindel \cite{NNRT}, who showed a bound of $\frac{C^{k}}{\sqrt{(2k)!}}$ (with $C>2$ and dependent on $H$). In formulas, our result is:
\begin{equation*}
\mathbb{E}\left(\int_{\Delta^{2k}[s,t]}dB^{I}\right)\leq \dfrac{(t-s)^{2kH}}{k!2^{k}}.
\end{equation*}

In 2003 Lyons and Victoir \cite{LV} developed a numerical methods for approximating the solution of parabolic PDEs and general SDEs driven by Brownian motions called cubature method on Weiner space. In the cubature method the main first step is to obtain the cubature formula in which the truncated signature of a path (in the case of \cite{LV} is the Brownian motion) is matched to a finite weighted sum of Dirac delta measures applied to the iterated integrals of the deterministic paths. In this work, we develop the framework of the cubature method for fBm, obtain the cubature formula for the fBm in a particular setting, and provide a bound for the approximation error for the general case.

This paper is structured in the following way. In section 2 we introduce some notations and state formally the main results. In section 3 we will discuss about preliminaries involving definitions and the results of other authors. In section 4, 5 and 6 we prove the first three main results of the article. In section 7 we discuss the cubature method for the fBm.
\section{Main Results}
In this section we introduce the main results of this paper. But first, we introduce some notation, which is in line with the one used in the papers of Baudoin and Coutin $\cite{BauCou}$, Lyons $\cite{L}$, and Lyons and Victoir $\cite{LV}$ and in the book by Lyons, Caruana, and L\'{e}vy $\cite{LCL}$.

The fractional Brownian motion is defined as follows.
\begin{defn}
Let H be a constant belonging to $(0,1)$. A fractional Brownian motion (fBm) $(B^{H} (t))_{t≥0}$ of Hurst index H is a continuous and centered Gaussian process with covariance function
\begin{equation*}
\mathbb{E}\left[B^{H} (t)B^{H} (s)\right]=\frac{1}{2}(t^{2H}+s^{2H}-|t-s|^{2H}).
\end{equation*}
\end{defn}
\noindent From now on we will denote $B:=B^{H}$. \\For $H=\frac{1}{2}$ then the fBm is a Bm. Further, the multi-dimensional fBm has coordinate components that are independent and identically distributed copies of one dimensional fBm. Moreover, recall that the fBm has the following two properties:
\\
$(i)\enspace \text{(scaling) for any $c>0$,}\enspace c^{H}B_{\cdot/c}\enspace\text{is a fBm,}$
\\
$(ii)\enspace \text{(stationary increments) for any $h>0$,}\enspace B_{\cdot+h}-B_{h}\enspace\text{is a fBm.}$\\\\
Now, we define the simplex $\Delta^{k}[s,t]$, where $T\in[0,\infty)$ and $s,t\in[0,T]$,
\begin{equation*}
\Delta^{k}[s,t]:=\{(t_{1},...,t_{k})\in[s,t]^{k}:t_{1}<...<t_{k}\}.
\end{equation*}
Further, we define the following iterated integrals. Let $I=(i_{1}, . . . , i_{k})\in\{1,..,d\}^{k}$ be a word with length $k$ then
\begin{equation*}
\int_{\Delta^{k}[s,t]}dB^{I}:=\int_{s\leq t_{1}<...<t_{k}\leq t}dB^{i_{1}}_{t_{1}}\cdot\cdot\cdot dB^{i_{k}}_{t_{k}}
\end{equation*}
and
\begin{equation*}
\int_{\Delta^{k}[s,t]}dB^{m,I}:=\int_{s\leq t_{1}<...<t_{k}\leq t}dB^{m,i_{1}}_{t_{1}}\cdot\cdot\cdot dB^{m,i_{k}}_{t_{k}},
\end{equation*}
where $B$ is the fractional Brownian motion with Hurst parameter $H$ and $B^{m}$ is its linear piecewise approximation. In addition, $B^{i}$ is the $i$-th coordinate component of the fBm $B$ and the iterated integrals can be defined in the sense of Young \cite{Young}. 
\\Moreover, the linear piecewise approximation $B^{m}$ is defined as follows. The only requirement we need is that the linear approximation comes from a uniform grid. Let $s,u\in[0,\infty)$ and consider an interval of time $[s,u]$. Let $t_{i}^{[s,u]}:=\frac{i}{m}(u-s)+s$ for $i=0,...,m$. If $t\in[t_{i}^{[s,u]},t_{i+1}^{[s,u]}]$ then
\begin{equation*}
B^{m}_{[u,s],t}:=B^{m}_{t_{i}^{[s,u]}}+\frac{m}{u-s}(t-t_{i}^{[s,u]})(B^{m}_{t_{i+1}^{[s,u]}}-B^{m}_{t_{i}^{[s,u]}}).
\end{equation*}
Observe that the definition of the linear piecewise approximation depends on the interval considered. Notice also that $B_{t}^{m}$ satisfies the scaling and stationary increments property, once the interval is properly modified. Indeed, we have that for the stationary increments property the interval to be taken is, for any $h>0$, $[s+h,u+h]$ and we have
\begin{equation*}
B^{m}_{[s+h,u+h],t+h}=B^{m}_{t_{i}^{[s+h,u+h]}}+\frac{m}{u-s}(t+h-t_{i}^{[s+h,u+h]})(B^{m}_{t_{i+1}^{[s+h,u+h]}}-B^{m}_{t_{i}^{[s+h,u+h]}})
\end{equation*}
\begin{equation*}
\stackrel{Law}{=}B^{m}_{t_{i}^{[u,s]}}+\frac{m}{u-s}(t-t_{i}^{[u,s]})(B^{m}_{t_{i+1}^{[u,s]}}-B^{m}_{t_{i}^{[u,s]}})=B^{m}_{[u,s],t}
\end{equation*}
and for the scaling property the interval to be taken, for any $c>0$, is $[s/c,u/c]$ and we have
\begin{equation*}
c^{H}B^{m}_{[u/c,s/c],t/c}=c^{H}B^{m}_{t_{i}^{[s/c,u/c]}}+c^{H}\frac{cm}{u-s}(t/c-t_{i}^{[s/c,u/c]})(B^{m}_{t_{i+1}^{[s/c,u/c]}}-B^{m}_{t_{i}^{[s/c,u/c]}})
\end{equation*}
\begin{equation*}
=c^{H}B^{m}_{t_{i}^{[s,u]}/c}+c^{H}\frac{cm}{u-s}(t/c-t_{i}^{[s,u]}/c)(B^{m}_{t_{i+1}^{[s,u]}/c}-B^{m}_{t_{i}^{[s,u]}/c})
\end{equation*}
\begin{equation*}
\stackrel{Law}{=}B^{m}_{t_{i}^{[u,s]}}+\frac{m}{u-s}(t-t_{i}^{[u,s]})(B^{m}_{t_{i+1}^{[u,s]}}-B^{m}_{t_{i}^{[u,s]}})=B^{m}_{[u,s],t}.
\end{equation*}
Further, from now on we will not write explicitly the interval of the linear approximation, unless it is necessary in order to avoid confusion. Hence, we denote $t_{i}:=t_{i}^{[s,u]}$ and $B^{m}_{t}:=B^{m}_{[u,s],t}$. 

We can now present our main results. The first result is about the rate of convergence of the expected signature of the linear piecewise approximation of the fBm to its exact value.
\begin{thm}\label{pr1}
Let $H > \frac{1}{2}$, $T\in[0,\infty)$ and $0\leq s<t\leq T$. Let $I = (i_{1}, . . . , i_{2k})$ be a word where $i_{l}\in\{1,...,d\}$ for $l=1,...,2k$, then for all m
\begin{equation*}
\Bigg|\mathbb{E}\left(\int_{\Delta^{2k}[s,t]}dB^{I}\right)-\mathbb{E}\left(\int_{\Delta^{2k}[s,t]}dB^{m,I}\right)\Bigg|\leq Cm^{-2H}(t-s)^{2kH},
\end{equation*}
where C is a finite constant and depending only on $k$ and $H$.
\end{thm}
\noindent It is important to stress that for the Brownian motion case, namely $H=\frac{1}{2}$, Ni and Xu in \cite{NiXu} proved that the sharp rate of convergence is given by $1$ (\textit{i.e.} a uniform bound of $Cm^{-1}$), which is in line with the result presented here. Moreover, the proof used in \cite{NiXu} cannot in principle be used to prove our result since Ni and Xu used the independence of the increments property (\textit{i.e.} the Markov semigroup property) of the Brownian motion, which does not hold for the fBm. Conversely, the proof used to prove Theorem 2.2 is based on the integral form of the covariance function of the fBm which is valid only for $H > \frac{1}{2}$, hence our proof cannot in principle be used to prove the result in \cite{NiXu}.

In the next theorem we focus on and refine the value of the coefficient $C$ of the previous theorem and we provide a bound for it.
\begin{thm}\label{pr2}
Let $H > \frac{1}{2}$, $T\in[0,\infty)$ and $0\leq s<t\leq T$. Let $I = (i_{1}, . . . , i_{2k})$ be a word where $i_{l}\in\{1,...,d\}$ for $l=1,...,2k$, then
\begin{equation*}
\limsup\limits_{m\rightarrow\infty}m^{2H}\Bigg|\mathbb{E}\left(\int_{\Delta^{2k}[s,t]}dB^{I}\right)-\mathbb{E}\left(\int_{\Delta^{2k}[s,t]}dB^{m,I}\right)\Bigg|\leq \dfrac{\tilde{A}k(2k-1)}{(k-1)!2^{k}}(t-s)^{2kH},
\end{equation*}
where $\tilde{A}$ is a finite constant depending only on $H$.
\end{thm}
The following theorem provides a sharp bound for the value for the expected signature of the fBm. In the $H > \frac{1}{2}$ case, this result in particular implies Chevyrev-Lyons'result that the expected signature of fBm has infinite radius of convergence. This in turns implies that the expected signature determines the signature in distribution.
\begin{thm}\label{pr3}
Let $H > \frac{1}{2}$, $T\in[0,\infty)$ and $0\leq s<t\leq T$. Let $I = (i_{1}, . . . , i_{2k})$ be a word where $i_{l}\in\{1,...,d\}$ for $l=1,...,2k$, then
\begin{equation*}
\mathbb{E}\left(\int_{\Delta^{2k}[s,t]}dB^{I}\right)\leq \dfrac{(t-s)^{2kH}}{k!2^{k}}.
\end{equation*}
\end{thm}
We move now to the cubature method. The results we present require an extended formal introduction, which is given in Section 7. The first main result regards the bound of the approximation error of the cubature method, namely the bound of the difference between the expected exact solution of an SDE driven by a fBm, which we denote by $\mathbb{E}\left(f(\xi_{T,x})\right)$, and the deterministic solution obtained by the cubature method, denoted by $\sum_{j=1}^{n}\lambda_{j}f(\Phi_{T,x}(\omega_{j}))$. 
In particular, with this result we extend Proposition 2.1, Lemma 3.1 and Proposition 3.2 of \cite{LV} to the fBm case for $H>1/2$. Before stating the theorem we introduce the following notation. Let $|I|$ indicate the length of the word $I$, where $I\in\{0,...,d\}^{k}$ for some $k\in\mathbb{N}$ and let $C_{b}^{\infty}(\mathbb{R}^{N},\mathbb{R}^{N})$, where $N\in\mathbb{N}$, be the space of $\mathbb{R}^{N}$-valued smooth functions defined in $\mathbb{R}^{N}$ whose derivatives of any order are bounded. We regard elements of $C_{b}^{\infty}(\mathbb{R}^{N},\mathbb{R}^{N})$ as vector fields on $\mathbb{R}^{N}$. Let $V_{0},...,V_{d}$ be such vector fields.
\begin{thm}\label{prNEW} Consider any function $f:\mathbb{R}^{N}\rightarrow\mathbb{R}$, whose derivatives of any order exist and are bounded, and vector fields $V_{0},...,V_{d}\in C_{b}^{\infty}(\mathbb{R}^{N},\mathbb{R}^{N})$. Let $T\in[0,\infty)$ and assume that there exists $M > 0$ and $0 \leq\gamma < 1/2$ such that, for every word $I$, $|(V_{i_{i}}\cdots V_{i_{k}}f)(x)| \leq	M^{|I|} (|I|!)^{\gamma}$. Then, we have
	\begin{equation*}
	\sup_{x\in\mathbb{R}^{N}}\Big|\mathbb{E}\left(f(\xi_{T,x})\right) -\sum_{j=1}^{n}\lambda_{j}f(\Phi_{T,x}(\omega_{j}))\Big|\leq\begin{cases}
	L_{1}(T)\quad\text{if $T\geq 1$},\\ L_{2}(T)\quad\text{if $T< 1$},
	\end{cases} 
	\end{equation*}
	where
	\begin{equation*}
	L_{1}(T):=CT^{(m+2)/2}\left(1+\sup\limits_{x\in\mathbb{R}^{N}}M_{x}^{(m+2)/2}\sum_{k=0}^{\infty}\frac{(dM_{x}KT)^{k}}{(k!)^{1/2-\gamma}}\right)
	\end{equation*}
	\begin{equation*}
L_{2}(T):=C' T^{2H}+C''T^{H(m+2)/2}\sup\limits_{x\in\mathbb{R}^{N}}M_{x}^{(m+2)/2}\sum_{k=0}^{\infty}\frac{(dM_{x}KT^{H})^{k}}{(k!)^{1/2-\gamma}},
	\end{equation*}
	and where $C,C',C''>0$ are constants independent of T and $K=\sqrt{\frac{2}{H(2H-1)}}$.
\end{thm}
In the second main results, we provide the cubature formula for the one dimensional fBm up to degree 5. We will discuss the concept of the degree of a cubature formula in Section 7. Now, let $C^{0}_{0,bv}([0, T ], \mathbb{R}^{d} )$ be the space of $\mathbb{R}^{d}$-valued continuous paths of bounded variation defined in $[0, T]$ and which starts at zero.
\begin{thm}\label{pr4}
Let $H \geq \frac{1}{2}$. Define $\hat{B}$ and $\hat{\omega}$ to be
\begin{equation*}
\hat{B}_{t_{l}}^{i_{l}}=
\begin{cases}
t_{l}, \qquad if \quad i_{l}=0,\\
B_{t_{l}} \qquad if \quad i_{l}=1,
\end{cases}\quad
and \qquad
\hat{\omega}_{t_{l},j}^{i_{l}}=
\begin{cases}
t_{l}, \qquad if \quad i_{l}=0,\\
\omega_{t_{l},j} \qquad if \quad i_{l}=1,
\end{cases}
\end{equation*}
for $l=1, . . . , k$ and $j=1,...,n$ and where $\omega_{j}\in C^{0}_{0,bv}([0, T ], \mathbb{R}^{d} )$. The weights $\{\lambda_1, \lambda_2, \lambda_3\}$ and the paths $\{\omega_{t,1},\omega_{t,2},\omega_{t,3}\} $ will satisfy the cubature formula 
\begin{equation*}
\mathbb{E}\left[\int_{0<t_{1}<,...,<t_{k}< T} d\hat{B}_{t_{1}}^{i_{1}}\cdot\cdot\cdot d\hat{B}_{t_{k}}^{i_{k}}\right]=\sum_{j=1}^{n}\lambda_{j}\int_{0<t_{1}<,...,<t_{k}< T}d\hat{\omega}_{t_{1},j}^{i_{1}}\cdot\cdot\cdot d\hat{\omega}_{t_{1},j}^{i_{k}},
\end{equation*}
for the following degree
\begin{equation*}
Degree=\begin{cases}
5 \qquad for \qquad \frac{1}{2}\leq H<\frac{2}{3},\\
4 \qquad for \qquad \frac{2}{3}\leq H<1,\\
\end{cases}
\end{equation*}
if $n=3$, $\lambda_{1}=\lambda_{2}=\frac{1}{6}$ and $\lambda_{3}=\frac{2}{3}$, and 
\begin{equation*}
\omega_{t,1}=\begin{cases}
(2\alpha-\beta) t, \qquad t\in[0,\frac{1}{3}],\\
(\alpha-\beta)+(2\beta-\alpha)t, \qquad t\in[\frac{1}{3},\frac{2}{3}],\\
(\beta-\alpha)+(2\alpha-\beta)t, \qquad t\in[\frac{2}{3},1],\\
\end{cases}
\end{equation*}
where
\begin{equation*}
\alpha:=\dfrac{2H\sqrt{3}+\sqrt{3}}{2H+1} \qquad and \qquad \beta:=\dfrac{\sqrt{-96H^{2}+66H+57}}{2H+1},
\end{equation*}
and $\omega_{t,2}=-\omega_{t,1}$ and $\omega_{t,3}=0$ for $t\in[0,1]$.
\end{thm}
\noindent It is important to remark that for $H = \frac{1}{2}$, \textit{i.e.} when the fractional Brownian motion is just a Brownian motion, the results obtained in this theorem perfectly coincide with the results of Lyons and Victoir obtained in \cite{LV}.
\section{Preliminaries}\label{Preliminaries}
One of the main concepts of this paper is the signature of the fractional Brownian motion. In order to give a definition of signature we need to introduce first the following concept: the \textit{p-variation} of a path.
\begin{defn} \label{p-var}
Let $p\geq 1$ be a real number. Let $X:J\rightarrow E$ be continuous path, where J is a compact interval and E is a finite dimensional Banach space. The p-variation of X on the interval J is defined by
\begin{equation*}
\|X\|_{p,J}=\left[\sup_{\mathcal{D}\subset J}\sum_{i=0}^{r-1}|X_{t_{i}}-X_{t_{i+1}}|^{p}\right]^{\frac{1}{p}},
\end{equation*}
where the supremum is taken over all the partitions of J.
\end{defn}
\noindent A path $X$ is said to be of finite $p$-variation over the interval $J$ if $\|X\|_{p,J}<\infty$.\\
For signature of a stochastic process we mean the following.
\begin{defn}
Let E be a finite dimensional Banach space. Let $X:[0,T]\rightarrow E$ be a continuous path with finite p-variation for some $p<2$. The signature of X is the element $S(X)$ of $T(E):=\{\mathbf{a}=(a_{0},a_{1},...)|\forall n\geq 0, a_{n}\in E^{\otimes n}\}$ defined as follows
\begin{equation*}
 S(X_{[0,T]})=(1,X_{[0,T]}^{1},X_{[0,T]}^{2},...),
\end{equation*}
where, for each $n\geq 1$, 
\begin{equation*}
X_{[0,T]}^{n}=\int_{0<u_{1}<...<u_{n}<T}dX_{u_{1}}\otimes...\otimes dX_{u_{n}}.
\end{equation*}
\end{defn}
\noindent The integration in the previous definition is defined in the sense of Young \cite{Young}.

The fractional Brownian motion is a $p$-variation path for all $p>\frac{1}{H}$. Hence, for $H\in(\frac{1}{2},1)$ we have $p<2$.
Notice that for $I$, a word composed by the letters $(i_{1},...,i_{2k})$ where each $i_{l}\in{1,...,d}$, we have that $dB^{I}:=dB^{i_{1}}\cdot\cdot\cdot dB^{i_{2k}}$. Therefore, we have that the $2k$-th element of the signature of the fBm and its piecewise approximation are respectively given by
\begin{equation}
\mathbb{E}\left(\int_{0<u_{1}<...<u_{2k}<T}dB_{u_{1}}\otimes\cdot\cdot\cdot \otimes dB_{u_{2k}}\right)=\sum_{i_{1}=1}^{d}\cdot\cdot\cdot \sum_{i_{2k}=1}^{d}\mathbb{E}\left(\int_{\Delta^{2k}[0,T]}dB^{I}\right)e_{i_{1}}\otimes\cdot\cdot\cdot \otimes e_{i_{2k}},\quad \text{and by}
\end{equation}
\begin{equation}
\mathbb{E}\left(\int_{0<u_{1}<...<u_{2k}<T}dB^{m}_{u_{1}}\otimes\cdot\cdot\cdot\otimes dB^{m}_{u_{2k}}\right)=\sum_{i_{1}=1}^{d}\cdot\cdot\cdot \sum_{i_{2k}=1}^{d}\mathbb{E}\left(\int_{\Delta^{2k}[0,T]}dB^{m,I}\right)e_{i_{1}}\otimes\cdot\cdot\cdot \otimes e_{i_{2k}},
\end{equation}
where $(e_{1},...,e_{d})$ is the basis of $\mathbb{R}^{d}$. Notice that we will consider only the $2k$ case since for odd values the expected value of the iterated integrals is zero due to the symmetry of the fBm.

We recall two results which can be found for example in Baudoin and Coutin paper $\cite{BauCou}$. First, we present a reformulation of the Isserlis' (or Wick's) theorem.
\begin{lem}
Let $G = (G_{ 1 }, . . . , G_{ 2k })$ be a centered Gaussian vector. We have
\begin{equation}
\mathbb{E}(G_{ 1 }\cdot\cdot\cdot G_{ 2k })=\dfrac{1}{k!2^{k}}\sum_{\sigma\in\mathcal{G}_{2k}}\prod_{l=1}^{k}\mathbb{E}(G_{\sigma(2l-1)}G_{\sigma(2l)}),
\end{equation}
where $\mathcal{G}_{2k}$ is the group of the permutations of the set $\{1, . . . ,2k\}$.
\end{lem}
\noindent Therefore, from the previous lemma we have
\begin{equation}
\mathbb{E}\left(\int_{\Delta^{2k}[s,t]}dB^{m,I}\right)=\dfrac{1}{k!2^{k}}\sum_{\sigma\in\mathcal{G}_{2k}}\int_{\Delta^{2k}[s,t]}\prod_{l=1}^{k}\mathbb{E}\left(\dfrac{dB^{m,i_{\sigma(2l-1)}}}{dt_{\sigma(2l-1)}}\dfrac{dB^{m,i_{\sigma(2l)}}}{dt_{\sigma(2l)}}\right)dt_{1}\cdot\cdot\cdot dt_{2k}.
\end{equation}
Notice that if $t_{\sigma(2l-1)}\in[t_{j},t_{j+1}]$ and $t_{\sigma(2l)}\in[t_{i},t_{i+1}]$ then
\begin{equation}
\mathbb{E}\left(\dfrac{dB^{m,i_{\sigma(2l-1)}}}{dt_{\sigma(2l-1)}}\dfrac{dB^{m,i_{\sigma(2l)}}}{dt_{\sigma(2l)}}\right)=\delta_{i_{\sigma(2l)},i_{\sigma(2l-1)}} H(2H-1)m^{2}\int_{t_{i}}^{t_{i+1}}\int_{t_{j}}^{t_{j+1}}|x-y|^{2H-2}dxdy.
\end{equation}
Further, recall Theorem 31 of $\cite{BauCou}$.
\begin{thm}\label{31}
Assume $H > \frac{1}{2}$. Letting $I = (i_{1}, . . . , i_{2k})$ be a word, then
\begin{equation}
\mathbb{E}\left(\int_{\Delta^{2k}[0,1]}dB^{I}\right)=\dfrac{H^{k}(2H-1)^{k}}{k!2^{k}}\sum_{\sigma\in\mathcal{G}_{2k}}\int_{\Delta^{2k}[0,1]}\prod_{l=1}^{k}\delta_{i_{\sigma(2l)},i_{\sigma(2l-1)}}|t_{\sigma(2l)}-t_{\sigma(2l-1)}|^{2H-2}dt_{1}\cdot\cdot\cdot dt_{2k}.
\end{equation}
\end{thm}
\noindent We conclude this section with the following remark.
\begin{rem}
By using the scaling and the stationary increments property of the fBm, we have that
\begin{equation*}
\mathbb{E}\left(\int_{\Delta^{2k}[s,t]}dB^{I}\right)=\mathbb{E}\left(\int_{\Delta^{2k}[0,t-s]}dB^{I}\right)=(t-s)^{2kH}\mathbb{E}\left(\int_{\Delta^{2k}[0,1]}dB^{I}\right),\quad \text{and}
\end{equation*}
\begin{equation*}
\mathbb{E}\left(\int_{\Delta^{2k}[s,t]}dB^{m,I}\right)=\mathbb{E}\left(\int_{\Delta^{2k}[0,t-s]}dB^{m,I}\right)=(t-s)^{2kH}\mathbb{E}\left(\int_{\Delta^{2k}[0,1]}dB^{m,I}\right).
\end{equation*}
Hence, it is sufficient to focus on the case $\Delta^{2k}[0,1]$ in the proofs of our results.
\end{rem}
\section{Proof of Theorem \ref{pr1}}\label{ChapterSharp}
The first result we prove regards rate of convergence of the expected signature of the piecewise approximation of the fractional Brownian motion to its exact value. The sharpness of the rate of convergence is a delicate matter and it will be discussed at the end of this section.\\
In this section we prove Theorem $\ref{pr1}$. As mentioned before, this theorem covers the weak convergence rate of the signature of the fBm. A previous result on the strong convergence rate was obtained by Friz and Riedel in $\cite{FR}$. Recall from Theorem $\ref{31}$ that we have 
\begin{equation*}
\Bigg|\mathbb{E}\left(\int_{\Delta^{2k}[0,1]}dB^{I}\right)-\mathbb{E}\left(\int_{\Delta^{2k}[0,1]}dB^{m,I}\right)\Bigg|
\end{equation*}
\begin{equation*}
=\Bigg|\dfrac{H^{k}(2H-1)^{k}}{k!2^{k}}\sum_{\sigma\in\mathcal{G}_{2k}}\int_{\Delta^{2k}[0,1]}\prod_{l=1}^{k}\delta_{i_{\sigma(2l)},i_{\sigma(2l-1)}}|t_{\sigma(2l)}-t_{\sigma(2l-1)}|^{2H-2}
\end{equation*}
\begin{equation*}
-\prod_{l=1}^{k}\delta_{i_{\sigma(2l)},i_{\sigma(2l-1)}}\sum_{i,j=1}^{m}\textbf{1}_{[t_{i},t_{i+1}]\times[t_{j},t_{j+1}]}(t_{\sigma(2l)},t_{\sigma(2l-1)})m^{2}\int_{t_{i}}^{t_{i+1}}\int_{t_{j}}^{t_{j+1}}|x-y|^{2H-2}dxdydt_{1}\cdot\cdot\cdot dt_{2k}\Bigg|.
\end{equation*}
\begin{figure}
\centerline{\includegraphics[scale=0.5]{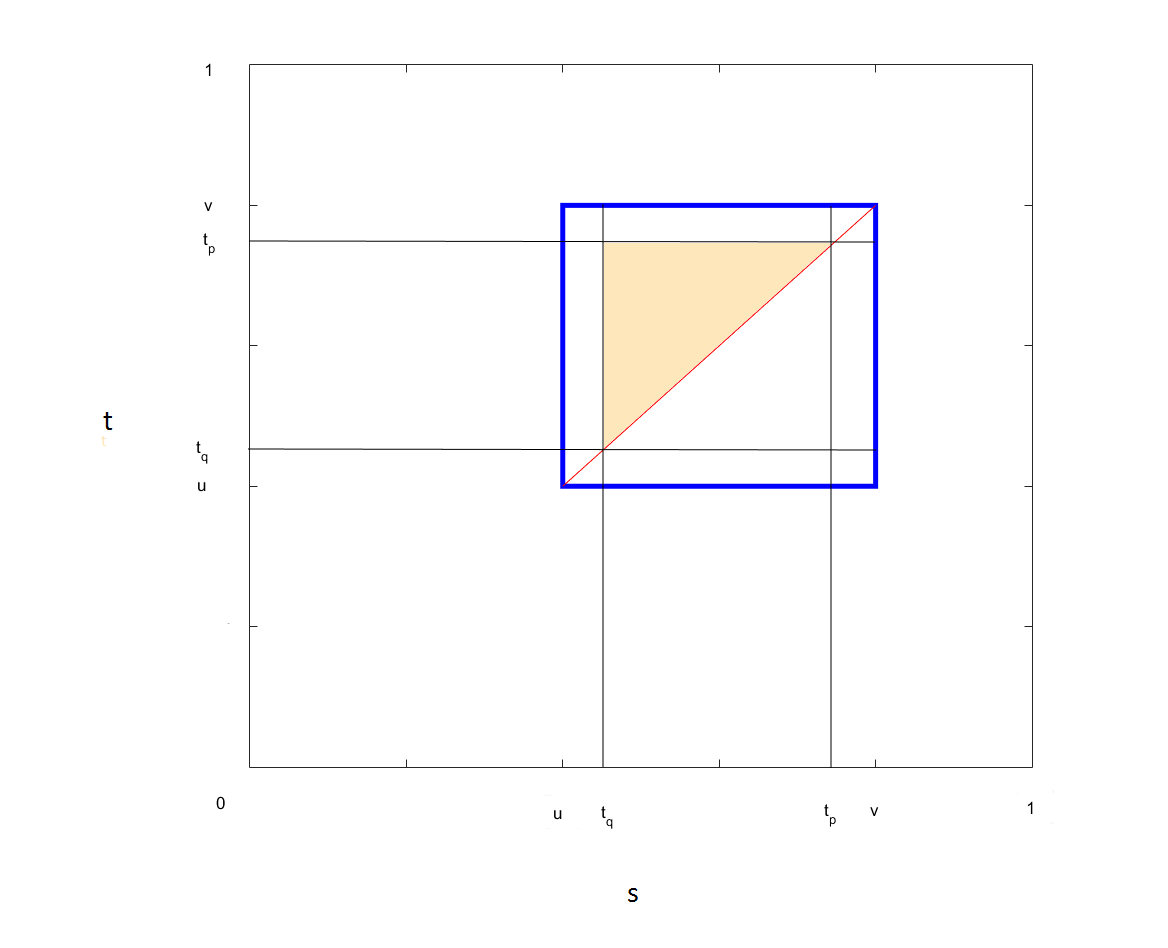}}
\caption{\small Representation of the area of the double integral (\ref{integrals}).}
\end{figure}
In order to understand this as a whole we need to first understand the single parts of it. In other words let us focus on
\begin{equation*}
\int_{u_{2}}^{v_{2}}\int_{u_{1}}^{v_{1}}\Big(|t-s|^{2H-2}-\sum_{i,j=1}^{m}\textbf{1}_{[t_{i},t_{i+1}]\times[t_{j},t_{j+1}]}(t,s)m^{2}\int_{t_{i}}^{t_{i+1}}\int_{t_{j}}^{t_{j+1}}|x-y|^{2H-2}dxdy\Big)dsdt,
\end{equation*}
where $u_{1},u_{2},v_{1},v_{2},t,s\in\{t_{1},...,t_{2k}\}$.
There are two possible cases. The first one is when $t$ and $s$ are next to each other. In this case, we have
\begin{equation}\label{first}
\int_{u}^{v}\int_{u}^{t}\Big(|t-s|^{2H-2}-\sum_{i,j=1}^{m}\textbf{1}_{[t_{i},t_{i+1}]\times[t_{j},t_{j+1}]}(t,s)m^{2}\int_{t_{i}}^{t_{i+1}}\int_{t_{j}}^{t_{j+1}}|x-y|^{2H-2}dxdy\Big)dsdt
\end{equation}
with $u,v,t,s\in\{t_{1},...,t_{2k}\}$ with $u\neq v\neq t\neq s$.
The second one is that $s$ and $t$ are not next to each other. Hence, we have
\begin{equation}\label{second}
\int_{u_{2}}^{v_{2}}\int_{u_{1}}^{v_{1}}\Big(|t-s|^{2H-2}-\sum_{i,j=1}^{m}\textbf{1}_{[t_{i},t_{i+1}]\times[t_{j},t_{j+1}]}(t,s)m^{2}\int_{t_{i}}^{t_{i+1}}\int_{t_{j}}^{t_{j+1}}|x-y|^{2H-2}dxdy\Big)dsdt
\end{equation}
with $u_{1}\neq u_{2}\neq v_{1}\neq v_{2}\neq t\neq s$.

We will focus first on the first case $(\ref{first})$, which is more delicate. We are going to split this double integral into different parts according to our piecewise approximation. In particular, let $t_{q}$ be the lowest grid time above $u$ and $t_{p}$ the highest grid time below $v$, as shown in the Figure 1. Then we have that our double integral can be rewritten in the following way:
\begin{equation}\label{integrals}
\int_{u<s<t<v}=\int_{u}^{v}\int_{u}^{t}=\int_{u}^{t_{q}}\int_{u}^{t}+\int_{t_{q}}^{t_{p}}\int_{u}^{t_{q}}+\int_{t_{p}}^{v}\int_{u}^{t_{q}}+\int_{t_{p}}^{v}\int_{t_{q}}^{t_{p}}+\int_{t_{p}}^{v}\int_{t_{p}}^{t}+\int_{t_{q}}^{t_{p}}\int_{t_{q}}^{t}.
\end{equation}
\noindent In order to better understand this decomposition it is helpful to look at the upper triangle in Figure 1 (\textit{i.e.} the triangle from the red diagonal to the upper and left blue edges). The first integral is the small triangle in the lower left. From here you first move up and then move right. The last integral is the central yellow triangle in Figure 1. For this integral the piecewise approximation and the exact value are equal, hence their difference is zero. This is because the piecewise approximation is equal to the exact value if the points in time considered coincide with the grid points.

What we want to show now is the speed at which the difference between the piecewise approximation and the exact value for these integrals is going to zero as the grid size goes to zero (\textit{i.e.} $m\rightarrow\infty$).
\\
In order to facilitate the reading and understanding of the content of this section, instead of having a 7 page long proof we decided to split the proof of Theorem \ref{pr1} in different small propositions.
\\ The following propositions regards the rate of convergence of each double integral of the right hand side of equation $(\ref{integrals})$. First, let us focus on the integrals: $\int_{u}^{t_{q}}\int_{u}^{t}$ and $\int_{t_{p}}^{v}\int_{t_{p}}^{t}$. We have the following proposition.
\begin{pro} The differences
\begin{equation*} 
\Bigg|\int_{u}^{t_{q}}\int_{u}^{t}\Big(|t-s|^{2H-2}-\sum_{i,j=1}^{m}\normalfont\textbf{1}_{[t_{i},t_{i+1}]\times[t_{j},t_{j+1}]}(t,s)m^{2}\int_{t_{i}}^{t_{i+1}}\int_{t_{j}}^{t_{j+1}}|x-y|^{2H-2}dxdy\Big)dsdt\Bigg|
\end{equation*}
and 
\begin{equation*} 
\Bigg|\int_{t_{p}}^{v}\int_{t_{p}}^{t}\Big(|t-s|^{2H-2}-\sum_{i,j=1}^{m}\normalfont\textbf{1}_{[t_{i},t_{i+1}]\times[t_{j},t_{j+1}]}(t,s)m^{2}\int_{t_{i}}^{t_{i+1}}\int_{t_{j}}^{t_{j+1}}|x-y|^{2H-2}dxdy\Big)dsdt\Bigg|
\end{equation*}
are bounded by
\begin{equation*}
\dfrac{m^{-2H}}{H(2H-1)}.
\end{equation*}
\end{pro}
\begin{proof}
\begin{equation*}
\int_{u}^{t_{q}}\int_{u}^{t}\Big(|t-s|^{2H-2}-\sum_{i,j=1}^{m}\textbf{1}_{[t_{i},t_{i+1}]\times[t_{j},t_{j+1}]}(t,s)m^{2}\int_{t_{i}}^{t_{i+1}}\int_{t_{j}}^{t_{j+1}}|x-y|^{2H-2}dxdy\Big)dsdt
\end{equation*}
\begin{equation*}
=\int_{u}^{t_{q}}\int_{u}^{t}\Big(|t-s|^{2H-2}-m^{2}\int_{t_{q-1}}^{t_{q}}\int_{t_{q-1}}^{t_{q}}|x-y|^{2H-2}dxdy\Big)dsdt
\end{equation*}
\begin{equation*}
=\dfrac{(t_{q}-u)^{2H}}{2H(2H-1)}-m^{2}\int_{u}^{t_{q}}\int_{u}^{t}\dfrac{(t_{q}-t_{q-1})^{2H}}{H(2H-1)}dsdt
\end{equation*}
\begin{equation*}
=\dfrac{(t_{q}-u)^{2H}}{2H(2H-1)}-m^{2-2H}\dfrac{(t_{q}-u)^{2}}{2H(2H-1)}.
\end{equation*}
Notice that
\begin{equation*}
\Bigg|\dfrac{(t_{q}-u)^{2H}}{2H(2H-1)}-m^{2-2H}\dfrac{(t_{q}-u)^{2}}{2H(2H-1)}\Bigg|\leq\dfrac{(t_{q}-u)^{2H}}{2H(2H-1)}+m^{2-2H}\dfrac{(t_{q}-u)^{2}}{2H(2H-1)}
\end{equation*}
\begin{equation*}
\leq \dfrac{m^{-2H}}{2H(2H-1)}+m^{2-2H}\dfrac{m^{-2}}{2H(2H-1)}=\dfrac{m^{-2H}}{H(2H-1)}.
\end{equation*}
The same reasoning done for this integral applies \textit{mutatis mutandis} to the integral $\int_{t_{p}}^{v}\int_{t_{p}}^{t}$.
\end{proof}
Let us focus on the integrals: $\int_{t_{q}}^{t_{p}}\int_{u}^{t_{q}}$ and $\int_{t_{p}}^{v}\int_{t_{q}}^{t_{p}}$. We have the following proposition.
\begin{pro}
The differences
\begin{equation*}
\Bigg|\int_{t_{q}}^{t_{p}}\int_{u}^{t_{q}}\Big(|t-s|^{2H-2}-\sum_{i,j=1}^{m}\normalfont\textbf{1}_{[t_{i},t_{i+1}]\times[t_{j},t_{j+1}]}(t,s)m^{2}\int_{t_{i}}^{t_{i+1}}\int_{t_{j}}^{t_{j+1}}|x-y|^{2H-2}dxdy\Big)dsdt\Bigg|
\end{equation*}
and
\begin{equation*}
\Bigg|\int_{t_{p}}^{v}\int_{t_{q}}^{t_{p}}\Big(|t-s|^{2H-2}-\sum_{i,j=1}^{m}\normalfont\textbf{1}_{[t_{i},t_{i+1}]\times[t_{j},t_{j+1}]}(t,s)m^{2}\int_{t_{i}}^{t_{i+1}}\int_{t_{j}}^{t_{j+1}}|x-y|^{2H-2}dxdy\Big)dsdt\Bigg|
\end{equation*}
are bounded by
\begin{equation*}
\left(\dfrac{2^{2H}+2}{H(2H-1)}+(4-4H)\sum_{i=1}^{\infty}i^{2H-3}\right)m^{-2H}.
\end{equation*}
\end{pro}
\begin{proof}
We have
\begin{equation*}
\int_{t_{q}}^{t_{p}}\int_{u}^{t_{q}}\Big(|t-s|^{2H-2}-\sum_{i,j=1}^{m}\textbf{1}_{[t_{i},t_{i+1}]\times[t_{j},t_{j+1}]}(t,s)m^{m}\int_{t_{i}}^{t_{i+1}}\int_{t_{j}}^{t_{j+1}}|x-y|^{2H-2}dxdy\Big)dsdt
\end{equation*}
\begin{equation*}
=\int_{t_{q}}^{t_{p}}\int_{u}^{t_{q}}\Big(|t-s|^{2H-2}-\sum_{t_{i}=t_{q}}^{t_{p-1}}\textbf{1}_{[t_{i},t_{i+1}]}(t)m^{2}\int_{t_{i}}^{t_{i+1}}\int_{t_{q-1}}^{t_{q}}|x-y|^{2H-2}dxdy\Big)dsdt
\end{equation*}
\begin{equation*}
=\sum_{t_{i}=t_{q}}^{t_{p-1}}\int_{t_{i}}^{t_{i+1}}\int_{u}^{t_{q}}\Big(|t-s|^{2H-2}-m^{2}\int_{t_{i}}^{t_{i+1}}\int_{t_{q-1}}^{t_{q}}|x-y|^{2H-2}dxdy\Big)dsdt.
\end{equation*}
Here we need to split the sum in two parts. The first part is composed by the first element of the sum, while the second is composed by the whole sum without the first element. Let us first consider the first part.
\begin{equation*}
\int_{t_{q}}^{t_{q+1}}\int_{u}^{t_{q}}\Big(|t-s|^{2H-2}-m^{2}\int_{t_{q}}^{t_{q+1}}\int_{t_{q-1}}^{t_{q}}|x-y|^{2H-2}dxdy\Big)dsdt
\end{equation*}
\begin{equation*}
=\dfrac{(t_{q+1}-u)^{2H}-(t_{q}-u)^{2H}-(t_{q+1}-t_{q})^{2H}}{2H(2H-1)}
\end{equation*}
\begin{equation*}
-m(t_{q+1}-t_{q})(t_{q}-u)\dfrac{(t_{q+1}-t_{q-1})^{2H}-(t_{q}-t_{q-1})^{2H}-(t_{q+1}-t_{q})^{2H}}{2H(2H-1)}.
\end{equation*}
Notice that every element in the numerators is smaller or equal than $m^{-2H}$ except $(t_{q+1}-u)^{2H}$ and $(t_{q+1}-t_{q-1})^{2H}$ for which the following relation holds $(t_{q+1}-u)^{2H}\leq (t_{q+1}-t_{q-1})^{2H}=2^{2H}m^{-2H}$. Therefore, taking absolute value of each element we have
\begin{equation*}
\leq \dfrac{2^{2H}+2}{H(2H-1)}m^{-2H}.
\end{equation*}
Now let's focus on the second part, that is
\begin{equation*}
\sum_{t_{i}=t_{q+1}}^{t_{p-1}}\int_{t_{i}}^{t_{i+1}}\int_{u}^{t_{q}}\Big(|t-s|^{2H-2}-m^{2}\int_{t_{i}}^{t_{i+1}}\int_{t_{q-1}}^{t_{q}}|x-y|^{2H-2}dydx\Big)dsdt
\end{equation*}
\begin{equation*}
=\sum_{t_{i}=t_{q+1}}^{t_{p-1}}\int_{t_{i}}^{t_{i+1}}\int_{u}^{t_{q}}\Big((t-s)^{2H-2}-m^{2}\int_{t_{i}}^{t_{i+1}}\int_{t_{q-1}}^{t_{q}}(x-y)^{2H-2}dydx\Big)dsdt
\end{equation*}
\begin{equation*}
=\sum_{t_{i}=t_{q+1}}^{t_{p-1}}\int_{t_{i}}^{t_{i+1}}\int_{u}^{t_{q}}m^{2}\Big(\int_{t_{i}}^{t_{i+1}}\int_{t_{q-1}}^{t_{q}}(t-s)^{2H-2}-(x-y)^{2H-2}dydx\Big)dsdt.
\end{equation*}
By applying Mean Value Theorem we have
\begin{equation*}
\leq \sum_{t_{i}=t_{q+1}}^{t_{p-1}}\int_{t_{i}}^{t_{i+1}}\int_{u}^{t_{q}}m^{2}\Big(\int_{t_{i}}^{t_{i+1}}\int_{t_{q-1}}^{t_{q}}\sup_{\xi\in[(t-s),(x-y)]}(2-2H)\xi^{2H-3}|t-s -(x-y)|dydx\Big)dsdt
\end{equation*}
\begin{equation*}
=\sum_{t_{i}=t_{q+1}}^{t_{p-1}}\int_{t_{i}}^{t_{i+1}}\int_{u}^{t_{q}}m^{2}\Big(\int_{t_{i}}^{t_{i+1}}\int_{t_{q-1}}^{t_{q}}(2-2H)(t_{i}-t_{q})^{2H-3}|t-x +y-s|dydx\Big)dsdt.
\end{equation*}
Observe that $|t-x|\leq m^{-1}$ and $|y-s|\leq m^{-1}$, hence $|t-x +y-s|\leq 2m^{-1}$. Thus, we have
\begin{equation*}
\leq\sum_{t_{i}=t_{q+1}}^{t_{p-1}}\int_{t_{i}}^{t_{i+1}}\int_{u}^{t_{q}}m^{2}\Big(\int_{t_{i}}^{t_{i+1}}\int_{t_{q-1}}^{t_{q}}(4-4H)(t_{i}-t_{q})^{2H-3}m^{-1}dydx\Big)dsdt
\end{equation*} 
\begin{equation*}
\leq (4-4H)m^{-3}\sum_{t_{i}=t_{q+1}}^{t_{p-1}}(t_{i}-t_{q})^{2H-3}.
\end{equation*}
Moreover, it is possible to notice that
\begin{equation*}
\leq (4-4H)m^{-3}\sum_{i=1}^{m}\left(\dfrac{i}{m}\right)^{2H-3}=(4-4H)m^{-2H}\sum_{i=1}^{m}i^{2H-3}\leq (4-4H)m^{-2H}\sum_{i=1}^{\infty}i^{2H-3}
\end{equation*}
\begin{equation*}
\leq Km^{-2H}
\end{equation*}
where $K$ is a finite constant independent of $m$, and that $\sum_{i=1}^{\infty}i^{2H-3}<\infty$ for $H<1$. 
\\
A similar reasoning applies to the integral $\int_{t_{p}}^{v}\int_{t_{q}}^{t_{p}}$.
\end{proof}
The last integral to be analysed is $\int_{t_{p}}^{v}\int_{u}^{t_{q}}$. 
\begin{pro}
The difference
\begin{equation*}
\Bigg|\int_{t_{p}}^{v}\int_{u}^{t_{q}}\Big(|t-s|^{2H-2}-\sum_{i,j=1}^{m}\normalfont\textbf{1}_{[t_{i},t_{i+1}]\times[t_{j},t_{j+1}]}(t,s)m^{2}\int_{t_{i}}^{t_{i+1}}\int_{t_{j}}^{t_{j+1}}|x-y|^{2H-2}dxdy\Big)dsdt\Bigg|
\end{equation*}
is bounded by
\begin{equation*}
\dfrac{3^{2H}+10(2^{2H})+2}{2H(2H-1)}m^{-2H}.
\end{equation*}
\end{pro}
\begin{proof}
Here we need to distinguish two cases. One in which $|u-v|\leq \frac{2}{m}$ and the other in which $|u-v|> \frac{2}{m}$. Let's focus first on the case $|u-v|> \frac{2}{m}$.
\begin{equation*}
\int_{t_{p}}^{v}\int_{u}^{t_{q}}\Big(|t-s|^{2H-2}-\sum_{i,j=1}^{m}\textbf{1}_{[t_{i},t_{i+1}]\times[t_{j},t_{j+1}]}(t,s)m^{2}\int_{t_{i}}^{t_{i+1}}\int_{t_{j}}^{t_{j+1}}|x-y|^{2H-2}dxdy\Big)dsdt
\end{equation*}
\begin{equation*}
=\int_{t_{p}}^{v}\int_{u}^{t_{q}}\Big(|t-s|^{2H-2}-m^{2}\int_{t_{p}}^{t_{p+1}}\int_{t_{q-1}}^{t_{q}}|x-y|^{2H-2}dxdy\Big)dsdt
\end{equation*}
\begin{equation*}
=\int_{t_{p}}^{v}\int_{u}^{t_{q}}m^{2}\Big(\int_{t_{p}}^{t_{p+1}}\int_{t_{q-1}}^{t_{q}}|t-s|^{2H-2}-|x-y|^{2H-2}dxdy\Big)dsdt.
\end{equation*}
By Mean Value Theorem we have
\begin{equation*}
\Bigg|\int_{t_{p}}^{v}\int_{u}^{t_{q}}m^{2}\Big(\int_{t_{p}}^{t_{p+1}}\int_{t_{q-1}}^{t_{q}}|t-s|^{2H-2}-|x-y|^{2H-2}dxdy\Big)dsdt\Bigg|
\end{equation*}
\begin{equation*}
\leq \int_{t_{p}}^{v}\int_{u}^{t_{q}}m^{2}\Big(\int_{t_{p}}^{t_{p+1}}\int_{t_{q-1}}^{t_{q}}\sup_{\xi\in[(t-s),(x-y)]}(2-2H)\xi^{2H-3}|t-s -(x-y)|dxdy\Big)dsdt
\end{equation*}
\begin{equation*}
=\int_{t_{p}}^{v}\int_{u}^{t_{q}}m^{2}\Big(\int_{t_{p}}^{t_{p+1}}\int_{t_{q-1}}^{t_{q}}(2-2H)(t_{p}-t_{q})^{2H-3}|t-x +y-s|dxdy\Big)dsdt.
\end{equation*}
Since $|u-v|> \frac{2}{m}$ then $t_{p}-t_{q}\geq m^{-1}$.
\begin{equation*}
\leq\int_{t_{p}}^{v}\int_{u}^{t_{q}}m^{2}\Big(\int_{t_{p}}^{t_{p+1}}\int_{t_{q-1}}^{t_{q}}(4-4H)m^{-2H+3}m^{-1}dxdy\Big)dsdt
\end{equation*}
\begin{equation*}
\leq(4-4H)m^{-2H}.
\end{equation*}
Finally consider the second case, \textit{i.e.} $|u-v|\leq \frac{2}{m}$.
\begin{equation*}
\int_{t_{p}}^{v}\int_{u}^{t_{q}}\Big(|t-s|^{2H-2}-m^{2}\int_{t_{p}}^{t_{p+1}}\int_{t_{q-1}}^{t_{q}}|x-y|^{2H-2}dxdy\Big)dsdt
\end{equation*}
\begin{equation*}
=\dfrac{(v-u)^{2H}-(t_{p}-u)^{2H}-(v-t_{q})^{2H}+(t_{p}-t_{q})^{2H}}{2H(2H-1)}
\end{equation*}
\begin{equation*}
-m^{2}(v-t_{p})(t_{q}-u)\dfrac{(t_{p+1}-t_{q-1})^{2H}-(t_{p}-t_{q-1})^{2H}-(t_{p+1}-t_{q})^{2H}+(t_{p}-t_{q})^{2H}}{2H(2H-1)}.
\end{equation*}
Now, since $|u-v|\leq \frac{2}{m}$ then $t_{p}-t_{q}\leq \frac{1}{m}$ so $t_{p+1}-t_{q-1}\leq \frac{3}{m}$. Therefore, we have
\begin{equation*}
\leq \left(\dfrac{3^{2H}+10(2^{2H})+2}{2H(2H-1)}\right)m^{-2H}.
\end{equation*}
Moreover, since
\begin{equation*}
(4-4H)<\dfrac{3^{2H}+10(2^{2H})+2}{2H(2H-1)},
\end{equation*}
then we can use the right hand side of the above inequality as the coefficient for the rate of convergence for the integral $\int_{t_{p}}^{v}\int_{u}^{t_{q}}$ independently of the value of $u$ and $v$.
\end{proof}
Therefore, we have the following proposition.
\begin{pro}\label{propositionintegrals}
The difference
\begin{equation*}
\Bigg|\int_{u<s<t<v}\Big(|t-s|^{2H-2}-\sum_{i,j=1}^{m}\normalfont\textbf{1}_{[t_{i},t_{i+1}]\times[t_{j},t_{j+1}]}(t,s)m^{2}\int_{t_{i}}^{t_{i+1}}\int_{t_{j}}^{t_{j+1}}|x-y|^{2H-2}dxdy\Big)dsdt\Bigg|
\end{equation*}
is bounded by
\begin{equation*}
Am^{-2H},
\end{equation*}
where $A$ is given by 
\begin{equation*}
A=2\left(\dfrac{1}{H(2H-1)}+\dfrac{2^{2H}+2}{H(2H-1)}+(4-4H)\sum_{i=1}^{\infty}i^{2H-3}\right)+\dfrac{3^{2H}+10(2^{2H})+2}{2H(2H-1)}.
\end{equation*}
\end{pro}
\begin{proof}
It is immediate from the previous propositions, from the fact that for the double integral $\int_{t_{q}}^{t_{p}}\int_{t_{q}}^{t}$ the difference is zero as discussed before and by considering equation $(\ref{integrals})$. Indeed, the difference for the double integral $\int_{t_{q}}^{t_{p}}\int_{t_{q}}^{t}$ is given by
\begin{equation*}
\int_{t_{q}}^{t_{p}}\int_{t_{q}}^{t}\Big(|t-s|^{2H-2}-\sum_{i,j=1}^{m}\normalfont\textbf{1}_{[t_{i},t_{i+1}]\times[t_{j},t_{j+1}]}(t,s)m^{2}\int_{t_{i}}^{t_{i+1}}\int_{t_{j}}^{t_{j+1}}|x-y|^{2H-2}dxdy\Big)dsdt.
\end{equation*}
Notice that this difference can be expressed as the sum of two classes of integrals. The first class contains the off-diagonal terms, which can be expressed as
\begin{equation*}
\int_{t_{i}}^{t_{i+1}}\int_{t_{j}}^{t_{j+1}}\Big(|t-s|^{2H-2}-m^{2}\int_{t_{i}}^{t_{i+1}}\int_{t_{j}}^{t_{j+1}}|x-y|^{2H-2}dxdy\Big)dsdt,
\end{equation*}
where $t_{i},t_{j}\in\{t_{q},t_{q+1},...,t_{p-1},t_{p}\}$. It is possible to see that this difference it is immediately zero.
\\ The second case is
\begin{equation*}
\int_{t_{i}}^{t_{i+1}}\int_{t_{i}}^{t}\Big(|t-s|^{2H-2}-m^{2}\int_{t_{i}}^{t_{i+1}}\int_{t_{i}}^{t_{i+1}}|x-y|^{2H-2}dxdy\Big)dsdt,
\end{equation*}
where $t_{i}\in\{t_{q},t_{q+1},...,t_{p-1}\}$. This difference is equal to
\begin{equation*}
=\dfrac{(t_{i+1}-t_{i})^{2H}}{2H(2H-1)}-\int_{t_{i}}^{t_{i+1}}\int_{t_{i}}^{t}m^{2}\dfrac{(t_{i+1}-t_{i})^{2H}}{H(2H-1)}dsdt
\end{equation*}
\begin{equation*}
=\dfrac{m^{-2H}}{2H(2H-1)}-\dfrac{m^{2-2H}}{H(2H-1)}\int_{t_{i}}^{t_{i+1}}\int_{t_{i}}^{t}dsdt
\end{equation*}
\begin{equation*}
=\dfrac{m^{-2H}}{2H(2H-1)}-\dfrac{m^{2-2H}}{H(2H-1)}\dfrac{m^{-2}}{2}=0.
\end{equation*}
This result could also be directly seen from the symmetric (with respect to the diagonal) property of the function $|t-s|^{2H-2}$.
\end{proof}
Now, recall the second case (\textit{i.e.} equation $(\ref{second})$), which we rewrite here:
\begin{equation*}
\int_{u_{2}}^{v_{2}}\int_{u_{1}}^{v_{1}}\Big(|t-s|^{2H-2}-\sum_{i,j=1}^{m}\textbf{1}_{[t_{i},t_{i+1}]\times[t_{j},t_{j+1}]}(t,s)m^{2}\int_{t_{i}}^{t_{i+1}}\int_{t_{j}}^{t_{j+1}}|x-y|^{2H-2}dxdy\Big)dsdt
\end{equation*}
with $u_{1}\neq u_{2}\neq v_{1}\neq v_{2}\neq t\neq s$.\\
Concerning this case we have not any more a triangle but a rectangle. First, observe that these rectangles never lie on the diagonal. This is because we have $v_{2}> u_{2}> v_{1}> u_{1}$. The value of the double integral for this rectangle can be obtained by the difference of double integrals of properly chosen triangles. That is
\begin{equation*}
\int_{u_{2}}^{v_{2}}\int_{u_{1}}^{v_{1}}=\int_{u_{1}}^{v_{2}}\int_{u_{1}}^{t}-\int_{u_{1}}^{u_{2}}\int_{u_{1}}^{t}-\int_{v_{1}}^{v_{2}}\int_{v_{1}}^{t}+\int_{v_{1}}^{u_{2}}\int_{v_{1}}^{t}.
\end{equation*}
Therefore, we have the following proposition.
\begin{pro}\label{newprop}
The difference
\begin{equation*}
\Bigg|\int_{u_{2}}^{v_{2}}\int_{u_{1}}^{v_{1}}\Big(|t-s|^{2H-2}-\sum_{i,j=1}^{m}\textbf{1}_{[t_{i},t_{i+1}]\times[t_{j},t_{j+1}]}(t,s)m^{2}\int_{t_{i}}^{t_{i+1}}\int_{t_{j}}^{t_{j+1}}|x-y|^{2H-2}dxdy\Big)dsdt\Bigg|
\end{equation*}
is bounded by
\begin{equation*}
4Am^{-2H},
\end{equation*}
where $A$ is given in Proposition \ref{propositionintegrals}.
\end{pro}
\begin{proof}
Let 
\begin{equation*}
f(t,s):= |t-s|^{2H-2}-\sum_{i,j=1}^{m}\textbf{1}_{[t_{i},t_{i+1}]\times[t_{j},t_{j+1}]}(t,s)m^{2}\int_{t_{i}}^{t_{i+1}}\int_{t_{j}}^{t_{j+1}}|x-y|^{2H-2}dxdy.
\end{equation*}
Then from the discussion in the above paragraph we have that
\begin{equation*}
\int_{u_{2}}^{v_{2}}\int_{u_{1}}^{v_{1}}f(t,s)dsdt=\int_{u_{1}}^{v_{2}}\int_{u_{1}}^{t}f(t,s)dsdt-\int_{u_{1}}^{u_{2}}\int_{u_{1}}^{t}f(t,s)dsdt-\int_{v_{1}}^{v_{2}}\int_{v_{1}}^{t}f(t,s)dsdt+\int_{v_{1}}^{u_{2}}\int_{v_{1}}^{t}f(t,s)dsdt
\end{equation*}
\begin{equation*}
\leq 4Am^{-2H},
\end{equation*}
where the last inequality follows from Proposition \ref{propositionintegrals}.
\end{proof}
\noindent Now we are ready to prove Theorem $\ref{pr1}$.
\begin{proof}[Proof (Theorem \ref{pr1}).]
\begin{equation*}
\Bigg|\mathbb{E}\left(\int_{\Delta^{2k}[0,1]}dB^{I}\right)-\mathbb{E}\left(\int_{\Delta^{2k}[0,1]}dB^{m,I}\right)\Bigg|=
\end{equation*}
\begin{equation*}
=\Bigg|\dfrac{H^{k}(2H-1)^{k}}{k!2^{k}}\sum_{\sigma\in\mathcal{G}_{2k}}\int_{\Delta^{2k}[0,1]}\prod_{l=1}^{k}\delta_{i_{\sigma(2l)},i_{\sigma(2l-1)}}|t_{\sigma(2l)}-t_{\sigma(2l-1)}|^{2H-2}
\end{equation*}
\begin{equation*}
-\prod_{l=1}^{k}\delta_{i_{\sigma(2l)},i_{\sigma(2l-1)}}\sum_{i,j=1}^{m}\textbf{1}_{[t_{i},t_{i+1}]\times[t_{j},t_{j+1}]}(t_{\sigma(2l)},t_{\sigma(2l-1)})m^{2}\int_{t_{i}}^{t_{i+1}}\int_{t_{j}}^{t_{j+1}}|x-y|^{2H-2}dxdydt_{1}\cdot\cdot\cdot dt_{2k}\Bigg|.
\end{equation*}
Now, by using the relation 
\begin{equation*}
\prod_{l=1}^{k} a_{l}-\prod_{l=1}^{k} b_{l}=(a_{1}-b_{1})\prod_{l=2}^{k} b_{i}+a_{1}(a_{2}-b_{2})\prod_{l=3}^{k} b_{i}+\cdot\cdot\cdot+\prod_{l=1}^{k-1} a_{l}(a_{k}-b_{k}),
\end{equation*}
where $a_{i},b_{i}\in\mathbb{R}$ for $i=1,...,k$, we have
\begin{equation*}
=\Bigg|\dfrac{H^{k}(2H-1)^{k}}{k!2^{k}}\sum_{\sigma\in\mathcal{G}_{2k}}\int_{\Delta^{2k}[0,1]}\delta_{i_{\sigma(2)},i_{\sigma(1)}}\Big(|t_{\sigma(2)}-t_{\sigma(1)}|^{2H-2}
\end{equation*}
\begin{equation*}
-\sum_{i,j=1}^{m}\textbf{1}_{[t_{i},t_{i+1}]\times[t_{j},t_{j+1}]}(t_{\sigma(2)},t_{\sigma(1)})m^{2}\int_{t_{i}}^{t_{i+1}}\int_{t_{j}}^{t_{j+1}}|x-y|^{2H-2}dxdy\Big)
\end{equation*}
\begin{equation*}
\prod_{l=2}^{k}\delta_{i_{\sigma(2l)},i_{\sigma(2l-1)}}\sum_{i,j=1}^{m}\textbf{1}_{[t_{i},t_{i+1}]\times[t_{j},t_{j+1}]}(t_{\sigma(2l)},t_{\sigma(2l-1)})m^{2}\int_{t_{i}}^{t_{i+1}}\int_{t_{j}}^{t_{j+1}}|x-y|^{2H-2}dxdy
\end{equation*}
\begin{equation*}
+...+\prod_{l=1}^{k-1}\delta_{i_{\sigma(2l)},i_{\sigma(2l-1)}}|t_{\sigma(2l)}-t_{\sigma(2l-1)}|^{2H-2}\delta_{i_{\sigma(2k)},i_{\sigma(2k-1)}}\Big(|t_{\sigma(2k)}-t_{\sigma(2k-1)}|^{2H-2}
\end{equation*}
\begin{equation}\label{last-probably}
-\sum_{i,j=1}^{m}\textbf{1}_{[t_{i},t_{i+1}]\times[t_{j},t_{j+1}]}(t_{\sigma(2k)},t_{\sigma(2k-1)})m^{2}\int_{t_{i}}^{t_{i+1}}\int_{t_{j}}^{t_{j+1}}|x-y|^{2H-2}dxdy\Big)dt_{1}\cdot\cdot\cdot dt_{2k}\Bigg|.
\end{equation}
By Fubini's theorem, given an integrable function $g:\mathbb{R}^{2k}\rightarrow\mathbb{R}$ and denoting $\textbf{1}_{\{s<t<r\}}:=\textbf{1}_{[s,r]}(t)$ we have that
\begin{equation*}
\int_{\Delta^{2k}[0,1]}g dt_{1}\cdots dt_{2k}= \int_{-\infty}^{\infty}\cdots\int_{-\infty}^{\infty}g \textbf{1}_{\{0<t_{1}<t_{2}\}}\cdots \textbf{1}_{\{t_{2k-1}<t_{2k}<1\}}dt_{1}\cdots dt_{2k}
\end{equation*}
\begin{equation*}
\stackrel{Fubini}{=} \int_{-\infty}^{\infty}\cdots\int_{-\infty}^{\infty}g \textbf{1}_{\{0<t_{1}<t_{2}\}}\cdots \textbf{1}_{\{t_{2k-1}<t_{2k}<1\}}
\end{equation*}
\begin{equation*}
dt_{\sigma(2l)}dt_{\sigma(2l-1)}dt_{1}\cdots dt_{\sigma(2l)-1}dt_{\sigma(2l)+1}\cdots dt_{\sigma(2l-1)-1}dt_{\sigma(2l-1)+1}\cdots dt_{2k}
\end{equation*}
\begin{equation*}
=\int_{-\infty}^{\infty}\cdots\int_{-\infty}^{\infty}g \textbf{1}_{\{t_{\sigma(2l)-1}<t_{\sigma(2l)}<t_{\sigma(2l)+1}\}} \textbf{1}_{\{t_{\sigma(2l-1)-1}<t_{\sigma(2l-1)}<t_{\sigma(2l-1)+1}\}}dt_{\sigma(2l)}dt_{\sigma(2l-1)}
\end{equation*}
\begin{equation*}
\textbf{1}_{\{0<t_{1}<t_{2}\}}\cdots \textbf{1}_{\{t_{\sigma(2l)-2}<t_{\sigma(2l)-1}<t_{\sigma(2l)+1}\}}\textbf{1}_{\{t_{\sigma(2l)-1}<t_{\sigma(2l)+1}<t_{\sigma(2l)+2}\}}\cdots \textbf{1}_{\{t_{\sigma(2l-1)-2}<t_{\sigma(2l-1)-1}<t_{\sigma(2l-1)+1}\}}
\end{equation*}
\begin{equation*}
\textbf{1}_{\{t_{\sigma(2l-1)-1}<t_{\sigma(2l-1)+1}<t_{\sigma(2l-1)+2}\}}\cdots \textbf{1}_{\{t_{2k-1}<t_{2k}<1\}}dt_{1}\cdots dt_{\sigma(2l)-1}dt_{\sigma(2l)+1}\cdots dt_{\sigma(2l-1)-1}dt_{\sigma(2l-1)+1}\cdots dt_{2k}
\end{equation*}
\begin{equation*}
=\int_{0<t_{1}<...<t_{\sigma(2l)-1}<t_{\sigma(2l)+1}<...<t_{\sigma(2l-1)-1}<t_{\sigma(2l-1)+1}<...<t_{2k}<1}\int_{t_{\sigma(2l-1)-1}}^{t_{\sigma(2l-1)+1}}\int_{t_{\sigma(2l)-1}}^{t_{\sigma(2l)+1}}g
\end{equation*}
\begin{equation*}
dt_{\sigma(2l)}dt_{\sigma(2l-1)}dt_{1}\cdots dt_{\sigma(2l)-1}dt_{\sigma(2l)+1}\cdots dt_{\sigma(2l-1)-1}dt_{\sigma(2l-1)+1}\cdots dt_{2k},
\end{equation*}
where we assumed without loss of generality that $t_{\sigma(2l)}<t_{\sigma(2l-1)}$. Notice that, for example, $\textbf{1}_{\{0<t_{1}<t_{2}\}}$ is a function of $t_{1}$ and that $\textbf{1}_{\{0<t_{1}<t_{2}\}}\textbf{1}_{\{t_{1}<t_{2}<t_{3}\}}=\textbf{1}_{\{0<t_{1}<t_{3}\}}\textbf{1}_{\{t_{1}<t_{2}<t_{3}\}}=\textbf{1}_{\{0<t_{1}<t_{2}\}}\textbf{1}_{\{0<t_{2}<t_{3}\}}$. In our case the function $g$ is given by the integrand of $(\ref{last-probably})$, which is integrable since $H>\frac{1}{2}$. Now, by using this fact together with Proposition \ref{propositionintegrals} and Proposition \ref{newprop}, we have that $(\ref{last-probably})$ is bounded by
\begin{equation*}
\leq \dfrac{4Am^{-2H}H^{k}(2H-1)^{k}}{k!2^{k}}
\end{equation*}
\begin{equation*}
\sum_{\sigma\in\mathcal{G}_{2k}}\bigg[\int_{0<t_{1}<...<t_{\sigma(2)-1}<t_{\sigma(2)+1}<...<t_{\sigma(1)-1}<t_{\sigma(1)+1}<...t_{2k}<1}
\end{equation*}
\begin{equation*}
\prod_{l=2}^{k}\delta_{i_{\sigma(2l)},i_{\sigma(2l-1)}}\sum_{i,j=1}^{m}\textbf{1}_{[t_{i},t_{i+1}]\times[t_{j},t_{j+1}]}(t_{\sigma(2l)},t_{\sigma(2l-1)})m^{2}\int_{t_{i}}^{t_{i+1}}\int_{t_{j}}^{t_{j+1}}|x-y|^{2H-2}dxdy
\end{equation*}
\begin{equation*}
dt_{1}\cdots dt_{\sigma(2)-1}dt_{\sigma(2)+1}\cdots dt_{\sigma(1)-1}dt_{\sigma(1)+1}\cdots dt_{2k}
\end{equation*}
\begin{equation*}
+...+\int_{0<t_{1}<...<t_{\sigma(2k)-1}<t_{\sigma(2k)+1}<...<t_{\sigma(2k-1)-1}<t_{\sigma(2k-1)+1}<...t_{2k}<1}
\end{equation*}
\begin{equation}\label{k-addends}
\prod_{l=1}^{k-1}\delta_{i_{\sigma(2l)},i_{\sigma(2l-1)}}|t_{\sigma(2l)}-t_{\sigma(2l-1)}|^{2H-2}dt_{1}\cdots dt_{\sigma(2k)-1}dt_{\sigma(2k)+1}\cdots dt_{\sigma(2k-1)-1}dt_{\sigma(2k-1)+1}\cdots dt_{2k}\bigg].
\end{equation}
Observe that we have not any more the absolute value since it is a positive value. Now, consider the single addends inside the square parenthesis in the formula above. Each element is bounded by 
\begin{equation*}
\int_{0}^{1}...\int_{0}^{1}\prod_{l=2}^{k}|t_{\sigma(2l)}-t_{\sigma(2l-1)}|^{2H-2}dt_{1}\cdots dt_{\sigma(2)-1}dt_{\sigma(2)+1}\cdots dt_{\sigma(1)-1}dt_{\sigma(1)+1}\cdots dt_{2k}.
\end{equation*}
The reason why we have $l=2$ to $k$ in the product does not matter. Indeed, once we take the integral to be over $[0, 1]^{2k-2}$ instead of $\Delta^{2k-2}[0,1]$ it is not important which permutation we consider. We can move the integrands around using Tonelli's theorem since they are positive value. Further, the above bound is the same as considering $m=0$, because when the extremes of the integrals are grid points the piecewise approximation and the exact value coincide. This integral reduces to
\begin{equation*}
=\dfrac{1}{H^{k-1}(2H-1)^{k-1}}.
\end{equation*}
Therefore, we can bound $(\ref{k-addends})$ by
\begin{equation}\label{rate}
\dfrac{4Am^{-2H}H^{k}(2H-1)^{k}(2k)!}{k!2^{k}}\dfrac{k}{H^{k-1}(2H-1)^{k-1}},
\end{equation}
where we have used the fact that $\sum_{\sigma\in\mathcal{G}_{2k}}$, which is the sum over all permutations of the set $\{1,...,2k\}$, corresponds to $(2k)!$. Therefore, we can rewrite the formula above as
\begin{equation*}
=Cm^{-2H},\quad\text{where}\quad C:=\dfrac{4AH(2H-1)(2k)!}{(k-1)!2^{k}}<\infty.
\end{equation*}
This concludes the proof of the rate of convergence.
\end{proof}
The sharpness of the bound is a delicate matter. It has been proved by $\cite{NiXu}$ that for the Brownian motion the sharp rate of convergence is $1$, which motivates us to believe in the sharpness of our result. In order to be completely sure about it we need to prove that
\begin{equation*}
\limsup\limits_{m\rightarrow\infty}m^{2H}\Bigg|\mathbb{E}\left(\int_{\Delta^{2k}[0,1]}dB^{I}\right)-\mathbb{E}\left(\int_{\Delta^{2k}[0,1]}dB^{m,I}\right)\Bigg|>0.
\end{equation*}
However, the piecewise approximation makes this idea extremely difficult to put in practice. This is because, when we focus on the double integral $\int_{u<s<t<v}$, the piecewise approximation and the exact value are equal when the times $u$ and $v$ correspond to grid points $t_{q}$ and $t_{p}$ respectively, but they are not equal otherwise. Hence, as $m\rightarrow\infty$ the difference between the exact value and the piecewise approximation goes to zero not linearly. This makes the problem of the sharpness of the rate of convergence to appear impossible.

A possible solution is to consider the iterated integral as a whole without focusing on the general double integral. But this seems a difficult way of dealing with the problem. Indeed, we proceeded on with this way at the beginning of this work obtaining a bound with a rate of convergence of $2H$ but with a non integrable coefficient. For these reasons we are confident that the rate of convergence of $2H$ is sharp.

A final remark regards the fact that $m^{-2H}<m^{-1}$, which means that the difference goes to zero faster for the fBm than the Bm. This is an expected result since the fBm with $H>\frac{1}{2}$ is a smoother path than the Bm. This higher smoothness is related to the positive correlation of the increments of the fBm with $H>\frac{1}{2}$. Indeed, we can see that the rate of convergence becomes smaller (hence the convergence faster) as the positive correlation (represented by $H$) increases.
\section{Proof of Theorem \ref{pr2}}
As it is possible to see from equation $(\ref{rate})$ the coefficient of the rate of convergence goes to infinity as the degree of the truncated signature $k$ goes to infinity. In other words, $C\rightarrow\infty$ as $k\rightarrow\infty$. In the following, we improve the estimate of this coefficient and show that it goes to zero as $k$ goes to infinity.

In this section we prove Theorem \ref{pr2}. However, first we prove the following proposition.
\begin{pro}\label{newprop-permutation}
The following equality holds
\begin{equation*}
\sum_{\sigma\in\mathcal{G}_{2k}}\int_{0<t_{1}<...<t_{\sigma(2j-1)-1}<t_{\sigma(2j-1)+1}<...<t_{\sigma(2j)-1}<t_{\sigma(2j)+1}<...<t_{2k}<1}
\end{equation*}
\begin{equation*}
\prod_{l=1,l\neq j}^{k}\delta_{i_{\sigma(2l)},i_{\sigma(2l-1)}}|t_{\sigma(2l)}-t_{\sigma(2l-1)}|^{2H-2}dt_{1}\cdots dt_{\sigma(2j)-1}dt_{\sigma(2j)+1}\cdots dt_{\sigma(2j-1)-1}dt_{\sigma(2j-1)+1}\cdots dt_{2k}
\end{equation*}
\begin{equation}\label{new-permutation}
=2k(2k-1)\sum_{\tau\in\mathcal{G}_{2k-2}}\int_{\Delta^{2k-2}[0,1]}\prod_{l=1}^{k-1}\delta_{i_{\tau(2l)},i_{\tau(2l-1)}}|s_{\tau(2l)}-s_{\tau(2l-1)}|^{2H-2}\text{d}s_{1}\cdots \text{d}s_{2k-2},
\end{equation}
where $\mathcal{G}_{2k-2}$ is the set of all permutations of the set $\{1,...,2k-2\}$ and $j=1,...,k$.
\end{pro}
\begin{proof}
For the moment, fix $\sigma\in\mathcal{G}_{2k}$ and consider the case $j=1$. By a simple change of variables we have
\begin{equation*}
\int_{0<t_{1}<...<t_{\sigma(1)-1}<t_{\sigma(1)+1}<...<t_{\sigma(2)-1}<t_{\sigma(2)+1}<...<t_{2k}<1}
\end{equation*}
\begin{equation*}
\prod_{l=2}^{k}\delta_{i_{\sigma(2l)},i_{\sigma(2l-1)}}|t_{\sigma(2l)}-t_{\sigma(2l-1)}|^{2H-2}dt_{1}\cdots dt_{\sigma(2)-1}dt_{\sigma(2)+1}\cdots dt_{\sigma(1)-1}dt_{\sigma(1)+1}\cdots dt_{2k}
\end{equation*}
\begin{equation*}
= \int_{\Delta^{2k-2}[0,1]}\prod_{l=1}^{k-1}\delta_{i_{\tau(2l)},i_{\tau(2l-1)}}|s_{\tau(2l)}-s_{\tau(2l-1)}|^{2H-2}\text{d}s_{1}\cdots \text{d}s_{2k-2},
\end{equation*}
where $\tau\in\mathcal{G}_{2k-2}$. Now, given a permutation $\tau\in\mathcal{G}_{2k-2}$ there are different $\sigma$s, say $\sigma_{1},\sigma_{2},...\in\mathcal{G}_{2k}$ such that they satisfy the same equality as above. In particular, for a given $\tau$ there are precisely $2k(2k-1)$ of these $\sigma$s that satisfy the equality. That is
\begin{equation*}
\sum_{j=1}^{2k(2k-1)}\int_{0<t_{1}<...<t_{\sigma_{j}(1)-1}<t_{\sigma_{j}(1)+1}<...<t_{\sigma_{j}(2)-1}<t_{\sigma_{j}(2)+1}<...<1}
\end{equation*}
\begin{equation*}
\prod_{l=2}^{k}\delta_{i_{\sigma_{j}(2l)},i_{\sigma_{j}(2l-1)}}|t_{\sigma_{j}(2l)}-t_{\sigma_{j}(2l-1)}|^{2H-2}dt_{1}\cdots dt_{\sigma(2)-1}dt_{\sigma(2)+1}\cdots dt_{\sigma(1)-1}dt_{\sigma(1)+1}\cdots dt_{2k}
\end{equation*}
\begin{equation*}
=2k(2k-1)\int_{\Delta^{2k-2}[0,1]}\prod_{l=1}^{k-1}\delta_{i_{\tau(2l)},i_{\tau(2l-1)}}|s_{\tau(2l)}-s_{\tau(2l-1)}|^{2H-2}\text{d}s_{1}\cdots \text{d}s_{2k-2}.
\end{equation*}
The $2k(2k-1)$ factor comes from the possible values that $\{\sigma(1),\sigma(2)\}$ may assume in the set $\{1,...,2k\}$ leaving the order of the remaining $2k-2$ elements, which are the values $\sigma(3),...,\sigma(2k)$, unchanged. That is the $2k(2k-1)$ comes from the $2$-permutation of $2k$.

The next step is to see that it is possible to reformulate all the permutations $\sigma\in\mathcal{G}_{2k}$ in terms of the permutations $\tau\in\mathcal{G}_{2k-2}$ times $2k(2k-1)$. In particular we have the following equality
\begin{equation*}
\sum_{\sigma\in\mathcal{G}_{2k}}\int_{0<t_{1}<...<t_{\sigma(1)-1}<t_{\sigma(1)+1}<...<t_{\sigma(2)-1}<t_{\sigma(2)+1}<...<1}
\end{equation*}
\begin{equation*}
\prod_{l=2}^{k}\delta_{i_{\sigma(2l)},i_{\sigma(2l-1)}}|t_{\sigma(2l)}-t_{\sigma(2l-1)}|^{2H-2}\text{d}t_{\sigma(2l)}\text{d}t_{\sigma(2l-1)}
\end{equation*}
\begin{equation*}
=2k(2k-1)\sum_{\tau\in\mathcal{G}_{2k-2}}\int_{\Delta^{2k-2}[0,1]}\prod_{l=1}^{k-1}\delta_{i_{\tau(2l)},i_{\tau(2l-1)}}|s_{\tau(2l)}-s_{\tau(2l-1)}|^{2H-2}\text{d}s_{1}\cdots \text{d}s_{2k-2}.
\end{equation*}
which is our desired equality $(\ref{new-permutation})$. The reason why this holds is that it is possible to decompose the permutations $\mathcal{G}_{2k}$ (the $\sigma$s) into the 2-permutations of $2k$ (which are $2k(2k-1)$ permutations for each permutation $\tau$ and they do not modify the value of the integral) times the permutations $\mathcal{G}_{2k-2}$ (the $\tau$s). Indeed, for $\mathcal{G}_{2k}$ we have $(2k)!$ permutations (and of course they are all different from each other) and for the other we have $(2k-2)!\cdot2k(2k-1)=(2k)!$ permutations (and they are also all different from each other).

It is possible to see that the same arguments apply to the case $j=1,...,k$. For example, take $j=k$ then it is easy to see that we have
\begin{equation*}
\sum_{\sigma\in\mathcal{G}_{2k}}\int_{0<t_{1}<...<t_{\sigma(2k)-1}<t_{\sigma(2k)+1}<...<t_{\sigma(2k-1)-1}<t_{\sigma(2k-1)+1}<...<t_{2k}<1}
\end{equation*}
\begin{equation*}
\prod_{l=1}^{k-1}\delta_{i_{\sigma(2l)},i_{\sigma(2l-1)}}|t_{\sigma(2l)}-t_{\sigma(2l-1)}|^{2H-2}dt_{1}\cdots dt_{\sigma(2k)-1}dt_{\sigma(2k)+1}\cdots dt_{\sigma(2k-1)-1}dt_{\sigma(2k-1)+1}\cdots dt_{2k}
\end{equation*}
\begin{equation*}
=2k(2k-1)\sum_{\tau\in\mathcal{G}_{2k-2}}\int_{\Delta^{2k-2}[0,1]}\prod_{l=1}^{k-1}\delta_{i_{\tau(2l)},i_{\tau(2l-1)}}|s_{\tau(2l)}-s_{\tau(2l-1)}|^{2H-2}\text{d}s_{1}\cdots \text{d}s_{2k-2}.
\end{equation*}
Since the proof is based on combinatorial arguments, which sometimes are difficult to grasp, we added a longer explanation of the result in the Appendix.
\end{proof}
We are now ready to prove Theorem \ref{pr2}.
\begin{proof}[Proof (Theorem \ref{pr2}).]
First, multiply our main object of study by $m^{2H}$, namely
\begin{equation*}
\Bigg|\mathbb{E}\left(\int_{\Delta^{2k}[0,1]}dB^{I}\right)-\mathbb{E}\left(\int_{\Delta^{2k}[0,1]}dB^{m,I}\right)\Bigg|m^{2H}.
\end{equation*}
We can use our estimates obtained in ($\ref{k-addends}$) to get:
\begin{equation*}
\Bigg|\mathbb{E}\left(\int_{\Delta^{2k}[0,1]}dB^{I}\right)-\mathbb{E}\left(\int_{\Delta^{2k}[0,1]}dB^{m,I}\right)\Bigg|m^{2H}\leq \dfrac{4AH^{k}(2H-1)^{k}}{k!2^{k}}
\end{equation*}
\begin{equation*}
\sum_{\sigma\in\mathcal{G}_{2k}}\bigg[\int_{0<t_{1}<...<t_{\sigma(2)-1}<t_{\sigma(2)+1}<...<t_{\sigma(1)-1}<t_{\sigma(1)+1}<...t_{2k}<1}
\end{equation*}
\begin{equation*}
\prod_{l=2}^{k}\delta_{i_{\sigma(2l)},i_{\sigma(2l-1)}}\sum_{i,j=1}^{m}\textbf{1}_{[t_{i},t_{i+1}]\times[t_{j},t_{j+1}]}(t_{\sigma(2l)},t_{\sigma(2l-1)})m^{2}\int_{t_{i}}^{t_{i+1}}\int_{t_{j}}^{t_{j+1}}|x-y|^{2H-2}dxdy
\end{equation*}
\begin{equation*}
dt_{1}\cdots dt_{\sigma(2)-1}dt_{\sigma(2)+1}\cdots dt_{\sigma(1)-1}dt_{\sigma(1)+1}\cdots dt_{2k}
\end{equation*}
\begin{equation*}
+...+\int_{0<t_{1}<...<t_{\sigma(2k)-1}<t_{\sigma(2k)+1}<...<t_{\sigma(2k-1)-1}<t_{\sigma(2k-1)+1}<...t_{2k}<1}
\end{equation*}
\begin{equation*}
\prod_{l=1}^{k-1}\delta_{i_{\sigma(2l)},i_{\sigma(2l-1)}}|t_{\sigma(2l)}-t_{\sigma(2l-1)}|^{2H-2}dt_{1}\cdots dt_{\sigma(2k)-1}dt_{\sigma(2k)+1}\cdots dt_{\sigma(2k-1)-1}dt_{\sigma(2k-1)+1}\cdots dt_{2k}\bigg].
\end{equation*}
Now, taking the limit as $m\rightarrow\infty$, by dominated convergence theorem we obtain
\begin{equation*}
\dfrac{4AH^{k}(2H-1)^{k}}{k!2^{k}}\sum_{\sigma\in\mathcal{G}_{2k}}\bigg[\int_{0<t_{1}<...<t_{\sigma(2)-1}<t_{\sigma(2)+1}<...<t_{\sigma(1)-1}<t_{\sigma(1)+1}<...t_{2k}<1}
\end{equation*}
\begin{equation*}
\prod_{l=2}^{k}\delta_{i_{\sigma(2l)},i_{\sigma(2l-1)}}|t_{\sigma(2l)}-t_{\sigma(2l-1)}|^{2H-2}dt_{1}\cdots dt_{\sigma(2)-1}dt_{\sigma(2)+1}\cdots dt_{\sigma(1)-1}dt_{\sigma(1)+1}\cdots dt_{2k}
\end{equation*}
\begin{equation*}
+...+\int_{0<t_{1}<...<t_{\sigma(2k)-1}<t_{\sigma(2k)+1}<...<t_{\sigma(2k-1)-1}<t_{\sigma(2k-1)+1}<...t_{2k}<1}
\end{equation*}
\begin{equation*}
\prod_{l=1}^{k-1}\delta_{i_{\sigma(2l)},i_{\sigma(2l-1)}}|t_{\sigma(2l)}-t_{\sigma(2l-1)}|^{2H-2}dt_{1}\cdots dt_{\sigma(2k)-1}dt_{\sigma(2k)+1}\cdots dt_{\sigma(2k-1)-1}dt_{\sigma(2k-1)+1}\cdots dt_{2k}\bigg].
\end{equation*}
From Proposition \ref{newprop-permutation} we have
\begin{equation}\label{problem2}
=\dfrac{4AH^{k}(2H-1)^{k}}{k!2^{k}}2k(2k-1)k\sum_{\tau\in\mathcal{G}_{2k-2}}\int_{\Delta^{2k-2}[0,1]}\prod_{l=1}^{k-1}\delta_{i_{\tau(2l)},i_{\tau(2l-1)}}|s_{\tau(2l)}-s_{\tau(2l-1)}|^{2H-2}\text{d}s_{1}\cdots \text{d}s_{2k-2}.
\end{equation}
Notice that the function $f(s_{1},...,s_{2k-2})$ defined as
\begin{equation*}
f(s_{1},...,s_{2k-2}):=\sum_{\tau\in\mathcal{G}_{2k-2}}\prod_{l=1}^{k-1}\delta_{i_{\tau(2l)},i_{\tau(2l-1)}}|s_{\tau(2l)}-s_{\tau(2l-1)}|^{2H-2}
\end{equation*}
is symmetric with respect to the diagonal of the $2k-2$-dimensional hypercube. Hence, we have
\begin{equation*}
\int_{\Delta^{2k-2}[0,1]}f(s_{1},...,s_{2k-2})ds_{1}\cdots ds_{2k-2}=\dfrac{1}{(2k-2)!}\int_{[0,1]^{2k-2}}f(s_{1},...,s_{2k-2})ds_{1}\cdots ds_{2k-2}.
\end{equation*}
Therefore, we have that the formula $(\ref{problem2})$ is equal to
\begin{equation*}
\dfrac{4AH^{k}(2H-1)^{k}}{k!2^{k}}\dfrac{2k(2k-1)k}{(2k-2)!}\sum_{\tau\in\mathcal{G}_{2k-2}}\int_{[0,1]^{2k-2}}\prod_{l=1}^{k-1}\delta_{i_{\tau(2l)},i_{\tau(2l-1)}}|s_{\tau(2l)}-s_{\tau(2l-1)}|^{2H-2}\text{d}s_{1}\cdots \text{d}s_{2k-2}
\end{equation*}
\begin{equation}\label{final}
\leq \dfrac{4AH^{k}(2H-1)^{k}}{k!2^{k}}\dfrac{2k(2k-1)k}{(2k-2)!}\dfrac{(2k-2)!}{H^{k-1}(2H-1)^{k-1}}=\dfrac{8AH(2H-1)}{(k-1)!2^{k}}k(2k-1)
\end{equation}
Recall that $A$ is given by
\begin{equation*}
A=2\left(\dfrac{1}{H(2H-1)}+\dfrac{2^{2H}+2}{H(2H-1)}+(4-4H)\sum_{i=1}^{\infty}i^{2H-3}\right)+\dfrac{3^{2H}+10(2^{2H})+2}{2H(2H-1)}
\end{equation*}
and let $\tilde{A}$ to be defined as
\begin{equation*}
\tilde{A}:=8AH(2H-1)=16\left(3+2^{2H}+H(2H-1)(4-4H)\sum_{i=1}^{\infty}i^{2H-3}\right)+4(3^{2H}+10(2^{2H})+2)
\end{equation*}
\begin{equation*}
=56(1+2^{2H})+4(3^{2H})+16H(2H-1)(4-4H)\sum_{i=1}^{\infty}i^{2H-3}.
\end{equation*}
Thus, equation $(\ref{final})$ is equal to $\dfrac{\tilde{A}k(2k-1)}{(k-1)!2^{k}}$.
\end{proof}
Notice that
\begin{equation*}
\dfrac{\tilde{A}k(2k-1)}{(k-1)!2^{k}}\rightarrow 0\qquad as \qquad k\rightarrow\infty.
\end{equation*}
Thus, the coefficient of the rate of convergence goes to zero as the number of iterated integrals increases (\textit{i.e.} as the order of the signature increases) and it goes very fast.
\section{Proof of Theorem \ref{pr3}}
In this section we shift the focus we had in the last two sections. Our object is now solely the expected signature of the fractional Brownian motion, for which we prove a simple but sharp estimate for it.
\\ Indeed, we prove Theorem \ref{pr3}. The proof of this theorem is very short and the arguments are similar to the ones used at the end of the proof of Theorem \ref{pr2}.
\begin{proof}
Consider the $2k$-term of the expected signature of the fractional Brownian motion for $H>\frac{1}{2}$, we have
\begin{equation*}
\mathbb{E}\left(\int_{\Delta^{2k}[0,1]}dB^{I}\right)=\dfrac{H^{k}(2H-1)^{k}}{k!2^{k}}\sum_{\sigma\in\mathcal{G}_{2k}}\int_{\Delta^{2k}[0,1]}\prod_{l=1}^{k}\delta_{i_{\sigma(2l)},i_{\sigma(2l-1)}}|t_{\sigma(2l)}-t_{\sigma(2l-1)}|^{2H-2}dt_{1}\cdot\cdot\cdot dt_{2k}
\end{equation*}
\begin{equation*}
=\dfrac{H^{k}(2H-1)^{k}}{k!2^{k}}\frac{1}{(2k)!}\sum_{\sigma\in\mathcal{G}_{2k}}\int_{[0,1]^{2k}}\prod_{l=1}^{k}\delta_{i_{\sigma(2l)},i_{\sigma(2l-1)}}|t_{\sigma(2l)}-t_{\sigma(2l-1)}|^{2H-2}dt_{1}\cdot\cdot\cdot dt_{2k}
\end{equation*}
\begin{equation*}
=\dfrac{H^{k}(2H-1)^{k}}{k!2^{k}}\frac{1}{(2k)!}\sum_{\sigma\in\mathcal{G}_{2k}}\prod_{l=1}^{k}\delta_{i_{\sigma(2l)},i_{\sigma(2l-1)}}\int_{0}^{1}\int_{0}^{1}|t-s|^{2H-2}dsdt
\end{equation*}
\begin{equation*}
=\dfrac{H^{k}(2H-1)^{k}}{k!2^{k}}\frac{1}{(2k)!}\frac{1}{H^{k}(2H-1)^{k}}\sum_{\sigma\in\mathcal{G}_{2k}}\prod_{l=1}^{k}\delta_{i_{\sigma(2l)},i_{\sigma(2l-1)}}\leq\dfrac{1}{k!2^{k}}
\dfrac{1}{(2k)!}(2k)!=\dfrac{1}{k!2^{k}},
\end{equation*}
by using only the fact that the function
\begin{equation*}
\sum_{\sigma\in\mathcal{G}_{2k}}\prod_{l=1}^{k}\delta_{i_{\sigma(2l)},i_{\sigma(2l-1)}}|t_{\sigma(2l)}-t_{\sigma(2l-1)}|^{2H-2}
\end{equation*}
is symmetric with respect to the diagonal.
\end{proof}
A similar but less sharp estimate was obtained in Proposition 4.8 of \cite{NNRT}. Indeed, using the explicit formulation of this proposition reported in Proposition 3.3 of \cite{BauZha}, it is possible to see that they obtained a uniform bound for the $2k$-th term of the expected signature of
\begin{equation*}
\frac{2^{k}}{H^{k}(2H-1)^{k}\sqrt{(2k)!}}.
\end{equation*}
Observe that
\begin{equation*}
\frac{2^{k}}{H^{k}(2H-1)^{k}\sqrt{(2k)!}}>\frac{2^{k}}{\sqrt{(2k)!}}>\dfrac{1}{k!2^{k}}.
\end{equation*}
The advantage of our result is that the estimate is sharp, it is independent of $H$ and the proof is very short. Moreover, we can also have an equality estimate for each particular word $I$ by just leaving the $\delta_{i_{\sigma(2l)},i_{\sigma(2l-1)}}$ in place. Indeed, in case we have a two dimensional fBm and the word $I$ does not contain just the letter $1$ then the upper bound becomes much smaller than $1/(k!2^{k})$. In particular, we have the following proposition which sharpen the estimate of Theorem \ref{pr3} according to which word $I$ is considered.
\begin{pro}
Let $p,n\in\mathbb{N}$. Consider a $n$-dimensional fBm with $n\geq p$ and a word $I$ which contains $p$ different letters then we have
\begin{equation*}
\mathbb{E}\left(\int_{\Delta^{2k}[s,t]}dB^{I}\right)\leq \dfrac{(t-s)^{2kH}}{k!2^{k}(2k)!}\frac{k!2^{p-1}(2(k-p+1))!}{(k-p+1)!}
\end{equation*}
for $p\leq k$ and zero otherwise.
\end{pro}
\begin{proof}
Observe that if $p>k$ then there is going to be a letter which cannot be paired with a letter of the same value. Hence, the expected iterated integral is zero. \\Regarding the case $p\leq k$, we proceed as follows. First, we pick a number $k$ and analyse the values obtained for different $p$. In particular, consider the case $2k=12$, hence $I=(i_{1},...,i_{12})$. Since what matters are pairs of letters, we can reformulate the word $I=(a_{1},...,a_{6})$ where $a_{j}$ is a pair of two letters with the same value. Now, if $p=k$ then $a_{1}\neq...\neq a_{6}$ and all the combinations becomes
\begin{equation*}
\sum_{\sigma\in\mathcal{G}_{2k}}\prod_{l=1}^{k}\delta_{i_{\sigma(2l)},i_{\sigma(2l-1)}}=6!2^{6},
\end{equation*}
where the $6!$ comes from the possible permutations of the $a_{j}$ in the word $I$ and the $2^{6}$ comes from the fact that inside each $a_{j}$ we have two letters that can assume two possible positions and this is true for $j=1,..,6$. In case $p=5$ then we can reformulate $I:=(a_{1},a_{1},...,a_{6})$ with $a_{1}=a_{2}$ and we obtain
\begin{equation*}
\sum_{\sigma\in\mathcal{G}_{2k}}\prod_{l=1}^{k}\delta_{i_{\sigma(2l)},i_{\sigma(2l-1)}}=\frac{6!}{2}2^{4}4!,
\end{equation*}
where $6!$ comes from again the possible permutations of $a_{j}$ and the $1/2$ comes from the fact that the order $...,a_{1},...,a_{2},...$ is the same as $...,a_{2},...,a_{1},...$. The relevance of the order is indeed considered but at the level of the letters, in the $4!$, as explained in the following line. Further, the $2^{4}$ comes from the two permutations of each $a_{3},...,a_{6}$ and the $4!$ comes from the permutations of the letter inside $(a_{1},a_{2})$. In case $p=4$ there are two possible cases $I=(a_{1},...,a_{6})$ with $a_{1}=a_{2}$ and $a_{3}=a_{4}$ (and the others are all different) or $a_{1}=a_{2}=a_{3}$ (and the others are all different). It is possible to check that the second case give us an higher upper bound so we use this case, for which we have 
\begin{equation*}
\sum_{\sigma\in\mathcal{G}_{2k}}\prod_{l=1}^{k}\delta_{i_{\sigma(2l)},i_{\sigma(2l-1)}}=\frac{6!}{3!}2^{3}6!,
\end{equation*}
where the first $6!$ comes from again the possible permutations of $a_{j}$, further the order of $a_{1},a_{2},a_{3}$ is not relevant now and, since there are $3!$ different order for each permutation of the $6$ elements, we divide by $3!$; in addition, the $2^{3}$ comes from the two permutations of each $a_{4},a_{5},a_{6}$ and the $6!$ comes from the permutations of the letter inside $(a_{1},a_{2},a_{3})$. The same procedure applies to $p=3,2,1$ and for any $k\in\mathbb{N}$.
\end{proof}
\section{The cubature method for the fBm and the proofs of Theorems \ref{prNEW} and \ref{pr4}}\label{cub}
\noindent We start this section with a brief introduction to the cubature method for the fractional Brownian motion. In this section we do \textit{not} use the notation $B_{t}:=B_{t}^{H}$ since we work with the Brownian motion as well.

The cubature method is a numerical method used to obtain approximate solutions to SDEs and parabolic PDEs. The first main step is to obtain the cubature formula, which we define now for the $d$-dimensional fBm. The setting is the same as the one presented for the Bm case in \cite{LV}. So, the probability space we are working on is $(C_{0}^{0}([0,T],\mathbb{R}^{d}),\mathcal{F},\mathbb{P})$ where $C_{0}^{0}([0,T],\mathbb{R}^{d})$ is the space of
$\mathbb{R}^{d}$-valued continuous functions defined in $[0, T ]$ and which starts at zero (i.e. the Wiener space), $\mathcal{F}$ is its Borel $\sigma$-field and $\mathbb{P}$ is the Wiener measure. As in \cite{LV}, let $B_{t}:C_{0}^{0}([0,T],\mathbb{R}^{d})\rightarrow\mathbb{R}^{d}$ such that $B^{i}_{t}(\omega)=\omega^{i}(t)$ for $t\in[0,T]$ and $i=1,..,d$, then $\{B_{t}\}_{t\in[0,T]}=\{B^{1}_{t},...,B^{d}_{t}\}_{t\in[0,T]}$ is a $\mathbb{R}^{d}$-valued Brownian motion on the probability space $(C_{0}^{0}([0,T],\mathbb{R}^{d}),\mathcal{F},\mathbb{P})$.

Now, recall that it is possible to write the fBm in terms of a Bm, as done in \cite{H}. In particular we will follow the formulation in \cite{H} and consider the process $B^{H,i}_{t}(\omega)=\int_{0}^{t}K(t,s)dB^{i}_{s}(\omega)\left(=\int_{0}^{t}K(t,s)d\omega^{i}(s)\right)$ for $t\in[0,T]$ and $i=1,..,d$. Then, the process $\{B^{H}_{t}\}_{t\in[0,T]}=\{B^{H,1}_{t},...,B^{H,d}_{t}\}_{t\in[0,T]}$ is a $\mathbb{R}^{d}$-valued fractional Brownian motion on the probability space $(C_{0}^{0}([0,T],\mathbb{R}^{d}),\mathcal{F},\mathbb{P})$.
\\ As in \cite{LV}, the process $\{\xi_{t,x}\}_{t\in[0,T]}$ is a stochastic process like our $\{B_{t}\}_{t\in[0,T]}$, that is $\xi_{t,x}:C_{0}^{0}([0,T],\mathbb{R}^{d})\rightarrow\mathbb{R}^{N}$ (for $N\in\mathbb{N}$), hence $\xi_{t,x}:\omega\mapsto \xi_{t,x}(\omega)$. Further, we will consider the process $\{\hat{B}^{H}_{t}\}_{t\in[0,T]}=\{t,B^{H,1}_{t},...,B^{H,d}_{t}\}_{t\in[0,T]}$.

First, we give the definition of the cubature formula for the fBm for $H\geq 1/2$ on the Wiener space.
\begin{defn}\label{cubature-definition}
Let $m\in\mathbb{N}$ and $H\geq\frac{1}{2}$. Let define $\mathcal{A}_{m}:=\{(i_{1},...,i_{k})\in\{0,...,d\}^{k},2Hk+(2-2H)\times\mathnormal{card}\{l,i_{l}=0\}\leq m \,\,\,\text{and}\,\,\, k\in\mathbb{N} \}$. We say that the paths
\begin{equation*}
\bar{\omega}_{1} , . . . , \bar{\omega}_{n}\in C_{0,bv}^{0}
([0, T ], \mathbb{R}^{d})
\end{equation*}
and the positive weights $\lambda_{1} , . . . ,\lambda_{n}$ define a cubature formula of degree $m$ at time $T$, if and only if, for all $(i_{1}, . . . , i_{k})\in \mathcal{A}_{m}$,
\begin{equation*}
\mathbb{E}\left[\int_{0<t_{1}<,...,<t_{k}< T} d\hat{B}_{t_{1}}^{H,i_{1}}\cdot\cdot\cdot d\hat{B}_{t_{k}}^{H,i_{k}}\right]=\sum_{j=1}^{n}\lambda_{j}\int_{0<t_{1}<,...,<t_{k}< T}d\hat{\omega}_{j}^{i_{1}}(t_{1})\cdot\cdot\cdot d\hat{\omega}_{j}^{i_{k}}(t_{k}).
\end{equation*}
where, for $l=1,...,k$,
\begin{equation*}
\hat{B}_{t_{l}}^{H,i_{l}}:=
\begin{cases}
t_{l}, \qquad if \quad i_{l}=0,\\
B^{i_{l}H}_{t_{l}} \qquad if \quad i_{l}\neq 0,
\end{cases}\quad \text{and}\quad \hat{\omega}_{t_{l}}^{i_{l}}:=
\begin{cases}
t_{l}, \qquad if \quad i_{l}=0,\\
\bar{\omega}^{i_{l}}_{t_{l}} \qquad if \quad i_{l}\neq 0.
\end{cases}
\end{equation*}
\end{defn}
\noindent Notice that the formulation of $\mathcal{A}_{m}$ comes from the fact that for a word $I$
\begin{equation*}
\mathbb{E}\left(\int_{\Delta^{k}[s,t]}d\hat{B}^{H,I}\right)=\mathbb{E}\left(\int_{\Delta^{k}[0,t-s]}d\hat{B}^{H,I}\right)=(t-s)^{kH+(1-H)\times\mathnormal{card}\{j,i_{j}=0\}}\mathbb{E}\left(\int_{\Delta^{k}[0,1]}d\hat{B}^{H,I}\right)
\end{equation*}
(see \cite{NNRT} Proposition 4.8 for similar discussions) and the reason why we multiply by 2 in $\mathcal{A}_{m}$ is to be consistent with the formulation provided by \cite{LV} for the Bm case. 

Once the cubature formula is obtained, it is possible to derive approximate solutions of SDEs driven by fBm. Indeed, let us recall that $C_{b}^{\infty}(\mathbb{R}^{N},\mathbb{R}^{N})$ be the space of $\mathbb{R}^{N}$-valued smooth functions defined in $\mathbb{R}^{N}$ whose derivatives of any order are bounded. We regard elements of $C_{b}^{\infty}(\mathbb{R}^{N},\mathbb{R}^{N})$ as vector fields on $\mathbb{R}^{N}$. Let $V_{0},...,V_{d}$ be such vector fields. Let $\xi_{t,x}$, $t\in[0, T ], x\in \mathbb{R}^{N}$, be the solution of the SDE
\begin{equation}\label{SDE}
\text{d}\xi_{t,x}=\sum_{i=0}^{d}V_{i}(\xi_{t,x})\text{d}\hat{B}^{H,i}_{t},\quad\text{with}\quad\xi_{0,x}=x.
\end{equation}
Moreover, let  $\Phi_{T,x}(\omega_{j}^{*})$, where $\omega_{j}^{*}\in C_{0}^{0}([0,T],\mathbb{R}^{d})$, be the solution at time T of the ODE 
\begin{equation}\label{ODE}
\text{d}y_{t,x}=\sum_{i=0}^{d}V_{i}(y_{t,x})\text{d}\hat{\omega}^{i}_{j}(t),\quad\text{with}\quad y_{0,x}=x.
\end{equation}

The core message of the cubature method is that the weighted sum of the solutions of the ODEs $(\ref{ODE})$ for $j=1,...,n$, where the weights are given by $\lambda_{1},...,\lambda_{n}$ of the cubature formula, approximates the expected solution of the SDE $(\ref{SDE})$. In particular, we have the Theorem \ref{prNEW}, which extends Proposition 2.1, Lemma 3.1 and Proposition 3.2, Proposition of \cite{LV} to the fBm case for $H>1/2$. Before proving the theorem let us recall some notation. Let $\|\cdot\|_{\mathbb{R}^{N}}$ denote the Euclidean norm in $\mathbb{R}^{N}$ and let $|I|$ indicate the length of the word $I$, where $I\in\{0,...,d\}^{k}$ for some $k\in\mathbb{N}$. We will also use the supremum norm $\|\cdot\|_{\infty}$, \textit{e.g.} $\|V_{i_{1}}\cdots V_{i_{k}}f\|_{\infty}=\sup_{x\in\mathbb{R}^{N}}\frac{|(V_{i_{1}}\cdots V_{i_{k}}f)(x)|}{\|x\|_{\mathbb{R}^{N}}}$. 
\begin{proof}[Proof (Theorem \ref{prNEW})]
First of all, notice that we have the following stochastic Taylor expansion (see Proposition 2.1 of \cite{LV} together with \cite{BauZha}). For any $f$ which is smooth (\textit{i.e.} infinitely differentiable) and whose derivatives of any order are bounded, we have
\begin{equation*}
f(\xi_{T,x})=\sum_{(i_{1},...,i_{k})\in\mathcal{A}_{m}}V_{i_{1}}\cdots V_{i_{k}}f(x)\int_{0<t_{1}<,...,<t_{k}< T} d\hat{B}_{t_{1}}^{H,i_{1}}\cdot\cdot\cdot d\hat{B}_{t_{k}}^{H,i_{k}}+R_{m}(T,x,f),\quad\text{where}
\end{equation*}
\begin{equation*}
R_{m}(T,x,f)=\sum_{(i_{1},...,i_{k})\in\mathcal{A}_{m}, (i_{0},...,i_{k})\notin\mathcal{A}_{m}}\int_{0<t_{0}<,...,<t_{k}< T}V_{i_{0}}\cdots V_{i_{k}}f(\xi_{t_{0},x}) d\hat{B}_{t_{0}}^{H,i_{0}}\cdot\cdot\cdot d\hat{B}_{t_{k}}^{H,i_{k}}.
\end{equation*}
Assume that the paths $\bar{\omega}_{1}, . . . ,\bar{\omega}_{n}\in C_{0,bv}^{0}([0,T],\mathbb{R}^{d})$ and the weights $\lambda_{1}, . . . ,\lambda_{n}$ define a cubature formula for the fBm of degree $m$ for time T.
Now notice that by the cubature formula, we have 
\begin{equation}\label{zzz}
\mathbb{E}\left[\int_{0<t_{1}<,...,<t_{k}< T} d\hat{B}_{t_{1}}^{H,i_{1}}\cdot\cdot\cdot d\hat{B}_{t_{k}}^{H,i_{k}}\right]=\sum_{j=1}^{n}\lambda_{j}\int_{0<t_{1}<,...,<t_{k}< T}d\hat{\omega}_{j}^{i_{1}}(t_{1})\cdot\cdot\cdot d\hat{\omega}_{j}^{i_{k}}(t_{k}),
\end{equation}
where $\hat{\omega}(t)=(t,\bar{\omega}(t))$. Observe that from our setting, we have
\begin{equation*}
\mathbb{E}\left[\int_{0<t_{1}<,...,<t_{k}< T} d\hat{B}_{t_{1}}^{H,i_{1}}\cdot\cdot\cdot d\hat{B}_{t_{k}}^{H,i_{k}}\right]=\int_{C_{0}^{0}([0,T],\mathbb{R}^{d})}\int_{0<t_{1}<,...,<t_{k}< T}d\hat{B}_{t_{1}}^{H,i_{1}}(\omega)\cdot\cdot\cdot d\hat{B}_{t_{k}}^{H,i_{k}}(\omega)\mathbb{P}(d\omega),
\end{equation*}
so eq.~$(\ref{zzz})$ can be rewritten as
\begin{equation*}
\int_{C_{0}^{0}([0,T],\mathbb{R}^{d})}\int_{0<t_{1}<,...,<t_{k}< T}d\hat{B}_{t_{1}}^{H,i_{1}}(\omega)\cdot\cdot\cdot d\hat{B}_{t_{k}}^{H,i_{k}}(\omega)\mathbb{P}(d\omega)=\sum_{j=1}^{n}\lambda_{j}\int_{0<t_{1}<,...,<t_{k}< T}d\hat{\omega}_{j}^{i_{1}}(t_{1})\cdot\cdot\cdot d\hat{\omega}_{j}^{i_{k}}(t_{k}).
\end{equation*}

From this formula we can obtain the probability measure $\mathbb{Q}_{T}=\sum_{j=1}^{n}\lambda_{j}\delta_{\omega_{j}^{\star}}$ on $(C_{0}^{0}([0,T],\mathbb{R}^{d}),\mathcal{F})$, namely on the Wiener space with its Borel $\sigma$-field. In particular, let
\begin{equation*}
Z(\omega):=\int_{0<t_{1}<,...,<t_{k}< T}d\hat{B}_{t_{1}}^{H,i_{1}}(\omega)\cdot\cdot\cdot d\hat{B}_{t_{k}}^{H,i_{k}}(\omega),
\end{equation*}
and so
\begin{equation*}
\int_{C_{0}^{0}([0,T],\mathbb{R}^{d})}\int_{0<t_{1}<,...,<t_{k}< T}d\hat{B}_{t_{1}}^{H,i_{1}}(\omega)\cdot\cdot\cdot d\hat{B}_{t_{k}}^{H,i_{k}}(\omega)\mathbb{P}(d\omega)=\int_{C_{0}^{0}([0,T],\mathbb{R}^{d})}Z(\omega)\mathbb{P}(d\omega),
\end{equation*}
and by applying the measure $\mathbb{Q}_{T}$ we have
\begin{equation*}
\mathbb{E}_{\mathbb{Q}_{T}}\left[Z\right]=\sum_{j=1}^{n}\lambda_{j}\int_{C_{0}^{0}([0,T],\mathbb{R}^{d})}Z(\omega)\delta_{\omega_{j}^{\star}}d\omega=\sum_{j=1}^{n}\lambda_{j}Z(\omega_{j}^{\star})
\end{equation*}
\begin{equation*}
=\sum_{j=1}^{n}\lambda_{j}\int_{0<t_{1}<,...,<t_{k}< T}d\hat{B}_{t_{1}}^{H,i_{1}}(\omega_{j}^{\star})\cdot\cdot\cdot d\hat{B}_{t_{k}}^{H,i_{k}}(\omega_{j}^{\star}),
\end{equation*}
where from the definition of the fBm we have, for $i_{l}\neq0$, $\hat{B}_{t_{l}}^{H,i_{l}}(\omega_{j}^{\star})=\int_{0}^{t_{l}}K(t_{l},s)d\omega^{\star,i_{l}}_{j}(s)$, while, for $i_{l}=0$, $\hat{B}_{t_{l}}^{H,i_{l}}(\omega_{j}^{\star})=t_{l}$.\\It is possible to observe that we have not mentioned how we select these $\omega^{\star}_{j}\in C_{0}^{0}([0,T],\mathbb{R}^{d})$. We do it now. Choose $\omega^{\star}_{j}\in C_{0}^{0}([0,T],\mathbb{R}^{d})$ such that $\hat{\omega}^{i_{l}}_{j}(t_{l})=\int_{0}^{t_{l}}K(t_{l},s)d\omega^{\star,i_{l}}_{j}(s)$ for $i_{l}\neq 0$. We know that such $\omega^{\star}_{j}$ exists since $\bar{\omega}\in C_{\text{bv},0}^{0}([0,T],\mathbb{R}^{d})$, where we had $\hat{\omega}(t)=(t,\bar{\omega}(t))$. Thus, $\mathbb{Q}_{T}=\sum_{j=1}^{n}\lambda_{j}\delta_{\omega_{j}^{\star}}$ is a probability measure on $(C_{0}^{0}([0,T],\mathbb{R}^{d}),\mathcal{F})$.
\\ Hence, we have $\hat{B}_{t_{l}}^{H,i_{l}}(\omega_{j}^{\star})=\hat{\omega}^{i_{l}}_{j}(t_{l})$ and, consequently,
\begin{equation*}
\sum_{j=1}^{n}\lambda_{j}\int_{0<t_{1}<,...,<t_{k}< T}d\hat{B}_{t_{1}}^{H,i_{1}}(\omega_{j}^{\star})\cdot\cdot\cdot d\hat{B}_{t_{k}}^{H,i_{k}}(\omega_{j}^{\star})=\sum_{j=1}^{n}\lambda_{j}\int_{0<t_{1}<,...,<t_{k}< T}d\hat{\omega}_{j}^{i_{1}}(t_{1})\cdot\cdot\cdot d\hat{\omega}_{j}^{i_{k}}(t_{k}),
\end{equation*}
that is
\begin{equation*}
\mathbb{E}_{\mathbb{Q}_{T}}\left[\int_{0<t_{1}<,...,<t_{k}< T} d\hat{B}_{t_{1}}^{H,i_{1}}\cdot\cdot\cdot d\hat{B}_{t_{k}}^{H,i_{k}}\right]=\sum_{j=1}^{n}\lambda_{j}\int_{0<t_{1}<,...,<t_{k}< T}d\hat{\omega}_{j}^{i_{1}}(t_{1})\cdot\cdot\cdot d\hat{\omega}_{j}^{i_{k}}(t_{k}).
\end{equation*}
Therefore, from eq. $(\ref{zzz})$ we obtain the following formulation:
\begin{equation}\label{equa}
\mathbb{E}\left[\int_{0<t_{1}<,...,<t_{k}< T} d\hat{B}_{t_{1}}^{H,i_{1}}\cdot\cdot\cdot d\hat{B}_{t_{k}}^{H,i_{k}}\right]=\mathbb{E}_{\mathbb{Q}_{T}}\left[\int_{0<t_{1}<,...,<t_{k}< T} d\hat{B}_{t_{1}}^{H,i_{1}}\cdot\cdot\cdot d\hat{B}_{t_{k}}^{H,i_{k}}\right].
\end{equation}
It is very important to notice that while $B_{t_{l}}^{H,i_{l}}(\omega)$ is a fBm under the Wiener measure, it is not any more under $\mathbb{Q}_{T}$.

Observe that equality $(\ref{equa})$ holds only for all $(i_{1}, . . . ,i_{k})\in\mathcal{A}_{m}$ by definition of the cubature formula. Moreover, notice that
\begin{equation*}
\mathbb{E}_{\mathbb{Q}_{T}}[f(\xi_{T,x})]=\sum_{j=1}^{n}\lambda_{j}\int_{C_{0}^{0}([0,T],\mathbb{R}^{d})}f(\xi_{T,x}(\omega_{j}))\delta_{\omega_{j}^{*}}d\omega=\sum_{j=1}^{n}\lambda_{j}f(\xi_{T,x}(\omega_{j}^{*}))
\end{equation*}
and that $\xi_{T,x}(\omega_{j}^{*})$ is the solution of 
\begin{equation*}
\text{d}\xi_{t,x}(\omega_{j}^{*})=\sum_{i=0}^{d}V_{i}(\xi_{t,x}(\omega_{j}^{*}))\text{d}\hat{B}^{H,i}_{t}(\omega_{j}^{*}),\quad\text{with}\quad\xi_{0,x}(\omega_{j}^{*})=x,
\end{equation*}
or, equivalently,  the solution of 
\begin{equation*}
\text{d}y_{t,x}=\sum_{i=0}^{d}V_{i}(y_{t,x})\text{d}\hat{\omega}^{i}_{j}(t),\quad\text{with}\quad y_{0,x}=x,
\end{equation*}
which we defined before to be $\Phi_{t,x}(\omega_{j}^{*})$. Therefore,
\begin{equation*}
\sum_{j=1}^{n}\lambda_{j}f(\Phi_{T,x}(\omega_{j}^{*}))=\mathbb{E}_{\mathbb{Q}_{T}}[f(\xi_{T,x})].
\end{equation*}

Let us now adapt the stochastic Taylor formula to the probability measure $\mathbb{Q}_{T}$. In particular, using the scaling property of the fBm, which is inherited by the respective cubature paths we have that
\begin{equation*}
|\mathbb{E}_{\mathbb{Q}_{T}}[R_{m}(T,x,f)]|=\Bigg|\mathbb{E}_{\mathbb{Q}_{T}}\left[\sum_{(i_{1},...,i_{k})\in\mathcal{A}_{m}, (i_{0},...,i_{k})\notin\mathcal{A}_{m}}\int_{0<t_{0}<,...,<t_{k}< T}V_{i_{0}}\cdots V_{i_{k}}f(\xi_{t_{0},x}) d\hat{B}_{t_{0}}^{H,i_{0}}\cdots d\hat{B}_{H,t_{k}}^{i_{k}}\right]\Bigg|
\end{equation*}
\begin{equation}\label{ineq}
\leq \sum_{j=1}^{n}\lambda_{j}\sum_{(i_{1},...,i_{k})\in\mathcal{A}_{m}, (i_{0},...,i_{k})\notin\mathcal{A}_{m}}\bigg|\int_{0<t_{0}<,...,<t_{k}< T}V_{i_{0}}\cdots V_{i_{k}}f(\xi_{t_{0},x}(\omega_{j}^{\star})) d\hat{\omega}_{j}^{i_{1}}(t_{1})\cdot\cdot\cdot d\hat{\omega}_{j}^{i_{k}}(t_{k})\bigg|
\end{equation}
\begin{equation*}
\leq \sum_{i=j}^{n}\lambda_{j}\sum_{(i_{1},...,i_{k})\in\mathcal{A}_{m}, (i_{0},...,i_{k})\notin\mathcal{A}_{m}}\|V_{i_{0}}\cdots V_{i_{k}}f\|_{\infty}
\end{equation*}
\begin{equation*}
T^{1+kH+(1-H)\times\mathnormal{card}\{l,i_{l}=0\}}\int_{0<t_{0}<,...,<t_{k}< 1}\big|d\hat{\omega}_{j}^{i_{0}}(t_{0})\big|\cdots \big|d\hat{\omega}_{j}^{i_{k}}(t_{k})\big|
\end{equation*}
\begin{equation*}
\leq C' T^{(m+2)/2} \sup_{(i_{1},...,i_{k})\in\mathcal{A}_{m}, (i_{0},...,i_{k})\notin\mathcal{A}_{m}}\|V_{i_{0}}\cdots V_{i_{k}}f\|_{\infty},\quad\text{if $T\geq 1$, and}
\end{equation*}
\begin{equation*}
(\ref{ineq})\leq \sum_{i=j}^{n}\lambda_{j}\sum_{(i_{1},...,i_{k})\in\mathcal{A}_{m}, (i_{0},...,i_{k})\notin\mathcal{A}_{m}}\|V_{i_{0}}\cdots V_{i_{k}}f\|_{\infty}
\end{equation*}
\begin{equation*}
T^{H+kH+(1-H)\times\mathnormal{card}\{l,i_{l}=0\}}\int_{0<t_{0}<,...,<t_{k}< 1}\big|d\hat{\omega}_{j}^{i_{0}}(t_{0})\big|\cdots \big|d\hat{\omega}_{j}^{i_{k}}(t_{k})\big|
\end{equation*}
\begin{equation*}
\leq C'' T^{2H} \sup_{(i_{1},...,i_{k})\in\mathcal{A}_{m}, (i_{0},...,i_{k})\notin\mathcal{A}_{m}}\|V_{i_{0}}\cdots V_{i_{k}}f\|_{\infty},\quad\text{if $T<1$,}
\end{equation*}
where $C',C''>0$  and does not depend on $T$. Notice that the $1$ in $1+kH+(1-H)\times\mathnormal{card}\{l,i_{l}=0\}$ and the $H$ in $H+kH+(1-H)\times\mathnormal{card}\{l,i_{l}=0\}$ comes from the fact that we might have $i_{0}=0$ (hence we will have the scaling $T$) or $i_{0}=1$ (hence we will have the scaling $T^{H}$), since $T\geq T^{H}$ for $T\geq1$ we take $T$ to get the upper bound for $T\geq1$ and $T^{H}$ otherwise for $T<1$. Moreover, for the last inequality in both cases we used the fact that for any word in $\mathcal{A}_{m}$ we have by its definition that $H\leq kH+(1-H)\times\mathnormal{card}\{l,i_{l}=0\}\leq m/2$.

Regarding $|\mathbb{E}[R_{m}(T,x,f)]|$, we are not able to get a similar bound. Indeed, \cite{BauZha} is partially devoted to the study of this quantity. More precisely, in \cite{BauZha} a different form of the remainder is considered, but the two can be linked together. In particular, from Theorem 3.4 in \cite{BauZha} we have
\begin{equation*}
|\mathbb{E}[R_{m}(T,x,f)]|\leq C_{\gamma}\frac{(dMKT)^{(m+2)/2}}{(((m+2)/2)!)^{1/2-\gamma}}\sum_{k=0}^{\infty}\frac{(dMKT)^{k}}{(k!)^{1/2-\gamma}},\quad\text{if $T\geq 1$ and}
\end{equation*}
\begin{equation*}
|\mathbb{E}[R_{m}(T,x,f)]|\leq C_{\gamma}\frac{(dMKT^{H})^{(m+2)/2}}{(((m+2)/2)!)^{1/2-\gamma}}\sum_{k=0}^{\infty}\frac{(dMKT^{H})^{k}}{(k!)^{1/2-\gamma}},\quad\text{if $T< 1$},
\end{equation*}
where $M$ and $K$ are defined in the statement of the theorem. We used $(m+2)/2$ for the ``$N+1$" in the statement of Theorem 3.4 in \cite{BauZha} because it is the minimum number of iterated integrals of $\mathcal{A}_{m+2}$. Further, we used $T$ instead of $T^{H}$ because we want a bound which is independent of the exact composition of the word $I$:
\begin{equation*}
\mathbb{E}\left(\bigg|\int_{\Delta^{k}[0,T]}d\hat{B}^{H,I}\bigg|\right)^{1/2}\leq T^{kH+(1-H)\times\mathnormal{card}\{j,i_{j}=0\}}\frac{K^{k}}{\sqrt{k!}}\leq \begin{cases}
T^{k}\dfrac{K^{k}}{\sqrt{k!}}\quad\text{if $T\geq 1$},\\T^{kH}\dfrac{K^{k}}{\sqrt{k!}}\quad\text{if $T< 1$}.
\end{cases} 
\end{equation*}
Notice that by doing this we are able to consider the drift, which is not considered explicitly in Theorem  3.4 in \cite{BauZha}. Moreover, we have not mentioned that the vector fields $V_{i}$s are analytic since being infinitely differentiable and having bounded derivatives imply that they are global analytic.

 Finally, observe that
\begin{equation*}
|\mathbb{E}[f(\xi_{T,x})]-\mathbb{E}_{\mathbb{Q}_{T}}[f(\xi_{T,x})]|\leq |\mathbb{E}[R_{m}(T,x,f)]|+|\mathbb{E}_{\mathbb{Q}_{T}}[R_{m}(T,x,f)]|
\end{equation*}
\begin{equation*}
+ \Bigg|(\mathbb{E}-\mathbb{E}_{\mathbb{Q}_{T}})\left[\sum_{(i_{1},...,i_{k})\in\mathcal{A}_{m}}V_{i_{1}}\cdots V_{i_{k}}f(x)\int_{0<t_{1}<,...,<t_{k}< T} d\hat{B}_{t_{1}}^{H,i_{1}}\cdot\cdot\cdot d\hat{B}_{t_{k}}^{H,i_{k}} \right]\Bigg|.
\end{equation*}
For the first two terms we have proved their upper bounds, while the last term is zero by definition of $\mathbb{Q}_{T}$, namely by the cubature formula. Then, we have
\begin{equation*}
\sup_{x\in\mathbb{R}^{N}}\Big|\mathbb{E}\left(f(\xi_{T,x})\right) -\sum_{j=1}^{n}\lambda_{j}f(\Phi_{T,x}(\omega_{j}))\Big|\leq\begin{cases}
Z_{1}(T)+Z_{3}(T)\quad\text{if $T\geq 1$},\\ Z_{2}(T)+Z_{4}(T)\quad\text{if $T< 1$}.
\end{cases} 
\end{equation*}
where
\begin{equation*}
Z_{1}(T):=C' T^{(m+2)/2}\sup_{(i_{1},...,i_{k})\in\mathcal{A}_{m}, (i_{0},...,i_{k})\notin\mathcal{A}_{m}}\|V_{i_{1}}\cdots V_{i_{k}}f\|_{\infty},
\end{equation*}
\begin{equation*}
Z_{2}(T):=C'' T^{2H}\sup_{(i_{1},...,i_{k})\in\mathcal{A}_{m}, (i_{0},...,i_{k})\notin\mathcal{A}_{m}}\|V_{i_{1}}\cdots V_{i_{k}}f\|_{\infty},
\end{equation*}
\begin{equation*}
Z_{3}(T):=C_{\gamma}T^{(m+2)/2}\frac{(dK)^{(m+2)/2}}{(((m+2)/2)!)^{1/2-\gamma}}\sup\limits_{x\in\mathbb{R}^{N}}M_{x}^{(m+2)/2}\sum_{k=0}^{\infty}\frac{(dM_{x}KT)^{k}}{(k!)^{1/2-\gamma}},
\end{equation*}
\begin{equation*}
Z_{4}(T):=C_{\gamma}T^{H(m+2)/2}\frac{(dK)^{(m+2)/2}}{(((m+2)/2)!)^{1/2-\gamma}}\sup\limits_{x\in\mathbb{R}^{N}}M_{x}^{(m+2)/2}\sum_{k=0}^{\infty}\frac{(dM_{x}KT^{H})^{k}}{(k!)^{1/2-\gamma}},
\end{equation*}
where $C',C'',C_{\gamma}>0$ are constants independent of T.
\end{proof}
\begin{rem}
Notice that our result differs from Proposition 3.2 of \cite{LV} for two reasons. First, in that proposition lack the consideration of the case $T<1$. Second, we do not have the possibility to use the It\^{o} isometry of Proposition 2.1 in \cite{LV}, but we have to rely on Theorem 3.4 in \cite{BauZha}. We also point out that in case sharper estimates for the remainder of the stochastic Taylor expansion for the fBm are found, these will directly improve our result.
\end{rem}
From the scaling properties of the fBm, we obtain the following simple proposition, which is an equivalent of Proposition 2.5 in \cite{LV} for the fBm.
\begin{pro}Assume that $\bar{\omega}_{1}, . . . ,\bar{\omega}_{n}\in C_{0,bv}^{0}([0,1],\mathbb{R}^{d})$ and the weights $\lambda_{1}, . . . ,\lambda_{n}$ define a cubature formula for the fBm on the Wiener space of degree $m$ for time 1. Define, for $j = 1,...,n$, the paths $\bar{\omega}^{(T)}_{j}\in C_{0,bv}^{0}([0,T],\mathbb{R}^{d})$ by $\bar{\omega}^{(T),i}_{j}(t)=T^{H} \bar{\omega}_{j}^{i}(t/T)$ for $i=1,...,d$. The paths $\bar{\omega}^{(T)}_{j}$ and the weights $\lambda_{j}$ , $j = 1,...,n$, then define a cubature formula for the fBm on the Wiener space of degree $m$ at time $T$.
\end{pro}
\begin{proof}
	It follows from the scaling property of the fractional Brownian motion.
\end{proof}

The last step of the cubature method is to extend it from small times to any times. This step is usually called concatenation step. Indeed, for the Bm case it is sufficient to divide the time interval in smaller subintervals and then concatenate the solutions obtained in each subinterval (see \cite{LV} for details).\\
In this article, we do not deal with the concatenation step. This is because for the fBm we do not have the independence of the increment property (\textit{i.e.} the Markov semigroup property) that we have for the Brownian motion case. Hence, the problem of the concatenation is an open and non-trivial problem. In Theorem \ref{pr4} we focused on the cubature formula for the one dimensional fBm up to degree 5. Due to the fact that the proof is tedious and mainly based on linear algebra computations we decided to only sketch it here.
\begin{proof}[Sketch of the proof (Theorem \ref{pr4})]
The first step is to compute the expected iterated integrals for different combinations of words, that is for different combinations of time path $t$ and fBm $B_{t}$. For example, for the degree $2H+2$ we need to compute:
\begin{equation*}
\mathbb{E}\left(\int_{0}^{1}\int_{0}^{u_{2}}dB_{u_{1}}du_{2}\right)=\int_{0}^{1}\mathbb{E}\left(B_{u_{2}}\right)du_{2}=0
\end{equation*}
and
\begin{equation*}
\mathbb{E}\left(\int_{0}^{1}\int_{0}^{u_{2}}du_{1}dB_{u_{2}}\right)=\mathbb{E}\left(\int_{0}^{1}u_{2}dB_{u_{2}}\right)=0.
\end{equation*}
The second step is to obtain the respective iterated integrals of the $\omega_{i}$ weighted by the $\lambda_{i}$. Hence, for degree $2H+2$ we have:
\begin{equation*}
0=\sum_{i=1}^{n}\lambda_{i}\int_{0}^{1}\int_{0}^{u_{2}}d\omega_{i,u_{1}}du_{2}\Rightarrow \sum_{i=1}^{n}\lambda_{i}\int_{0}^{1}\omega_{i,u_{2}}du_{2}=0
\end{equation*}
and
\begin{equation*}
0=\sum_{i=1}^{n}\lambda_{i}\int_{0}^{1}\int_{0}^{u_{2}}du_{1}d\omega_{i,u_{2}}\Rightarrow \sum_{i=1}^{n}\lambda_{i}\int_{0}^{1}u_{2}d\omega_{i,u_{2}}=0.
\end{equation*}
Then we have a system of equations where the unknowns are the $\omega_{i}$, the $\lambda_{i}$ and $n$. Solving this system give us our result.
\end{proof}
We conclude with a final remark on the solution(s) obtained for the cubature formula.
\begin{rem}
From the proof of this theorem we obtain two solutions of our system of equations, which determine the slope of our paths $\omega$s. The reason why we focused only on one solution is because Lyons and Victoir focused on that solution in their paper. They used MATHEMATICA to produce a solution, without explicitly justify their decision. However, both solutions are feasible. Hence, we have two valid solutions for the cubature formula of the fractional Brownian motion for this kind of structure (\textit{i.e.} piecewise linear paths with change of slopes at $t=\frac{1}{3}$ and $t=\frac{2}{3}$).
\\
The reason why they did not justify their decision is probably due to the fact that the structure adopted is already arbitrary and hence it is important to have a solution and not to have a particular or unique solution.
\end{rem}
\section{Acknowledgements}
The author would like to thank Horatio Boedihardjo for his assistance, his comments and constructive discussions through out the writing of this work. Further, the author would like to thank Dan Crisan and Thomas Cass for useful remarks and discussions, and Tobias Kuna for an important observation regarding Theorem \ref{pr3}. Finally, the author would like to thank the CDT in MPE for providing funding for this research.
\section{Appendix 0: Discussion of Proposition 5.1}
To clarify the result of Proposition 5.1 and its proof we explain as follows.
\\The key point is to decompose the permutation $\sigma$ of $\mathcal{G}_{2k}$ into a permutation $\tau\in\mathcal{G}_{2k-2}$ and a 2-permutation of $2k$. In particular, the number of permutations of the set $(t_{1},...,t_{2k})$ is equal to the number of permutations of the set of $2k-2$ elements times the $2$-permutations of $2k$ (where $k$-permutations of $n$ are the different ordered arrangements of a $k$-element subset of an $n$-set), that is
\begin{equation*}
(2k)!=(2k-2)!\cdot\frac{(2k)!}{(2k-2)!}
\end{equation*}
Now, consider the case $t_{\sigma(1)}=t_{4}$ and $t_{\sigma(1)}=t_{9}$ without any loss of generality. It is possible to see that by a simple change of variables (or better change of notation)
\begin{equation}
\int_{0<t_{1}<t_{2}<t_{3}<t_{5}<...<t_{8}<t_{10}<...<1}\prod_{l=2}^{k}\delta_{i_{\sigma(2l)},i_{\sigma(2l-1)}}|t_{\sigma(2l)}-t_{\sigma(2l-1)}|^{2H-2}\text{d}t_{\sigma(2l)}\text{d}t_{\sigma(2l-1)}
\end{equation}
\begin{equation}\label{integral}
= \int_{\Delta^{2k-2}[0,1]}\prod_{l=1}^{k-1}\delta_{i_{\tau(2l)},i_{\tau(2l-1)}}|s_{\tau(2l)}-s_{\tau(2l-1)}|^{2H-2}\text{d}s_{1}\cdots \text{d}s_{2k-2}
\end{equation}
where $\tau$ is a permutation of the set $(1,...,2k-2)$. The main problem can take place on the dependence of $\tau$ on $\sigma$. This is true. However, notice that we would get the same integral if the permutation of the $2k-2$ elements of the set $\{1,...,2k\}\setminus \{\sigma(1),\sigma(2)\}$ was the same. That is we would get same $\tau$ for different $\sigma$s. We will explain the argument in details in the next pages. We start with an example.
\\ \textbf{Example 1.} Assume that we have that our permutation $\sigma$ is just the identity, that is $\sigma(1)=1,\sigma(2)=2,...,\sigma(2k)=2k$, then we would have
\begin{equation}
\int_{0<t_{3}<...<t_{2k}<1}\prod_{l=2}^{k}\delta_{i_{\sigma(2l)},i_{\sigma(2l-1)}}|t_{\sigma(2l)}-t_{\sigma(2l-1)}|^{2H-2}\text{d}t_{\sigma(2l)}\text{d}t_{\sigma(2l-1)}
\end{equation}
\begin{equation}
=\int_{0<t_{3}<...<t_{2k}<1}\prod_{l=2}^{k}\delta_{i_{2l},i_{2l-1}}|t_{2l}-t_{2l-1}|^{2H-2}\text{d}t_{2l}\text{d}t_{2l-1}
\end{equation}
\begin{equation}
=\int_{\Delta^{2k-2}[0,1]}\prod_{l=1}^{k-1}\delta_{i_{2l},i_{2l-1}}|t_{2l}-t_{2l-1}|^{2H-2}\text{d}t_{2l}\text{d}t_{2l-1}.
\end{equation}
However, we can get the same integral if we have that $\sigma(1)=2k-1,\sigma(2)=2k$ and the rest is $\sigma(3)=1,\sigma(4)=2,...,\sigma(2k)=2k-2$. Indeed, in this case we would have
\begin{equation}
\int_{0<t_{1}<...<t_{2k-2}<1}\prod_{l=2}^{k}\delta_{i_{\sigma(2l)},i_{\sigma(2l-1)}}|t_{\sigma(2l)}-t_{\sigma(2l-1)}|^{2H-2}\text{d}t_{\sigma(2l)}\text{d}t_{\sigma(2l-1)}
\end{equation}
\begin{equation}
=\int_{0<t_{1}<...<t_{2k-2}<1}\prod_{l=1}^{k-1}\delta_{i_{2l},i_{2l-1}}|t_{2l}-t_{2l-1}|^{2H-2}\text{d}t_{2l}\text{d}t_{2l-1}
\end{equation}
\begin{equation}
=\int_{\Delta^{2k-2}[0,1]}\prod_{l=1}^{k-1}\delta_{i_{2l},i_{2l-1}}|t_{2l}-t_{2l-1}|^{2H-2}\text{d}t_{2l}\text{d}t_{2l-1}.
\end{equation}
Continuing with this argument, it is possible to observe that we get the same integral $(\ref{integral})$ if $\sigma(1)=4$ and $\sigma(2)=9$, while the permutation of the other $2k-2$ terms remains the same, which in this case is $\sigma(3)=1,\sigma(4)=2,\sigma(5)=3,\sigma(6)=5,\sigma(7)=6,\sigma(8)=7,\sigma(9)=8,\sigma(10)=10,\sigma(11)=11,...,\sigma(2k)=2k$. Then we have
\begin{equation}
\int_{0<t_{1}<t_{2}<t_{3}<t_{5}<...<t_{8}<t_{10}<...<1}\prod_{l=2}^{k}\delta_{i_{\sigma(2l)},i_{\sigma(2l-1)}}|t_{\sigma(2l)}-t_{\sigma(2l-1)}|^{2H-2}\text{d}t_{\sigma(2l)}\text{d}t_{\sigma(2l-1)}
\end{equation}
\begin{equation*}
=\int_{0<t_{1}<t_{2}<t_{3}<t_{5}<...<t_{8}<t_{10}<...<1}\delta_{i_{2},i_{1}}|t_{2}-t_{1}|^{2H-2}\delta_{i_{5},i_{3}}|t_{5}-t_{3}|^{2H-2}\delta_{i_{7},i_{6}}|t_{7}-t_{6}|^{2H-2}\delta_{i_{10},i_{8}}|t_{10}-t_{8}|^{2H-2}
\end{equation*}
\begin{equation}
\prod_{l=6}^{k}\delta_{i_{\sigma(2l)},i_{\sigma(2l-1)}}|t_{\sigma(2l)}-t_{\sigma(2l-1)}|^{2H-2}\text{d}t_{1}\text{d}t_{2}\text{d}t_{3}\text{d}t_{5}\cdots\text{d}t_{8}\text{d}t_{10}\cdots\text{d}t_{2k}
\end{equation}
and with a simple change of notation
\begin{equation}
=\int_{\Delta^{2k-2}[0,1]}\prod_{l=1}^{k-1}\delta_{i_{2l},i_{2l-1}}|t_{2l}-t_{2l-1}|^{2H-2}\text{d}t_{2l}\text{d}t_{2l-1}.
\end{equation}
From this example it is possible to see that if the permutation of the remaining $2k-2$ elements (\textit{i.e.} elements from the set $\{1,...,2k\}\setminus \{\sigma(1),\sigma(2)\}$) is the same it is not important which value the points $\sigma(1),\sigma(2)$ take, we always get the same integral $(\ref{integral})$. Hence, concerning the value of the integral, we have an independence between the position of $\sigma(1),\sigma(2)$ and the permutation of the $2k-2$ remaining elements.

The above example is based on the permutation (of the $2k-2$ elements) $\tau(1)=1,...,\tau(2k-2)=2k-2$. However, this was just because it was easier to explain it and less cumbersome in details. Indeed, the argument can be extended to any permutation $\tau\in\mathcal{G}_{2k-2}$ (\textit{i.e.} any permutation of the $2k-2$ elements). Indeed, we have the following second example.
\\\textbf{Example 2.} Fix a permutation $\tau$ of the set $(1,...,2k-2)$. Say for example $(1,3,5,7,2,4,6,8,9,...,2k-2)$. Then consider the permutation $\sigma$ of the $2k$ set $(1,2,3,5,7,9,4,6,8,10,11,...,2k)$. Then we have
\begin{equation}
\int_{0<t_{3}<...<t_{2k}<1}\prod_{l=2}^{k}\delta_{i_{\sigma(2l)},i_{\sigma(2l-1)}}|t_{\sigma(2l)}-t_{\sigma(2l-1)}|^{2H-2}\text{d}t_{\sigma(2l)}\text{d}t_{\sigma(2l-1)}
\end{equation}
(skipping the $\delta$s since they follows the terms $|t_{\sigma(2p)}-t_{\sigma(2p-1)}|^{2H-2}$ with the same order)
\begin{equation*}
=\int_{0<t_{3}<...<t_{2k}<1}|t_{5}-t_{3}|^{2H-2}|t_{9}-t_{7}|^{2H-2}|t_{6}-t_{4}|^{2H-2}|t_{10}-t_{8}|^{2H-2}
\end{equation*}
\begin{equation}
\prod_{l=6}^{k}|t_{\sigma(2l)}-t_{\sigma(2l-1)}|^{2H-2}\text{d}t_{3}\cdots\text{d}t_{2k}
\end{equation}
\begin{equation}\label{integral2}
= \int_{\Delta^{2k-2}[0,1]}\prod_{l=1}^{k-1}\delta_{i_{\tau(2l)},i_{\tau(2l-1)}}|s_{\tau(2l)}-s_{\tau(2l-1)}|^{2H-2}\text{d}s_{1}\cdots \text{d}s_{2k-2}.
\end{equation}
Now, notice that we can get the same integral $(\ref{integral2})$ from the permutation $\tilde{\sigma}$: $(2k-1,2k,1,3,5,7,2,4,6,8,9,...,2k-2)$. This is because
\begin{equation}
\int_{0<t_{1}<...<t_{2k-2}<1}\prod_{l=2}^{k}\delta_{i_{\sigma(2l)},i_{\sigma(2l-1)}}|t_{\sigma(2l)}-t_{\sigma(2l-1)}|^{2H-2}\text{d}t_{\sigma(2l)}\text{d}t_{\sigma(2l-1)}
\end{equation}
\begin{equation*}
= \int_{0<t_{1}<...<t_{2k-2}<1}|t_{3}-t_{1}|^{2H-2}|t_{7}-t_{5}|^{2H-2}|t_{4}-t_{2}|^{2H-2}|t_{8}-t_{6}|^{2H-2}
\end{equation*}
\begin{equation}
\prod_{l=5}^{k-1}|t_{\sigma(2l)}-t_{\sigma(2l-1)}|^{2H-2}\text{d}t_{1}\cdots\text{d}t_{2k-2}
\end{equation}
\begin{equation}
= \int_{\Delta^{2k-2}[0,1]}\prod_{l=1}^{k-1}\delta_{i_{\tau(2l)},i_{\tau(2l-1)}}|t_{\tau(2l)}-t_{\tau(2l-1)}|^{2H-2}\text{d}t_{1}\cdots \text{d}t_{2k-2}.
\end{equation}
Again we can get the same integral $(\ref{integral2})$ from the permutation $\hat{\sigma}$: $(4,9,1,3,6,8,2,5,7,10,11,...,2k)$ (here $\sigma(1)=4,\sigma(2)=9$). This is because
\begin{equation}
\int_{0<t_{1}<t_{2}<t_{3}<t_{5}<...<t_{8}<t_{10}<...<1}\prod_{l=2}^{k}\delta_{i_{\sigma(2l)},i_{\sigma(2l-1)}}|t_{\sigma(2l)}-t_{\sigma(2l-1)}|^{2H-2}\text{d}t_{\sigma(2l)}\text{d}t_{\sigma(2l-1)}
\end{equation}
\begin{equation*}
=\int_{0<t_{1}<t_{2}<t_{3}<t_{5}<...<t_{8}<t_{10}<...<1}|t_{3}-t_{1}|^{2H-2}|t_{8}-t_{6}|^{2H-2}|t_{5}-t_{2}|^{2H-2}|t_{10}-t_{7}|^{2H-2}
\end{equation*}
\begin{equation}
\prod_{l=6}^{k}|t_{\sigma(2l)}-t_{\sigma(2l-1)}|^{2H-2}\text{d}t_{1}\text{d}t_{2}\text{d}t_{3}\text{d}t_{5}\cdots\text{d}t_{8}\text{d}t_{10}\cdots\text{d}t_{2k}
\end{equation}
\begin{equation}
= \int_{\Delta^{2k-2}[0,1]}\prod_{l=1}^{k-1}\delta_{i_{\tau(2l)},i_{\tau(2l-1)}}|s_{\tau(2l)}-s_{\tau(2l-1)}|^{2H-2}\text{d}s_{1}\cdots \text{d}s_{2k-2}.
\end{equation}
The principle is the following. Take a permutation $\tau$ of the set of $2k-2$ elements, in our Example 2 was $(1,3,5,7,2,4,6,8,9,...,2k-2)$. Take two points $\sigma(1),\sigma(2)$, in the last case of Example 2 $\sigma(1)=4,\sigma(2)=9$. We need to find the permutations $\sigma\in\mathcal{G}_{2k}$ such that we have the same integral $(\ref{integral2})$. How can we find them? To get these permutations $\sigma$s we need just to follow the following algorithm. First, we put $\sigma(1),\sigma(2)$ for the first two positions in the $2k$ set, while for the others we put a \textit{modification} of $(1,3,5,7,2,4,6,8,9,...,2k-2)$, which depends on the value of $\sigma(1),\sigma(2)$. We denote this \textit{modification} $(y_{1},...,y_{2k-2})$. Hence, we have $(\sigma(1),\sigma(2),y_{1},...,y_{2k-2})$. In particular, the \textit{modification} follows this rule. Assume $\sigma(1)<\sigma(2)$ without loss of generality and let $x$ be an element of the set $(1,3,5,7,2,4,6,8,9,...,2k-2)$ so $x_{1}=1,x_{2}=3,x_{3}=5,...$. If the values below $x_{i}<\sigma(1)$ then $y_{i}=x_{i}$ (this is the case for $1,3,2$ in our example). If $\sigma(1)\leq x_{i}<\sigma(2)-1$ then $y_{i}=x_{i}+1$ (this is the case for $5,7,4,6$ in our example).  If $x_{i}\geq\sigma(2)-1$ then $y_{i}=x_{i}+2$ (this is the case for $8,9,10,...,2k-2$ in our example).
\\
\\
Now there are two questions to be answered. The first question is: we say that there are different possibilities (\textit{i.e.} different $\sigma$s) to get the same integral (say for example integral $(\ref{integral2})$) but how many exactly? The answer is $\frac{(2k)!}{(2k-2)!}=2k(2k-1)$ which is the arrangements of a fixed length $2$ of elements taken from a given set of size $2k$, in other words, these $2$-permutations of $2k$ are the different ordered arrangements of a $2$-element subset of an $2k$-set (see Wikipedia or Wolfram Alpha on permutation). For example, $2$-permutations of $3$ are $3!/(3-2)!=6$. Indeed, consider the set $\{1,2,3\}$ then we have $\{1,2\},\{1,3\},\{2,3\},\{2,1\},\{3,1\},\{3,2\}$. In our case we have $2$-permutations on $2k$ because the $2$ comes from $\sigma(1),\sigma(2)$ and the $2k$ from the size of the set.
\\
This implies that there are $2k(2k-1)$ different permutations in $\mathcal{G}_{2k}$ (call them $\sigma_{1},\sigma_{2},...,\sigma_{2k(2k-1)}$) such that we get the same integral in terms of a permutation $\tau$ (like integral $(\ref{integral2})$). Hence, we have
\begin{equation*}
\sum_{j=1}^{2k(2k-1)}\int_{0<t_{1}<...<t_{\sigma_{j}(1)-1}<t_{\sigma_{j}(1)+1}<...<t_{\sigma_{j}(2)-1}<t_{\sigma_{j}(2)+1}<...<1}
\end{equation*}
\begin{equation}
\prod_{l=2}^{k}\delta_{i_{\sigma_{j}(2l)},i_{\sigma_{j}(2l-1)}}|t_{\sigma_{j}(2l)}-t_{\sigma_{j}(2l-1)}|^{2H-2}\text{d}t_{\sigma_{j}(2l)}\text{d}t_{\sigma_{j}(2l-1)}
\end{equation}
\begin{equation}
=2k(2k-1)\int_{\Delta^{2k-2}[0,1]}\prod_{l=1}^{k-1}\delta_{i_{\tau(2l)},i_{\tau(2l-1)}}|s_{\tau(2l)}-s_{\tau(2l-1)}|^{2H-2}\text{d}s_{1}\cdots \text{d}s_{2k-2}.
\end{equation}
\\
\\
The second and last question is the following: by using the arguments above can we cover all the permutations $\sigma\in\mathcal{G}_{2k}$? In other words, can we reformulate all the permutations $\sigma\in\mathcal{G}_{2k}$ in terms of the permutations $\tau\in\mathcal{G}_{2k-2}$ times $2k(2k-1)$? The answer is yes and in particular we have the following equality
\begin{equation*}
\sum_{\sigma\in\mathcal{G}_{2k}}\int_{0<t_{1}<...<t_{\sigma(1)-1}<t_{\sigma(1)+1}<...<t_{\sigma(2)-1}<t_{\sigma(2)+1}<...<1}
\end{equation*}
\begin{equation}
\prod_{l=2}^{k}\delta_{i_{\sigma(2l)},i_{\sigma(2l-1)}}|t_{\sigma(2l)}-t_{\sigma(2l-1)}|^{2H-2}\text{d}t_{\sigma(2l)}\text{d}t_{\sigma(2l-1)}
\end{equation}
\begin{equation}
=2k(2k-1)\sum_{\tau\in\mathcal{G}_{2k-2}}\int_{\Delta^{2k-2}[0,1]}\prod_{l=1}^{k-1}\delta_{i_{\tau(2l)},i_{\tau(2l-1)}}|s_{\tau(2l)}-s_{\tau(2l-1)}|^{2H-2}\text{d}s_{1}\cdots \text{d}s_{2k-2}.
\end{equation}
This is because we can decompose the permutations of $\mathcal{G}_{2k}$ (the $\sigma$s) into permutations of $\sigma(1),\sigma(2)$ (which are $2k(2k-1)$ permutations for each permutation $\tau$ and they do not modify the value of the integral) and the permutations of $\mathcal{G}_{2k-2}$ (the $\tau$s). Indeed, for $\mathcal{G}_{2k}$ we have $(2k)!$ permutations (and of course they are all different from each other) and for the other we have $(2k-2)!\cdot2k(2k-1)=(2k)!$ (and they are also all different from each other).
\\
Again, the key point is to decompose the permutation $\sigma$ of $\mathcal{G}_{2k}$ into a permutation $\tau\in\mathcal{G}_{2k-2}$ and a 2-permutation on $2k$.
\section{Appendix 1: Iterated integrals for the cubature method}
In this appendix we study the iterated integrals with respect to the path $(t,B^{H}_{t})$. Notice that we will use the notation $B_{t}:=B^{H}_{t}$. We will not consider the case of iterated integrals of time solely since they bring no information for the construction of the cubature. Further, we will use many times the following formula for the fractional Brownian motion:
\begin{equation*}
\mathbb{E}\left((B_{t}-B_{s})^{2k}\right)=\dfrac{(2k)!}{k!2^{k}}|t-s|^{2Hk}
\end{equation*}
In particular, we have:
\\
Degree$=2H$:
\begin{equation*}
\mathbb{E}\left(\int_{0}^{1}dB_{u_{1}}\right)=0
\end{equation*}
\\
Degree$=4H$:
\begin{equation*}
\mathbb{E}\left(\int_{0}^{1}\int_{0}^{u_{2}}dB_{u_{1}}dB_{2}\right)=\mathbb{E}\left(\dfrac{B_{1}^{2}}{2}\right)=\dfrac{1}{2}
\end{equation*}
\\
Degree$=2H+2$:
\begin{equation*}
\mathbb{E}\left(\int_{0}^{1}\int_{0}^{u_{2}}dB_{u_{1}}du_{2}\right)=\int_{0}^{1}\mathbb{E}\left(B_{u_{2}}\right)du_{2}=0
\end{equation*}
and
\begin{equation*}
\mathbb{E}\left(\int_{0}^{1}\int_{0}^{u_{2}}du_{1}dB_{u_{2}}\right)=\mathbb{E}\left(\int_{0}^{1}u_{2}dB_{u_{2}}\right)=0
\end{equation*}
\\
Degree$=6H$:
\begin{equation*}
\mathbb{E}\left(\int_{0}^{1}\int_{0}^{u_{3}}\int_{0}^{u_{2}}dB_{u_{1}}dB_{2}dB_{3}\right)=\mathbb{E}\left(\dfrac{B_{1}^{3}}{3!}\right)=0
\end{equation*}
\\
Degree$=2+4H$:
\begin{equation*}
\mathbb{E}\left(\int_{0}^{1}\int_{0}^{u_{3}}\int_{0}^{u_{2}}dB_{u_{1}}dB_{u_{2}}du_{3}\right)=\int_{0}^{1}\mathbb{E}\left(\dfrac{B_{u_{3}}^{2}}{2}\right)du_{3}=\int_{0}^{1}\dfrac{u_{3}^{2H}}{2}du_{3}=\dfrac{1}{2(2H+1)}
\end{equation*}
and
\begin{equation*}
\mathbb{E}\left(\int_{0}^{1}\int_{0}^{u_{3}}\int_{0}^{u_{2}}dB_{u_{1}}du_{2}dB_{u_{3}}\right)=\mathbb{E}\left(\int_{0}^{1}\int_{0}^{u_{3}}B_{u_{2}}du_{2}dB_{u_{3}}\right)=\mathbb{E}\left(\int_{0}^{1}\int_{u_{2}}^{1}B_{u_{2}}dB_{u_{3}}du_{2}\right)
\end{equation*}
\begin{equation*}
=\int_{0}^{1}\mathbb{E}\left(B_{u_{2}}(B_{1}-B_{u_{2}})\right)du_{3}=\int_{0}^{1}\dfrac{1}{2}\left(1-u_{2}^{2H}-(1-u_{2})^{2H}\right)du_{2}=\dfrac{1}{2}-\dfrac{2}{2(2H+1)}=\dfrac{2H-1}{2(2H+1)}
\end{equation*}
and
\begin{equation*}
\mathbb{E}\left(\int_{0}^{1}\int_{0}^{u_{3}}\int_{0}^{u_{2}}du_{1}dB_{u_{2}}dB_{u_{3}}\right)=\mathbb{E}\left(\int_{0}^{1}\int_{0}^{u_{3}}\int_{u_{1}}^{u_{3}}dB_{u_{2}}du_{1}dB_{u_{3}}\right)=\mathbb{E}\left(\int_{0}^{1}\int_{0}^{u_{3}}B_{u_{3}}-B_{u_{1}}du_{1}dB_{u_{3}}\right)
\end{equation*}
\begin{equation*}
=\mathbb{E}\left(\int_{0}^{1}\int_{u_{1}}^{1}B_{u_{3}}-B_{u_{1}}dB_{u_{3}}du_{1}\right)=\int_{0}^{1}\mathbb{E}\left(\dfrac{B_{1}^{2}}{2}-\dfrac{B_{u_{1}}^{2}}{2}-B_{u_{1}}(B_{1}-B_{u_{1}})\right)du_{1}
\end{equation*}
\begin{equation*}
=\int_{0}^{1}\mathbb{E}\left(\dfrac{B_{1}^{2}}{2}+\dfrac{B_{u_{1}}^{2}}{2}-B_{u_{1}}B_{1}\right)du_{1}=\int_{0}^{1}\dfrac{1}{2}+\dfrac{u_{1}^{2H}}{2}-\dfrac{1}{2}(1+u_{1}^{2H}-(1-u_{1})^{2H})du_{1}
\end{equation*}
\begin{equation*}
=\dfrac{1}{2}\int_{0}^{1}(1-u_{1})^{2H}du_{1}=\dfrac{1}{2(2H+1)}
\end{equation*}
\\
Degree$=8H$:
\begin{equation*}
\mathbb{E}\left(\int_{0}^{1}\int_{0}^{u_{4}}\int_{0}^{u_{3}}\int_{0}^{u_{2}}dB_{u_{1}}dB_{2}dB_{3}dB_{4}\right)=\mathbb{E}\left(\dfrac{B_{1}^{4}}{4!}\right)=\dfrac{1}{8}
\end{equation*}
\\
Degree$=4+2H$:
\begin{equation*}
\mathbb{E}\left(\int_{0}^{1}\int_{0}^{u_{3}}\int_{0}^{u_{2}}dB_{u_{1}}du_{2}du_{3}\right)=\int_{0}^{1}\int_{0}^{u_{3}}\mathbb{E}\left(B_{u_{2}}\right)du_{2}du_{3}=0
\end{equation*}
and
\begin{equation*}
\mathbb{E}\left(\int_{0}^{1}\int_{0}^{u_{3}}\int_{0}^{u_{2}}du_{1}dB_{u_{2}}du_{3}\right)=\mathbb{E}\left(\int_{0}^{1}\int_{0}^{u_{3}}\int_{u_{1}}^{u_{3}}dB_{u_{2}}du_{1}du_{3}\right)=\int_{0}^{1}\int_{0}^{u_{3}}\mathbb{E}\left(B_{u_{3}}-B_{u_{1}}\right)du_{1}du_{3}=0
\end{equation*}
and
\begin{equation*}
\mathbb{E}\left(\int_{0}^{1}\int_{0}^{u_{3}}\int_{0}^{u_{2}}du_{1}du_{2}dB_{u_{3}}\right)=\mathbb{E}\left(\int_{0}^{1}\dfrac{u_{3}^{2}}{2}dB_{u_{3}}\right)=0
\end{equation*}
Degree$=6H+2$:
\begin{equation*}
\mathbb{E}\left(\int_{0}^{1}\int_{0}^{u_{4}}\int_{0}^{u_{3}}\int_{0}^{u_{2}}dB_{u_{1}}dB_{u_{2}}dB_{u_{3}}du_{4}\right)=\int_{0}^{1}\mathbb{E}\left(\dfrac{B_{u_{4}}^{3}}{3!}\right)du_{4}=0
\end{equation*}
and
\begin{equation*}
\mathbb{E}\left(\int_{0}^{1}\int_{0}^{u_{4}}\int_{0}^{u_{3}}\int_{0}^{u_{2}}dB_{u_{1}}dB_{u_{2}}du_{3}dB_{u_{4}}\right)=\mathbb{E}\left(\int_{0}^{1}\int_{0}^{u_{4}}\dfrac{B_{u_{3}}^{2}}{2!}du_{3}dB_{u_{4}}\right)=\mathbb{E}\left(\int_{0}^{1}\int_{u_{3}}^{1}\dfrac{B_{u_{3}}^{2}}{2!}dB_{u_{4}}du_{3}\right)
\end{equation*}
\begin{equation*}
=\int_{0}^{1}\mathbb{E}\left((B_{1}-B_{u_{3}})\dfrac{B_{u_{3}}^{2}}{2!}\right)du_{3}=0
\end{equation*}
and
\begin{equation*}
\mathbb{E}\left(\int_{0}^{1}\int_{0}^{u_{4}}\int_{0}^{u_{3}}\int_{0}^{u_{2}}dB_{u_{1}}du_{2}dB_{u_{3}}dB_{u_{4}}\right)=\mathbb{E}\left(\int_{0}^{1}\int_{0}^{u_{4}}\int_{0}^{u_{3}}B_{u_{2}}du_{2}dB_{u_{3}}dB_{u_{4}}\right)
\end{equation*}
\begin{equation*}
=\mathbb{E}\left(\int_{0}^{1}\int_{0}^{u_{4}}\int_{u_{2}}^{u_{4}}B_{u_{2}}dB_{u_{3}}du_{2}dB_{u_{4}}\right)=\mathbb{E}\left(\int_{0}^{1}\int_{0}^{u_{4}}(B_{u_{4}}-B_{u_{2}})B_{u_{2}}du_{2}dB_{u_{4}}\right)
\end{equation*}
\begin{equation*}
=\mathbb{E}\left(\int_{0}^{1}\int_{u_{2}}^{1}(B_{u_{4}}-B_{u_{2}})B_{u_{2}}dB_{u_{4}}du_{2}\right)=\int_{0}^{1}\mathbb{E}\left(\dfrac{B_{u_{2}}}{2}(B_{1}-B_{u_{2}})^{2}-B_{u_{2}}^{2}(B_{1}-B_{u_{2}})\right)du_{2}=0
\end{equation*}
and
\begin{equation*}
\mathbb{E}\left(\int_{0}^{1}\int_{0}^{u_{4}}\int_{0}^{u_{3}}\int_{0}^{u_{2}}du_{1}dB_{u_{2}}dB_{u_{3}}dB_{u_{4}}\right)=\mathbb{E}\left(\int_{0}^{1}\int_{0}^{u_{4}}\int_{0}^{u_{3}}\int_{u_{1}}^{u_{3}}dB_{u_{2}}du_{1}dB_{u_{3}}dB_{u_{4}}\right)
\end{equation*}
\begin{equation*}
=\mathbb{E}\left(\int_{0}^{1}\int_{0}^{u_{4}}\int_{0}^{u_{3}}(B_{u_{3}}-B_{u_{1}})du_{1}dB_{u_{3}}dB_{u_{4}}\right)=\mathbb{E}\left(\int_{0}^{1}\int_{0}^{u_{4}}\int_{u_{1}}^{u_{4}}(B_{u_{3}}-B_{u_{1}})dB_{u_{3}}du_{1}dB_{u_{4}}\right)
\end{equation*}
\begin{equation*}
=\mathbb{E}\left(\int_{0}^{1}\int_{0}^{u_{4}}\dfrac{B_{u_{4}}^{2}}{2}-\dfrac{B_{u_{1}}^{2}}{2}-B_{u_{1}}(B_{u_{4}}-B_{u_{1}})du_{1}dB_{u_{4}}\right)=\mathbb{E}\left(\int_{0}^{1}\int_{u_{1}}^{1}\dfrac{B_{u_{4}}^{2}}{2}-\dfrac{B_{u_{1}}^{2}}{2}-B_{u_{1}}(B_{u_{4}}-B_{u_{1}})dB_{u_{4}}du_{1}\right)
\end{equation*}
\begin{equation*}
=\int_{0}^{1}\mathbb{E}\left(\dfrac{B_{1}^{3}}{3!}-\dfrac{B_{u_{1}}^{3}}{3!}-\dfrac{B_{u_{1}}^{2}}{2}(B_{1}-B_{u_{1}})-B_{u_{1}}\left(\dfrac{B_{1}^{2}}{2}-\dfrac{B_{u_{1}}^{2}}{2}\right)+B_{u_{1}}^{2}(B_{1}-B_{u_{1}})\right)du_{1}=0
\end{equation*}
Degree$=10H$:
\begin{equation*}
\mathbb{E}\left(\int_{0}^{1}\int_{0}^{u_{5}}\int_{0}^{u_{4}}\int_{0}^{u_{3}}\int_{0}^{u_{2}}dB_{u_{1}}dB_{2}dB_{3}dB_{4}dB_{5}\right)=\mathbb{E}\left(\dfrac{B_{1}^{5}}{5!}\right)=0
\end{equation*}
\\
\\
Now we need to match them with the corresponding deterministic iterated integrals. In other words, we have
\\
\\
For degree$=2H$:
\begin{equation*}
0=\sum_{i=1}^{n}\lambda_{i}\int_{0}^{1}d\omega_{u_{1},i}\Rightarrow \sum_{i=1}^{n}\lambda_{i}\omega_{i,1}=0
\end{equation*}
\\
For degree$=4H$:
\begin{equation*}
\dfrac{1}{2}=\sum_{i=1}^{n}\lambda_{i}\int_{0}^{1}\int_{0}^{u_{2}}d\omega_{i,u_{1}}d\omega_{i,u_{2}}\Rightarrow \dfrac{1}{2}=\dfrac{1}{2}\sum_{i=1}^{n}\lambda_{i}\omega_{i,1}^{2}\Rightarrow \sum_{i=1}^{n}\lambda_{i}\omega_{i,1}^{2}=1
\end{equation*}
\\
For degree$=2H+2$:
\begin{equation*}
0=\sum_{i=1}^{n}\lambda_{i}\int_{0}^{1}\int_{0}^{u_{2}}d\omega_{i,u_{1}}du_{2}\Rightarrow \sum_{i=1}^{n}\lambda_{i}\int_{0}^{1}\omega_{i,u_{2}}du_{2}=0
\end{equation*}
and
\begin{equation*}
0=\sum_{i=1}^{n}\lambda_{i}\int_{0}^{1}\int_{0}^{u_{2}}du_{1}d\omega_{i,u_{2}}\Rightarrow \sum_{i=1}^{n}\lambda_{i}\int_{0}^{1}u_{2}d\omega_{i,u_{2}}=0
\end{equation*}
For degree$=6H$:
\begin{equation*}
0=\sum_{i=1}^{n}\lambda_{i}\int_{0}^{1}\int_{0}^{u_{3}}\int_{0}^{u_{2}}d\omega_{i,u_{1}}d\omega_{i,u_{2}}d\omega_{i,u_{3}}\Rightarrow \dfrac{1}{3!}\sum_{i=1}^{n}\lambda_{i}\omega_{i,1}^{3}=0\Rightarrow
\sum_{i=1}^{n}\lambda_{i}\omega_{i,1}^{3}=0
\end{equation*}
\\
For degree$=4H+2$:
\begin{equation*}
\dfrac{1}{2(2H+1)}=\sum_{i=1}^{n}\lambda_{i}\int_{0}^{1}\int_{0}^{u_{3}}\int_{0}^{u_{2}}d\omega_{i,u_{1}}d\omega_{i,u_{2}}du_{3}\Rightarrow \sum_{i=1}^{n}\lambda_{i}\int_{0}^{1}\omega_{i,u_{3}}^{2}du_{3}=\dfrac{1}{2H+1}
\end{equation*}
and
\begin{equation*}
\dfrac{2H-1}{2(2H+1)}=\sum_{i=1}^{n}\lambda_{i}\int_{0}^{1}\int_{0}^{u_{3}}\int_{0}^{u_{2}}d\omega_{i,u_{1}}du_{2}d\omega_{i,u_{3}}\Rightarrow \sum_{i=1}^{n}\lambda_{i}\int_{0}^{1}\int_{0}^{u_{3}}\omega_{i,u_{2}}du_{2}d\omega_{i,u_{3}}=\dfrac{2H-1}{2(2H+1)}
\end{equation*}
and
\begin{equation*}
\dfrac{1}{2(2H+1)}=\sum_{i=1}^{n}\lambda_{i}\int_{0}^{1}\int_{0}^{u_{3}}\int_{0}^{u_{2}}du_{1}d\omega_{i,u_{2}}d\omega_{i,u_{3}}\Rightarrow \sum_{i=1}^{n}\lambda_{i}\int_{0}^{1}\int_{0}^{u_{3}}u_{2}d\omega_{i,u_{2}}d\omega_{i,u_{3}}=\dfrac{1}{2(2H+1)}
\end{equation*}
For degree$=8H$:
\begin{equation*}
\dfrac{1}{8}=\sum_{i=1}^{n}\lambda_{i}\int_{0}^{1}\int_{0}^{u_{4}}\int_{0}^{u_{3}}\int_{0}^{u_{2}}d\omega_{i,u_{1}}d\omega_{i,u_{2}}d\omega_{i,u_{3}}d\omega_{i,u_{4}}\Rightarrow \dfrac{1}{4!}\sum_{i=1}^{n}\lambda_{i}\omega_{i,1}^{4}=\dfrac{1}{8}\Rightarrow
\sum_{i=1}^{n}\lambda_{i}\omega_{i,1}^{4}=3
\end{equation*}
\\
For degree$=2H+4$:
\begin{equation*}
0=\sum_{i=1}^{n}\lambda_{i}\int_{0}^{1}\int_{0}^{u_{3}}\int_{0}^{u_{2}}d\omega_{i,u_{1}}du_{2}du_{3}\Rightarrow \sum_{i=1}^{n}\lambda_{i}\int_{0}^{1}\int_{0}^{u_{3}}\omega_{i,u_{2}}du_{2}du_{3}=0
\end{equation*}
and
\begin{equation*}
0=\sum_{i=1}^{n}\lambda_{i}\int_{0}^{1}\int_{0}^{u_{3}}\int_{0}^{u_{2}}du_{1}d\omega_{i,u_{2}}du_{3}\Rightarrow \sum_{i=1}^{n}\lambda_{i}\int_{0}^{1}\int_{0}^{u_{3}}u_{2}d\omega_{i,u_{2}}du_{3}=0
\end{equation*}
and
\begin{equation*}
0=\sum_{i=1}^{n}\lambda_{i}\int_{0}^{1}\int_{0}^{u_{3}}\int_{0}^{u_{2}}du_{1}du_{2}d\omega_{i,u_{3}}\Rightarrow \sum_{i=1}^{n}\lambda_{i}\int_{0}^{1}\dfrac{u_{3}^{2}}{2}d\omega_{i,u_{3}}=0\Rightarrow\sum_{i=1}^{n}\lambda_{i}\int_{0}^{1}u_{3}^{2}d\omega_{i,u_{3}}=0
\end{equation*}
For degree$=6H+2$:
\begin{equation*}
0=\sum_{i=1}^{n}\lambda_{i}\int_{0}^{1}\int_{0}^{u_{4}}\int_{0}^{u_{3}}\int_{0}^{u_{2}}d\omega_{i,u_{1}}d\omega_{i,u_{2}}d\omega_{i,u_{3}}du_{4}\Rightarrow \dfrac{1}{3!}\sum_{i=1}^{n}\lambda_{i}\int_{0}^{1}\omega_{i,u_{4}}^{3}du_{4}=0\Rightarrow \sum_{i=1}^{n}\lambda_{i}\int_{0}^{1}\omega_{i,u_{4}}^{3}du_{4}=0
\end{equation*}
and
\begin{equation*}
0=\sum_{i=1}^{n}\lambda_{i}\int_{0}^{1}\int_{0}^{u_{4}}\int_{0}^{u_{3}}\int_{0}^{u_{2}}d\omega_{i,u_{1}}d\omega_{i,u_{2}}du_{3}d\omega_{i,u_{4}}\Rightarrow \dfrac{1}{2}\sum_{i=1}^{n}\lambda_{i}\int_{0}^{1}\int_{0}^{u_{4}}\omega_{i,u_{3}}^{2}du_{3}d\omega_{i,u_{4}}=0
\end{equation*}
\begin{equation*}
\Rightarrow \sum_{i=1}^{n}\lambda_{i}\int_{0}^{1}\int_{0}^{u_{4}}\omega_{i,u_{3}}^{3}du_{3}d\omega_{i,u_{4}}=0
\end{equation*}
and
\begin{equation*}
0=\sum_{i=1}^{n}\lambda_{i}\int_{0}^{1}\int_{0}^{u_{4}}\int_{0}^{u_{3}}\int_{0}^{u_{2}}d\omega_{i,u_{1}}du_{2}d\omega_{i,u_{3}}d\omega_{i,u_{4}}\Rightarrow \sum_{i=1}^{n}\lambda_{i}\int_{0}^{1}\int_{0}^{u_{4}}\int_{0}^{u_{3}}\omega_{i,u_{2}}du_{2}d\omega_{i,u_{3}}d\omega_{i,u_{4}}=0
\end{equation*}
and
\begin{equation*}
0=\sum_{i=1}^{n}\lambda_{i}\int_{0}^{1}\int_{0}^{u_{4}}\int_{0}^{u_{3}}\int_{0}^{u_{2}}du_{1}d\omega_{i,u_{2}}d\omega_{i,u_{3}}d\omega_{i,u_{4}}\Rightarrow \sum_{i=1}^{n}\lambda_{i}\int_{0}^{1}\int_{0}^{u_{4}}\int_{0}^{u_{3}}u_{2}d\omega_{i,u_{2}}d\omega_{i,u_{3}}d\omega_{i,u_{4}}=0
\end{equation*}
For degree$=10H$:
\begin{equation*}
0=\sum_{i=1}^{n}\lambda_{i}\int_{0}^{1}\int_{0}^{u_{5}}\int_{0}^{u_{4}}\int_{0}^{u_{3}}\int_{0}^{u_{2}}d\omega_{i,u_{1}}d\omega_{i,u_{2}}d\omega_{i,u_{3}}d\omega_{i,u_{4}}d\omega_{i,u_{5}}\Rightarrow \dfrac{1}{5!}\sum_{i=1}^{n}\lambda_{i}\omega_{i,1}^{5}=0\Rightarrow
\sum_{i=1}^{n}\lambda_{i}\omega_{i,1}^{5}=0
\end{equation*}
\section{Appendix 2: Extended proof of Theorem 2.6}
In this appendix we present the extended proof of \textbf{Theorem 2.6}.
\begin{proof}
	Let us now start to investigate the form of the functions $\omega_{j}$ for $j=1,...,n$.
	\\
	From the Appendix it is possible to see that we have 17 equations, hence we need to have 17 unknowns.
	\\
	Following the work done by Lyons and Victoir we are going to choose two symmetric paths and one path which has constant value zero. This reduces the number of equations to 5.
	Hence, assume that there are two continuous functions, $\omega_{1,s}$ and $\omega_{2,s}$, with the property that $\omega_{1,s}=-\omega_{2,s}$ for $s\in[0,1]$. Further assume that there is a third path $\omega_{3,s}=0$ for $s\in[0,1]$. With this formulation only 5 equations need to be taken into consideration since the other 12 are already satisfied. The 5 equations are the following. First,
	\begin{equation}\label{2}
	\dfrac{1}{2}=\sum_{i=1}^{n}\lambda_{i}\int_{0}^{1}\int_{0}^{u_{2}}d\omega_{i,u_{1}}d\omega_{i,u_{2}}\Rightarrow \dfrac{1}{2}=\dfrac{1}{2}\sum_{i=1}^{n}\lambda_{i}\omega_{i,1}^{2}\Rightarrow \sum_{i=1}^{n}\lambda_{i}\omega_{i,1}^{2}=1
	\end{equation}
	Second,
	\begin{equation}\label{6}
	\dfrac{1}{2(2H+1)}=\sum_{i=1}^{n}\lambda_{i}\int_{0}^{1}\int_{0}^{u_{3}}\int_{0}^{u_{2}}d\omega_{i,u_{1}}d\omega_{i,u_{2}}du_{3}\Rightarrow \sum_{i=1}^{n}\lambda_{i}\int_{0}^{1}\omega_{i,u_{3}}^{2}du_{3}=\dfrac{1}{2H+1}
	\end{equation}
	Third,
	\begin{equation}\label{7}
	\dfrac{2H-1}{2(2H+1)}=\sum_{i=1}^{n}\lambda_{i}\int_{0}^{1}\int_{0}^{u_{3}}\int_{0}^{u_{2}}d\omega_{i,u_{1}}du_{2}d\omega_{i,u_{3}}\Rightarrow \sum_{i=1}^{n}\lambda_{i}\int_{0}^{1}\int_{0}^{u_{3}}\omega_{i,u_{2}}du_{2}d\omega_{i,u_{3}}=\dfrac{2H-1}{2(2H+1)}
	\end{equation}
	Fourth,
	\begin{equation}\label{8}
	\dfrac{1}{2(2H+1)}=\sum_{i=1}^{n}\lambda_{i}\int_{0}^{1}\int_{0}^{u_{3}}\int_{0}^{u_{2}}du_{1}d\omega_{i,u_{2}}d\omega_{i,u_{3}}\Rightarrow \sum_{i=1}^{n}\lambda_{i}\int_{0}^{1}\int_{0}^{u_{3}}u_{2}d\omega_{i,u_{2}}d\omega_{i,u_{3}}=\dfrac{1}{2(2H+1)}
	\end{equation}
	Fifth,
	\begin{equation}\label{9}
	\dfrac{1}{8}=\sum_{i=1}^{n}\lambda_{i}\int_{0}^{1}\int_{0}^{u_{4}}\int_{0}^{u_{3}}\int_{0}^{u_{2}}d\omega_{i,u_{1}}d\omega_{i,u_{2}}d\omega_{i,u_{3}}d\omega_{i,u_{4}}\Rightarrow \dfrac{1}{4!}\sum_{i=1}^{n}\lambda_{i}\omega_{i,1}^{4}=\dfrac{1}{8}\Rightarrow
	\sum_{i=1}^{n}\lambda_{i}\omega_{i,1}^{4}=3
	\end{equation}
	The other 12 equations are automatically zero by the symmetric properties of the $\omega_{1,s}$ and $\omega_{,s2}$, and by the fact that $\omega_{3,s}=0$, since the 12 equations involve odd integrals of the $\omega$s.
	\\
	Now we need to have maximum 5 unknowns in order to solve the system of equations. We have actually have 6 equations since the sum of the weights $\sum_{i=1}^{3}\lambda_{i}=1$, because we are considering a probability measure. Assume that
	\begin{equation*}
	\omega_{1,s}=
	\begin{cases}
	as\qquad for \qquad s\in[0,\frac{1}{3}],\\
	b_{1}s+b_{0}\qquad for \qquad s\in[\frac{1}{3},\frac{2}{3}],\\
	c_{1}s+c_{0}\qquad for \qquad s\in[\frac{2}{3},1].\\
	\end{cases}
	\end{equation*}
	With this formulation we have 5 unknowns which are $a$, $b_{1}$, $c_{1}$, $\lambda_{1}$ and $\lambda_{3}$. The reason why $b_{0}$ and $c_{0}$ are not unknowns is because they have to take certain values in order to make the path $\omega_{1,s}$ continuous. Notice that $\lambda_{2}$ is not an unknowns since $\lambda_{2}=\lambda_{1}$. We have 5 unknowns for 6 equations there is a risk that the system cannot be solved. However, we hope that two of the 5 equations are the same. An alternative approach is to let the points where the slope changes, which we fixed to be at $\frac{1}{3}$ and $\frac{2}{3}$, be two unknowns.
	\\
	Let us now solve the system. Consider equation $(\ref{2})$, we have
	\begin{equation*}
	\lambda_{1}\omega_{1,1}^{2}+\lambda_{2}\omega_{2,1}^{2}=1\Rightarrow 2\lambda_{1}\omega_{1,1}^{2}=1\Rightarrow 2\lambda_{1}(c_{1}+c_{0})^{2}=1
	\end{equation*}
	Consider equation $(\ref{6})$, we have
	\begin{equation*}
	2\lambda_{1}\int_{0}^{1}\omega_{1,u_{3}}^{2}du_{3}=\dfrac{1}{2H+1}\Rightarrow \int_{0}^{\frac{1}{3}}a^{2}u_{3}^{2}du_{3}+\int_{\frac{1}{3}}^{\frac{2}{3}}(b_{1}u_{3}+b_{0})^{2}du_{3}+\int_{\frac{2}{3}}^{1}(c_{1}u_{3}+c_{0})^{2}du_{3}=\dfrac{1}{2\lambda_{1}(2H+1)}
	\end{equation*}
	\begin{equation*}
	\Rightarrow \dfrac{a^{2}}{81}+\dfrac{1}{3b_{1}}\left(\dfrac{8}{27}b_{1}^{3}+b_{0}^{3}+\dfrac{4}{3}b_{1}^{2}b_{0}+2b_{1}b_{0}^{2}\right)-\dfrac{1}{3b_{1}}\left(\dfrac{1}{27}b_{1}^{3}+b_{0}^{3}+\dfrac{1}{3}b_{1}^{2}b_{0}+b_{1}b_{0}^{2}\right)
	\end{equation*}
	\begin{equation*}
	+\dfrac{1}{3c_{1}}\left(c_{1}^{3}+c_{0}^{3}+3c_{1}^{2}c_{0}+3c_{1}c_{0}^{2}\right)-\dfrac{1}{3c_{1}}\left(\dfrac{8}{27}c_{1}^{3}+c_{0}^{3}+\dfrac{4}{3}c_{1}^{2}c_{0}+2c_{1}c_{0}^{2}\right)=\dfrac{1}{2\lambda_{1}(2H+1)}
	\end{equation*}
	\begin{equation}\label{third}
	\Rightarrow \dfrac{a^{2}}{81}+\dfrac{1}{3}\left(\dfrac{7}{27}b_{1}^{2}+b_{1}b_{0}+b_{0}^{2}\right)+\dfrac{1}{3}\left(\dfrac{19}{27}c_{1}^{2}+\dfrac{5}{3}c_{1}c_{0}+c_{0}^{2}\right)=\dfrac{1}{2\lambda_{1}(2H+1)}
	\end{equation}
	Consider now equation $(\ref{8})$, we have
	\begin{equation*}
	\int_{0}^{1}\int_{0}^{u_{3}}\int_{0}^{u_{2}}du_{1}d\omega_{1,u_{2}}d\omega_{1,u_{3}}=\dfrac{1}{4\lambda_{1}(2H+1)}
	\end{equation*}
	By Fubini's theorem we have
	\begin{equation*}
	\int_{0}^{1}\int_{0}^{u_{3}}\int_{u_{1}}^{u_{3}}d\omega_{1,u_{2}}du_{1}d\omega_{1,u_{3}}=\dfrac{1}{4\lambda_{1}(2H+1)}\Rightarrow\int_{0}^{1}\int_{0}^{u_{3}}(\omega_{1,u_{3}}-\omega_{1,u_{1}})du_{1}d\omega_{1,u_{3}}=\dfrac{1}{4\lambda_{1}(2H+1)}
	\end{equation*}
	\begin{equation*}
	\Rightarrow\int_{0}^{1}\int_{u_{1}}^{1}(\omega_{1,u_{3}}-\omega_{1,u_{1}})d\omega_{1,u_{3}}du_{1}=\dfrac{1}{4\lambda_{1}(2H+1)}\Rightarrow\int_{0}^{1}\dfrac{\omega_{1,1}^{2}}{2}-\dfrac{\omega_{1,u_{1}}^{2}}{2}-\omega_{1,u_{1}}(\omega_{1,1}-\omega_{1,u_{1}})du_{1}=\dfrac{1}{4\lambda_{1}(2H+1)}
	\end{equation*}
	\begin{equation*}
	\Rightarrow \int_{0}^{1}\dfrac{\omega_{1,1}^{2}}{2}+\dfrac{\omega_{1,u_{1}}^{2}}{2}-\omega_{1,1}\omega_{1,u_{1}}du_{1}
	=\dfrac{1}{4\lambda_{1}(2H+1)}
	\end{equation*}
	\begin{equation*}
	\Rightarrow (c_{1}+c_{0})^{2}-2(c_{1}+c_{0})\left[\int_{0}^{\frac{1}{3}}au_{1}du_{1}+\int_{\frac{1}{3}}^{\frac{2}{3}}b_{1}u_{1}+b_{0}du_{1}+\int_{\frac{2}{3}}^{1}c_{1}u_{1}+c_{0}du_{1}\right]
	\end{equation*}
	\begin{equation*}
	+\int_{0}^{\frac{1}{3}}a^{2}u_{1}^{2}du_{1}+\int_{\frac{1}{3}}^{\frac{2}{3}}(b_{1}u_{1}+b_{0})^{2}du_{1}+\int_{\frac{2}{3}}^{1}(c_{1}u_{1}+c_{0})^{2}du_{1}=\dfrac{1}{2\lambda_{1}(2H+1)}
	\end{equation*}
	\begin{equation*}
	\Rightarrow (c_{1}+c_{0})^{2}-2(c_{1}+c_{0})\left[a\dfrac{1}{18}+b_{1}\dfrac{1}{6}+b_{0}\dfrac{1}{3}+c_{1}\dfrac{5}{18}+c_{0}\dfrac{1}{3}\right]
	\end{equation*}
	\begin{equation*}
	+\dfrac{a^{2}}{81}+\dfrac{1}{3}\left(\dfrac{7}{27}b_{1}^{2}+b_{1}b_{0}+b_{0}^{2}\right)+\dfrac{1}{3}\left(\dfrac{19}{27}c_{1}^{2}+\dfrac{5}{3}c_{1}c_{0}+c_{0}^{2}\right)=\dfrac{1}{2\lambda_{1}(2H+1)}
	\end{equation*}
	\begin{equation*}
	\Rightarrow\dfrac{a^{2}}{81}-2(c_{1}+c_{0})\left(a\dfrac{1}{18}+b_{1}\dfrac{1}{6}+b_{0}\dfrac{1}{3}\right)+\dfrac{1}{3}\left(\dfrac{7}{27}b_{1}^{2}+b_{1}b_{0}+b_{0}^{2}\right)
	\end{equation*}
	\begin{equation*}
	+c_{1}^{2}\left(1-\dfrac{5}{9}+\dfrac{19}{81}\right)+c_{0}^{2}\left(1-\dfrac{2}{3}+\dfrac{1}{3}\right)+c_{1}c_{0}\left(2-\dfrac{2}{3}-\dfrac{5}{9}+\dfrac{5}{9}\right)=\dfrac{1}{2\lambda_{1}(2H+1)}
	\end{equation*}
	\begin{equation*}
	\Rightarrow\dfrac{a^{2}}{81}-2(c_{1}+c_{0})\left(a\dfrac{1}{18}+b_{1}\dfrac{1}{6}+b_{0}\dfrac{1}{3}\right)+\dfrac{1}{3}\left(\dfrac{7}{27}b_{1}^{2}+b_{1}b_{0}+b_{0}^{2}\right)+\dfrac{55c_{1}^{2}}{81}+\dfrac{2c_{0}^{2}}{3}+\dfrac{4c_{1}c_{0}}{3}=\dfrac{1}{2\lambda_{1}(2H+1)}
	\end{equation*}
	Consider now equation $(\ref{7})$, we have
	\begin{equation*}
	\int_{0}^{1}\int_{0}^{u_{3}}\omega_{1,u_{2}}du_{2}d\omega_{1,u_{3}}=\dfrac{2H-1}{4\lambda_{1}(2H+1)}
	\end{equation*}
	By Fubini's theorem we have
	\begin{equation*}
	\int_{0}^{1}\int_{u_{2}}^{1}\omega_{1,u_{2}}d\omega_{1,u_{3}}du_{2}=\dfrac{2H-1}{4\lambda_{1}(2H+1)}\Rightarrow\int_{0}^{1}\omega_{1,u_{2}}(\omega_{1,1}-\omega_{1,u_{2}})du_{2}=\dfrac{2H-1}{4\lambda_{1}(2H+1)}
	\end{equation*}
	\begin{equation*}
	\Rightarrow (c_{1}+c_{0})\left[\int_{0}^{\frac{1}{3}}au_{1}du_{1}+\int_{\frac{1}{3}}^{\frac{2}{3}}b_{1}u_{1}+b_{0}du_{1}+\int_{\frac{2}{3}}^{1}c_{1}u_{1}+c_{0}du_{1}\right]
	\end{equation*}
	\begin{equation*}
	-\int_{0}^{\frac{1}{3}}a^{2}u_{1}^{2}du_{1}-\int_{\frac{1}{3}}^{\frac{2}{3}}(b_{1}u_{1}+b_{0})^{2}du_{1}-\int_{\frac{2}{3}}^{1}(c_{1}u_{1}+c_{0})^{2}du_{1}=\dfrac{2H-1}{4\lambda_{1}(2H+1)}
	\end{equation*}
	\begin{equation*}
	\Rightarrow (c_{1}+c_{0})\left[a\dfrac{1}{18}+b_{1}\dfrac{1}{6}+b_{0}\dfrac{1}{3}+c_{1}\dfrac{5}{18}+c_{0}\dfrac{1}{3}\right]
	\end{equation*}
	\begin{equation*}
	-\dfrac{a^{2}}{81}-\dfrac{1}{3}\left(\dfrac{7}{27}b_{1}^{2}+b_{1}b_{0}+b_{0}^{2}\right)-\dfrac{1}{3}\left(\dfrac{19}{27}c_{1}^{2}+\dfrac{5}{3}c_{1}c_{0}+c_{0}^{2}\right)=\dfrac{2H-1}{4\lambda_{1}(2H+1)}
	\end{equation*}
	\begin{equation*}
	\Rightarrow(c_{1}+c_{0})\left(a\dfrac{1}{18}+b_{1}\dfrac{1}{6}+b_{0}\dfrac{1}{3}\right)-\dfrac{a^{2}}{81}-\dfrac{1}{3}\left(\dfrac{7}{27}b_{1}^{2}+b_{1}b_{0}+b_{0}^{2}\right)
	\end{equation*}
	\begin{equation*}
	+c_{1}^{2}\left(\dfrac{5}{18}-\dfrac{19}{81}\right)+c_{0}^{2}\left(\dfrac{1}{3}-\dfrac{1}{3}\right)+c_{1}c_{0}\left(\dfrac{1}{3}+\dfrac{5}{18}-\dfrac{5}{9}\right)=\dfrac{2H-1}{4\lambda_{1}(2H+1)}
	\end{equation*}
	\begin{equation*}
	\Rightarrow(c_{1}+c_{0})\left(a\dfrac{1}{18}+b_{1}\dfrac{1}{6}+b_{0}\dfrac{1}{3}\right)-\dfrac{a^{2}}{81}-\dfrac{1}{3}\left(\dfrac{7}{27}b_{1}^{2}+b_{1}b_{0}+b_{0}^{2}\right)+\dfrac{7c_{1}^{2}}{162}+\dfrac{c_{1}c_{0}}{18}=\dfrac{2H-1}{4\lambda_{1}(2H+1)}
	\end{equation*}
	Consider now equation $(\ref{9})$, we have
	\begin{equation*}
	\lambda_{1}\omega_{1,1}^{4}+\lambda_{2}\omega_{2,1}^{4}=3\Rightarrow 2\lambda_{1}\omega_{1,1}^{4}=3\Rightarrow \lambda_{1}(c_{1}+c_{0})^{4}=\dfrac{3}{2}
	\end{equation*}
	Therefore, we have the following system of equations
	\begin{equation}\label{system}
	\begin{cases}
	2\lambda_{1}+\lambda_{3}=1 \qquad with \qquad \lambda_{1},\lambda_{1}\in[0,1],\\
	2\lambda_{1}(c_{1}+c_{0})^{2}=1,\\
	\dfrac{a^{2}}{81}+\dfrac{1}{3}\left(\dfrac{7}{27}b_{1}^{2}+b_{1}b_{0}+b_{0}^{2}\right)+\dfrac{1}{3}\left(\dfrac{19}{27}c_{1}^{2}+\dfrac{5}{3}c_{1}c_{0}+c_{0}^{2}\right)=\dfrac{1}{2\lambda_{1}(2H+1)},\\
	\dfrac{a^{2}}{81}-2(c_{1}+c_{0})\left(a\dfrac{1}{18}+b_{1}\dfrac{1}{6}+b_{0}\dfrac{1}{3}\right)+\dfrac{1}{3}\left(\dfrac{7}{27}b_{1}^{2}+b_{1}b_{0}+b_{0}^{2}\right)+\dfrac{55c_{1}^{2}}{81}+\dfrac{2c_{0}^{2}}{3}+\dfrac{4c_{1}c_{0}}{3}=\dfrac{1}{2\lambda_{1}(2H+1)},\\
	(c_{1}+c_{0})\left(a\dfrac{1}{18}+b_{1}\dfrac{1}{6}+b_{0}\dfrac{1}{3}\right)-\dfrac{a^{2}}{81}-\dfrac{1}{3}\left(\dfrac{7}{27}b_{1}^{2}+b_{1}b_{0}+b_{0}^{2}\right)+\dfrac{7c_{1}^{2}}{162}+\dfrac{c_{1}c_{0}}{18}=\dfrac{2H-1}{4\lambda_{1}(2H+1)},\\
	\lambda_{1}(c_{1}+c_{0})^{4}=\dfrac{3}{2},
	\end{cases}
	\end{equation}
	with the following unknowns $\lambda_{1},\lambda_{3},a,b_{1},c_{1}$.
	\\\\
	Consider the two equations in the system above
	\begin{equation*}
	\lambda_{1}(c_{1}+c_{0})^{4}=\dfrac{3}{2} \qquad and \qquad 2\lambda_{1}(c_{1}+c_{0})^{2}=1
	\end{equation*}
	we have
	\begin{equation*}
	\Rightarrow \lambda_{1}\dfrac{1}{4\lambda_{1}^{2}}=\dfrac{3}{2} \Rightarrow \lambda_{1}=\dfrac{1}{6}.
	\end{equation*}
	and
	\begin{equation}\label{knew}
	\Rightarrow c_{1}+c_{0}=\sqrt{3}
	\end{equation}
	Further, using $2\lambda_{1}+\lambda_{3}=1$ we have
	\begin{equation*}
	\Rightarrow \lambda_{3}=\dfrac{2}{3}
	\end{equation*}
	Now, consider the equations
	\begin{equation*}
	\dfrac{a^{2}}{81}-2(c_{1}+c_{0})\left(a\dfrac{1}{18}+b_{1}\dfrac{1}{6}+b_{0}\dfrac{1}{3}\right)+\dfrac{1}{3}\left(\dfrac{7}{27}b_{1}^{2}+b_{1}b_{0}+b_{0}^{2}\right)+\dfrac{55c_{1}^{2}}{81}+\dfrac{2c_{0}^{2}}{3}+\dfrac{4c_{1}c_{0}}{3}=\dfrac{1}{2\lambda_{1}(2H+1)}
	\end{equation*}
	and
	\begin{equation*}
	(c_{1}+c_{0})\left(a\dfrac{1}{18}+b_{1}\dfrac{1}{6}+b_{0}\dfrac{1}{3}\right)-\dfrac{a^{2}}{81}-\dfrac{1}{3}\left(\dfrac{7}{27}b_{1}^{2}+b_{1}b_{0}+b_{0}^{2}\right)+\dfrac{7c_{1}^{2}}{162}+\dfrac{c_{1}c_{0}}{18}=\dfrac{2H-1}{4\lambda_{1}(2H+1)}.
	\end{equation*}
	By summing them, we have
	\begin{equation*}
	-(c_{1}+c_{0})\left(a\dfrac{1}{18}+b_{1}\dfrac{1}{6}+b_{0}\dfrac{1}{3}\right)+\dfrac{7c_{1}^{2}}{162}+\dfrac{c_{1}c_{0}}{18}+\dfrac{55c_{1}^{2}}{81}+\dfrac{2c_{0}^{2}}{3}+\dfrac{4c_{1}c_{0}}{3}=\dfrac{3}{2}.
	\end{equation*}
	\begin{equation}\label{system1}
	\Rightarrow-(c_{1}+c_{0})\left(a\dfrac{1}{18}+b_{1}\dfrac{1}{6}+b_{0}\dfrac{1}{3}\right)+\dfrac{13c_{1}^{2}}{18}+\dfrac{25c_{1}c_{0}}{18}+\dfrac{2c_{0}^{2}}{3}=\dfrac{3}{2}.
	\end{equation}
	Further, by taking the difference of the two equations
	\begin{equation*}
	\dfrac{a^{2}}{81}+\dfrac{1}{3}\left(\dfrac{7}{27}b_{1}^{2}+b_{1}b_{0}+b_{0}^{2}\right)+\dfrac{1}{3}\left(\dfrac{19}{27}c_{1}^{2}+\dfrac{5}{3}c_{1}c_{0}+c_{0}^{2}\right)=\dfrac{1}{2\lambda_{1}(2H+1)}
	\end{equation*}
	and
	\begin{equation*}
	\dfrac{a^{2}}{81}-2(c_{1}+c_{0})\left(a\dfrac{1}{18}+b_{1}\dfrac{1}{6}+b_{0}\dfrac{1}{3}\right)+\dfrac{1}{3}\left(\dfrac{7}{27}b_{1}^{2}+b_{1}b_{0}+b_{0}^{2}\right)+\dfrac{55c_{1}^{2}}{81}+\dfrac{2c_{0}^{2}}{3}+\dfrac{4c_{1}c_{0}}{3}=\dfrac{1}{2\lambda_{1}(2H+1)}
	\end{equation*}
	we obtain
	\begin{equation*}
	2(c_{1}+c_{0})\left(a\dfrac{1}{18}+b_{1}\dfrac{1}{6}+b_{0}\dfrac{1}{3}\right)+\dfrac{1}{3}\left(\dfrac{19}{27}c_{1}^{2}+\dfrac{5}{3}c_{1}c_{0}+c_{0}^{2}\right)-\dfrac{55c_{1}^{2}}{81}-\dfrac{2c_{0}^{2}}{3}-\dfrac{4c_{1}c_{0}}{3}=0
	\end{equation*}
	\begin{equation}\label{system2}
	\Rightarrow2(c_{1}+c_{0})\left(a\dfrac{1}{18}+b_{1}\dfrac{1}{6}+b_{0}\dfrac{1}{3}\right)-\dfrac{4c_{1}^{2}}{9}-\dfrac{c_{0}^{2}}{3}-\dfrac{7c_{1}c_{0}}{9}=0
	\end{equation}
	Using equations $(\ref{system1})$ and $(\ref{system2})$ we have
	\begin{equation*}
	\dfrac{13c_{1}^{2}}{9}+\dfrac{25c_{1}c_{0}}{9}+\dfrac{4c_{0}^{2}}{3}-\dfrac{4c_{1}^{2}}{9}-\dfrac{c_{0}^{2}}{3}-\dfrac{7c_{1}c_{0}}{9}=3
	\end{equation*}
	\begin{equation*}
	\Rightarrow c_{1}^{2}+2c_{1}c_{0}+c_{0}^{2}=3 \Rightarrow c_{1}+c_{0}=\sqrt{3}.
	\end{equation*}
	We already knew this information from equation $(\ref{knew})$, but we did not use it here. This means that two of our constraints are in fact only one. Hence, we have one constraint less than expected, meaning that we are able to find a solution of our system since now the number of equations is the same as the number of unknowns.\\
	Recall that there are two underlying equations, which comes from the continuity condition of $\omega$, which are
	\begin{equation*} 
	\dfrac{a}{3}=\dfrac{b_{1}}{3}+b_{0} \qquad and \qquad \dfrac{2b_{1}}{3}+b_{0}=\dfrac{2c_{1}}{3}+c_{0}
	\end{equation*}
	Let us continue to focus on the equation $(\ref{system1})$. Substitute $a=b_{1}+3b_{0}$, $b_{0}=\dfrac{2c_{1}}{3}+c_{0}-\dfrac{2b_{1}}{3}$ and $c_{0}=\sqrt{3}-c_{1}$. In other words, take
	\begin{equation*}
	a=b_{1}+2c_{1}+3c_{0}-2b_{1}=2c_{1}+3c_{0}-b_{1}=2c_{1}+3\sqrt{3}-3c_{1}-b_{1}=3\sqrt{3}-c_{1}-b_{1}
	\end{equation*}
	and
	\begin{equation*}
	b_{0}=\dfrac{2c_{1}}{3}+\sqrt{3}-c_{1}-\dfrac{2b_{1}}{3}=\sqrt{3}-\dfrac{c_{1}}{3}-\dfrac{2b_{1}}{3}
	\end{equation*}
	Hence, we have that equation $(\ref{system1})$ becomes
	\begin{equation*}
	-\sqrt{3}\left((3\sqrt{3}-c_{1}-b_{1})\dfrac{1}{18}+b_{1}\dfrac{1}{6}+(\sqrt{3}-\dfrac{c_{1}}{3}-\dfrac{2b_{1}}{3})\dfrac{1}{3}\right)+\dfrac{13c_{1}^{2}}{18}+\dfrac{25c_{1}(\sqrt{3}-c_{1})}{18}+\dfrac{2(\sqrt{3}-c_{1})^{2}}{3}=\dfrac{3}{2}
	\end{equation*}
	\begin{equation*}
	\Rightarrow-\sqrt{3}\left(\dfrac{\sqrt{3}}{2}-\dfrac{c_{1}}{6}-\dfrac{b_{1}}{9}\right)+2+\dfrac{c_{1}\sqrt{3}}{18}=\dfrac{3}{2}
	\end{equation*} 
	\begin{equation*}
	\Rightarrow\dfrac{b_{1}\sqrt{3}}{9}+\dfrac{2c_{1}\sqrt{3}}{9}=1\Rightarrow b_{1}+2c_{1}=3\sqrt{3}.
	\end{equation*} 
	\begin{equation*}
	\Rightarrow b_{1}=3\sqrt{3}-2c_{1}.
	\end{equation*} 
	We can now consider the third equation of our system (\textit{i.e.} equation $(\ref{third})$) and rewrite it in terms of $c_{1}$. First, let us rewrite all the unknowns that we need in terms of $c_{1}$:
	\begin{equation*}
	a=3\sqrt{3}-c_{1}-b_{1}=3\sqrt{3}-c_{1}-3\sqrt{3}+2c_{1}=c_{1}
	\end{equation*}
	and
	\begin{equation*}
	b_{0}=\sqrt{3}-\dfrac{c_{1}}{3}-\dfrac{2b_{1}}{3}=\sqrt{3}-\dfrac{c_{1}}{3}-2\sqrt{3}+\dfrac{4c_{1}}{3}=c_{1}-\sqrt{3}
	\end{equation*}
	Notice that $a=c_{1}$ and $b_{0}=-c_{0}$, which is in accordance with the results of Lyons and Victoir.\\
	We can now proceed with the substitution
	\begin{equation*}
	\dfrac{c_{1}^{2}}{81}+\dfrac{1}{3}\left(\dfrac{7}{27}(3\sqrt{3}-2c_{1})^{2}+(3\sqrt{3}-2c_{1})(c_{1}-\sqrt{3})+(c_{1}-\sqrt{3})^{2}\right)
	\end{equation*}
	\begin{equation*}
	+\dfrac{1}{3}\left(\dfrac{19}{27}c_{1}^{2}+\dfrac{5}{3}c_{1}(\sqrt{3}-c_{1})+(\sqrt{3}-c_{1})^{2}\right)=\dfrac{1}{2\lambda_{1}(2H+1)}
	\end{equation*}
	\begin{equation*}
	\Rightarrow \dfrac{c_{1}^{2}}{81}+\dfrac{1}{3}\left(7+\dfrac{28c_{1}^{2}}{27}-\dfrac{28c_{1}\sqrt{3}}{9}+3c_{1}\sqrt{3}-9-2c_{1}^{2}+2c_{1}\sqrt{3}+c_{1}^{2}+3-2c_{1}\sqrt{3}\right)
	\end{equation*}
	\begin{equation*}
	+\dfrac{1}{3}\left(\dfrac{19c_{1}^{2}}{27}+\dfrac{5c_{1}\sqrt{3}}{3}-\dfrac{5c_{1}^{2}}{3}+3+c_{1}^{2}-2c_{1}\sqrt{3}\right)=\dfrac{1}{2\lambda_{1}(2H+1)}
	\end{equation*}
	\begin{equation*}
	c_{1}^{2}\dfrac{1}{27}-c_{1}\dfrac{4\sqrt{3}}{27}+\dfrac{4}{3}=\dfrac{3}{2H+1} \Rightarrow c_{1}^{2}\dfrac{1}{27}-c_{1}\dfrac{4\sqrt{3}}{27}=\dfrac{5-8H}{3(2H+1)}
	\end{equation*}
	\begin{equation*}
	\Rightarrow c_{1}=\dfrac{4H\sqrt{3}+2\sqrt{3}-\sqrt{-96H^{2}+66H+57}}{2H+1}
	\end{equation*}
	If $H=\frac{1}{2}$ then
	\begin{equation*}
	c_{1}=\sqrt{3}\left(2-\sqrt{\dfrac{11}{2}}\right)
	\end{equation*}
	which is in accordance with the results of Lyons and Victoir. 
	\\
	Therefore, we know all the unknowns. In particular, the path $\omega_{1,t}$ is given by:
	\begin{equation*}
	\begin{cases}
	\dfrac{4H\sqrt{3}+2\sqrt{3}-\sqrt{-96H^{2}+66H+57}}{2H+1}t, \qquad t\in[0,\frac{1}{3}],\\
	\dfrac{2H\sqrt{3}+\sqrt{3}-\sqrt{-96H^{2}+66H+57}}{2H+1}+\dfrac{2\sqrt{-96H^{2}+66H+57}-2H\sqrt{3}-\sqrt{3}}{2H+1}t, \qquad t\in[\frac{1}{3},\frac{2}{3}],\\
	\dfrac{\sqrt{-96H^{2}+66H+57}-2H\sqrt{3}-\sqrt{3}}{2H+1}+\dfrac{4H\sqrt{3}+2\sqrt{3}-\sqrt{-96H^{2}+66H+57}}{2H+1}t, \qquad t\in[\frac{2}{3},1],\\
	\end{cases}
	\end{equation*}
	If you let $\alpha$ and $\beta$ to be:
	\begin{equation*}
	\alpha:=\dfrac{2H\sqrt{3}+\sqrt{3}}{2H+1} \qquad and \qquad \beta:=\dfrac{\sqrt{-96H^{2}+66H+57}}{2H+1}
	\end{equation*}
	then we can rewrite our path $\omega_{1,t}$ can be written in the following form:
	\begin{equation*}
	\begin{cases}
	(2\alpha-\beta) t, \qquad t\in[0,\frac{1}{3}],\\
	(\alpha-\beta)+(2\beta-\alpha)t, \qquad t\in[\frac{1}{3},\frac{2}{3}],\\
	(\beta-\alpha)+(2\alpha-\beta)t, \qquad t\in[\frac{2}{3},1],\\
	\end{cases}
	\end{equation*} 
	Now, we proceed with a check of our result. Indeed we substitute the values obtained for our unknowns in our system of equations to check the correctness of these values.\\
	In particular, let us focus first on the equation
	\begin{equation*}
	\dfrac{a^{2}}{81}-2(c_{1}+c_{0})\left(a\dfrac{1}{18}+b_{1}\dfrac{1}{6}+b_{0}\dfrac{1}{3}\right)+\dfrac{1}{3}\left(\dfrac{7}{27}b_{1}^{2}+b_{1}b_{0}+b_{0}^{2}\right)+\dfrac{55c_{1}^{2}}{81}+\dfrac{2c_{0}^{2}}{3}+\dfrac{4c_{1}c_{0}}{3}=\dfrac{1}{2\lambda_{1}(2H+1)}
	\end{equation*}
	Let us rewrite it in terms of $c_{1}$:
	\begin{equation*}
	\dfrac{c_{1}^{2}}{81}-2\sqrt{3}\left(c_{1}\dfrac{1}{18}+(3\sqrt{3}-2c_{1})\dfrac{1}{6}+(c_{1}-\sqrt{3})\dfrac{1}{3}\right)+
	\end{equation*}
	\begin{equation*}
	+\dfrac{1}{3}\left(\dfrac{7}{27}(3\sqrt{3}-2c_{1})^{2}+(3\sqrt{3}-2c_{1})(c_{1}-\sqrt{3})+(c_{1}-\sqrt{3})^{2}\right)
	\end{equation*}
	\begin{equation*}
	+\dfrac{55c_{1}^{2}}{81}+\dfrac{2(\sqrt{3}-c_{1})^{2}}{3}+\dfrac{4c_{1}(\sqrt{3}-c_{1})}{3}=\dfrac{3}{2H+1}
	\end{equation*}
	\begin{equation*}
	\Rightarrow c_{1}^{2}\dfrac{1}{27}-c_{1}\dfrac{4\sqrt{3}}{27}+\dfrac{4}{3}=\dfrac{3}{2H+1} 
	\end{equation*}
	as before. Now, let us focus on the equation
	\begin{equation*}
	(c_{1}+c_{0})\left(a\dfrac{1}{18}+b_{1}\dfrac{1}{6}+b_{0}\dfrac{1}{3}\right)-\dfrac{a^{2}}{81}-\dfrac{1}{3}\left(\dfrac{7}{27}b_{1}^{2}+b_{1}b_{0}+b_{0}^{2}\right)+\dfrac{7c_{1}^{2}}{162}+\dfrac{c_{1}c_{0}}{18}=\dfrac{2H-1}{4\lambda_{1}(2H+1)}
	\end{equation*}
	Following the same procedure we have
	\begin{equation*}
	\sqrt{3}\left(c_{1}\dfrac{1}{18}+(3\sqrt{3}-2c_{1})\dfrac{1}{6}+(c_{1}-\sqrt{3})\dfrac{1}{3}\right)-\dfrac{c_{1}^{2}}{81}
	\end{equation*}
	\begin{equation*}
	-\dfrac{1}{3}\left(\dfrac{7}{27}(3\sqrt{3}-2c_{1})^{2}+(3\sqrt{3}-2c_{1})(c_{1}-\sqrt{3})+(c_{1}-\sqrt{3})^{2}\right)
	\end{equation*}
	\begin{equation*}
	+\dfrac{7c_{1}^{2}}{162}+\dfrac{c_{1}(\sqrt{3}-c_{1})}{18}=\dfrac{3(2H-1)}{2(2H+1)}
	\end{equation*}
	\begin{equation*}
	\Rightarrow -c_{1}^{2}\dfrac{1}{27}+c_{1}\dfrac{4\sqrt{3}}{27}+\dfrac{1}{6}=\dfrac{3(2H-1)}{2(2H+1)} \Rightarrow c_{1}^{2}\dfrac{1}{27}-c_{1}\dfrac{4\sqrt{3}}{27}=\dfrac{5-8H}{3(2H+1)}.
	\end{equation*}
	as before. The other equations are straightforward to check. Therefore, our solution is consistent and when $H=\frac{1}{2}$ our solution is the same solution obtained by Lyons and Victoir in their paper.
\end{proof}

\end{document}